\documentclass[12pt]{article}
\usepackage[utf8x]{inputenc}
\usepackage{amsmath}
\usepackage{amsfonts}
\usepackage{latexsym}
\usepackage{amsthm}
\usepackage{xargs}
\usepackage{graphicx}
\usepackage{stmaryrd}
\usepackage{color}
\usepackage{caption}
\usepackage{subcaption}
\usepackage{euscript}
\usepackage{url}
\usepackage{emptypage}

\newtheorem{theorem}{Theorem}[section]

\newtheorem{corollary}[theorem]{Corollary}

\newtheorem{lemma}[theorem]{Lemma}

\newtheorem{proposition}[theorem]{Proposition}
\newtheorem{remark}[theorem]{Remark}

\newcommand{\dist}{\kappa}
\newcommand{\coef}{\textup{d}}
\newcommand{\param}{\sigma}
\newcommand{\ost}{\omega}
\newcommand{\distin}{\bar{\dist}}

\newcommand{\re}{\mathrm{Re}\,}
\newcommand{\im}{\mathrm{Im}\,}
\newcommand{\uns}{\mathrm{u}}
\newcommand{\sta}{\mathrm{s}}
\newcommand{\us}{*}
\newcommand{\inn}{\mathrm{in}}

\newcommand{\vu}{u}

\newcommand{\vw}{w}
\newcommand{\vs}{s}

\newcommand{\xb}{\bar{x}}
\newcommand{\yb}{\bar{y}}
\newcommand{\zb}{\bar{z}}
\newcommand{\tb}{\bar{t}}

\newcommand{\hetz}{Z_0}
\newcommand{\hetr}{R_0}
\newcommand{\hetth}{\Theta_0}
\newcommand{\Fb}{\mathbf{F}}
\newcommand{\Gb}{\mathbf{G}}
\newcommand{\Hb}{\mathbf{H}}
\newcommand{\vt}{\vartheta}
\newcommand{\Diffw}{\mathcal{E}}
\newcommand{\kin}{k^\inn}
\newcommand{\kintilde}{{\tilde k}^\inn}
\newcommand{\Pinsmall}{P_1^\inn}
\newcommand{\Pin}{P^\inn}
\newcommand{\psiin}{\psi_\inn}
\newcommand{\xiin}{\xi_\inn}

\newcommand{\Fout}{\mathcal{F}^{\rm out}}
\newcommand{\Lout}{\mathcal{L}^{\rm out}}
\newcommand{\Linner}{\mathcal{L}}
\newcommand{\Minner}{\mathcal{M}}
\newcommand{\Mmatch}{\mathcal{M}_1}
\newcommand{\Ginner}{\mathcal{G}}
\newcommand{\Gdiffin}{\hat{\mathcal{G}}}
\newcommand{\Ldiffin}{\hat{\mathcal{L}}}
\newcommand{\Gmatch}{\mathcal{G}_0}
\newcommand{\oprhsphi}{\mathcal{A}}
\newcommand{\oprhsP}{\mathcal{B}}

\newcommandx{\DoutT}[4][1=\dist, 2=T,3=\beta]{D_{#1,#3,#2}^{#4}}
\newcommand{\Doutinter}{D_{\dist,\beta}}
\newcommand{\Doutintertilde}{\tilde{D}_{\dist,\beta}}
\newcommandx{\Dmchout}[5][1=\dist, 2={,}, 3=\beta_1, 4=\beta_2]{D_{#1,#3,#4}^{\mathrm{mch}#2#5}}
\newcommandx{\Dmchin}[2][1={,}]{\mathcal{D}_{\dist,\beta_1,\beta_2}^{\mathrm{mch}#1#2}}
\newcommandx{\Tout}[1][1=\ost]{\mathbb{T}_{#1}}
\newcommand{\Din}[1]{\mathcal{D}_{\beta_0,\distin}^{\mathrm{in},#1}}
\newcommand{\Dinu}[1]{D_{\beta_0,\distin}^{\mathrm{in},#1}}
\newcommand{\DoutTinvars}[1]{\mathcal{D}_{\dist,\beta,T}^{#1}}
\newcommand{\Ein}{E_{\beta_0,\distin}}

\newcommandx{\Bsin}[3][3=\ost]{\mathcal{X}^{#1}_{#2,#3}}
\newcommandx{\Bsinfloor}[3][3=\ost]{\tilde{\mathcal{X}}^{#1}_{#2,#3}}
\newcommandx{\Bsdiffin}[2][1=\ost]{\mathcal{X}_{#1,#2}}
\newcommandx{\Bsdiffinfloor}[2][2=\ost]{\tilde{\mathcal{X}}_{#1,#2}}

\newcommand{\thmoutloce}{Theorem~2.4 in~\cite{BCS16a}}
\newcommand{\secDifference}{Section~2.4 in~\cite{BCS16a} }

\newcommand{\thmdifpartsolinjectivee}{Theorem~2.6 in~\cite{BCS16a}}

\newcommand{\thmktildeave}{Theorem~2.8 in~\cite{BCS16a}}

\newcommand{\lemUpsilone}{Lemma~2.10 in~\cite{BCS16a}}
\newcommand{\rmkL}{Remark~5.7 in~\cite{BCS16a} }
\newcommand{\mainthmregular}{Theorem~2.13 in~\cite{BCS16a}}

\title{Breakdown of a 2D heteroclinic connection in the Hopf-zero singularity (II). The generic case}
\author{I. Baldom\'a, O. Castej\'on, T. M. Seara}

\begin{document}
\maketitle

\section*{Summary}
In this paper we prove the breakdown of the two-dimensional stable and unstable manifolds 
associated to two saddle-focus points which appear in the unfoldings of the Hopf-zero singulariry.
The method consists in obtaining an asymptotic formula for the difference between this manifolds which turns to be exponentially small
respect to the unfolding parameter. The formula obtained is explicit but depends on the so-called Stokes constants, which arise in the study
of original vector field and which corresponds to the so called inner equation in singular perturbation theory. 

{\sl Keywords:} Exponentially small splitting, Hopf-zero bifurcation, inner equation, Stokes constant.

\section{Introduction and main result}\label{IntroHopfZero2D-inner}
The present work is the natural continuation of~\cite{BCS16a}.
We briefly explain the setting we deal with below, for a more complete introduction we refer to the reader to the mentioned work~\cite{BCS16a}.

The Hopf-zero singularity is a vector field $X^*$ in $\mathbb{R}^3$ having the origin as a fixed point with linear part having the zero eigenvalue and
a pair of purely imaginary eigenvalues. When we consider versal analytic unfoldings with two parameters (one parameter if we consider
conservative unfoldings), $X_{\mu,\nu}$ such that $X_{0,0}= X^*$, we generically encounter a
\textit{beyond all orders phenomenon}: for values of $(\mu,\nu)$ belonging to some open set, performing the normal form procedure up to \textit{any} order,
we find the same qualitative behavior.
Concretely, if $(\mu,\nu)$ belongs to an appropriate curve in the parameter space, the normal form vector field of order $n$, $X_{\mu,\nu}^n$, possesses two
critical points of saddle-focus type having two coincident heteroclinic connections (a curve and a surface corresponding to the invariant
manifolds of the critical points).

This qualitative behavior changes when we consider the whole unfolding $X_{\mu,\nu}$. Normal form theory assures that
$X_{\mu,\nu}=X_{\mu,\nu}^n+F_{\mu,\nu}^n$ with $X_{\mu,\nu}^n$ the normal form vector field and $F_{\mu,\nu}^n$ the remainder.
Standard perturbation theory assures that $X_{\mu,\nu}$ also has two critical points of saddle-focus type with invariant manifolds associated to them.
However, generically, the heteroclinic connections do not persist anymore. The question then is how to measure
the distance between them. Since this is a beyond of orders phenomenon, in the analytic setting, these  distances turn out to be exponentially small,
see~\cite{BCS13} and~\cite{BCS16a} for more details.

In~\cite{BCS13}, the distance between the one dimensional manifolds was computed. In~\cite{BCS16a}, the distance between the two dimensional
manifolds was computed for a class of non generic analytic unfoldings.
In fact, it was proven that this distance is dominated for a suitable version of the Melnikov function.
This was done by adding an artificial parameter to make the perturbation terms smaller.
Let us to explain this with more detail in order to compare both (generic and non generic) scenarios.
Consider the normal form up to order $2$ and decompose $X_{\mu,\nu}^2 = \tilde{X}_{\mu,\nu}^2+ P_{\mu,\nu}^2$ with $P_{\mu,\nu}^2$
depending only on the parameters $\mu,\nu$.
We introduce the parameter $q$ as:
\begin{equation}\label{parameterp}
X_{\mu,\nu} = \tilde{X}_{\mu,\nu}^2 + (\sqrt\mu)^{q} (P_{\mu,\nu}^2+ F_{\mu,\nu}^2).
\end{equation}
The artificial parameter $q$ determines whether we are in the non generic ($q>0$) or in the generic ($q=0$) setting.
We call \textit{regular case} to the case  $q>0$. It is worthy of mention that the results
given in~\cite{BCS16a} still are valid for the case $q=0$ which turns out to be called \textit{singular case}.

In the present work, we prove the asymptotic formula for the distance between the two dimensional invariant manifolds
of generic analytic unfoldings of the Hopf-zero singularity, which is of order $\mathcal{O}\left (\mu^{-\bar a} e^{-\frac{a }{\sqrt{\mu}}}\right )$
for some constants $a, \bar a>0$. 
To do so, we use previous results in~\cite{BCS16a} and we
introduce the so-called \emph{inner equation}, which is an equation  independent of parameters which corresponds to the original vector field $X^*$.
It turns out that the difference between two suitable solutions of this equation, approximates the distance
of our invariant manifolds.
Previous works proving exponentially small phenomena have to deal with inner equations.
In~\cite{Gel97-2}, the corresponding inner equations were studied for several periodically perturbed second order equations.
In~\cite{GS01} and \cite{MSS11b} there is a rigorous study of the inner equation of the H\'enon map and the Mcmillan map respectively
using Resurgence Theory~\cite{Ecalle81a, Ecalle81b}.
In~\cite{OSS03} there is a rigorous analysis of the inner equation for the Hamilton-Jacobi equation associated to a pendulum equation
with a certain perturbation term, also using Resurgence Theory.
Besides, there are other works where functional analysis techniques are used to deal with more general cases.
\cite{BaldomaInner}  is the only result which deals with the
inner equation associated to very general type of polynomial Hamiltonian systems with a fast perturbation.
In~\cite{BM12}, the inner equation for generalized standard maps is studied.
In~\cite{BaSe08} the authors study
the inner equation associated to the splitting of the one-dimensional heteroclinic connection of the Hopf-zero singularity in the conservative
case.

Let us now to enunciate properly the main result in this work. As in~\cite{BCS16a}, we perform the normal form procedure
up to order three and we see that, rescaling of variables
and renaming of parameters, $X_{\mu,\nu}$ can be written as
\begin{align}\label{sistema-NForder3}
 \frac{d\xb}{d\tb}&=\xb\left(\nu-\beta_1\zb\right)+\yb\left(\alpha_0+\alpha_1\nu+\alpha_2\mu+\alpha_3\zb\right)+\bar{f}(\xb,\yb,\zb,\mu,\nu),\medskip\nonumber\\
\frac{d\yb}{d\tb}&=-\xb\left(\alpha_0+\alpha_1\nu+\alpha_2\mu+\alpha_3\zb\right)+\yb\left(\nu-\beta_1\zb\right)+\bar{g}(\xb,\yb,\zb,\mu,\nu),\medskip\\
\frac{d\zb}{dt}&=-\mu+\zb^2+\gamma_2(\xb^2+\yb^2)+ \bar{h}(\xb,\yb,\zb,\mu,\nu)\nonumber
\end{align}
with $\bar{f}(\xb,\yb,\zb,\mu,\nu), \bar{g}(\xb,\yb,\zb,\mu,\nu) = \mathcal{O}_3(\xb,\yb,\zb,\mu,\nu)$ and
$$
\bar{h}(\xb,\yb,\zb,\mu,\nu)=\gamma_3\mu^2+\gamma_4\nu^2+\gamma_5\mu\nu+ \mathcal{O}_3(\xb,\yb,\zb,\mu,\nu).
$$

The case in~\cite{BCS16a} consists in considering, for $q\geq 0$,
$(\sqrt \mu) ^q(\bar{f},\bar{g},\bar{h})$ instead of $(\bar{f},\bar{g},\bar{h})$ as perturbation terms.

The main result in this work is:
\begin{theorem}\label{mainthm-inner-intro}
The system~\eqref{sistema-NForder3}, with $\mu,\,\beta_1,\gamma_2>0$ and $|\nu|<\beta_1\sqrt{\mu}$
has two critical points $\bar S_\pm(\mu,\nu)$ of saddle-focus type.

Fix $\vu \in \mathbb{R}$ and let $\bar D^{\uns,\sta}(\vu,\theta,\mu,\nu)$ ($\bar D^{\uns,\sta}(\vu,\theta,\mu)$ in the conservative case)
be the distance between the two dimensional unstable manifold of $\bar S_-(\mu,\nu)$ and the two dimensional stable manifold of
$\bar S_+(\mu,\nu)$ when they meet the plane $\zb=\sqrt{\mu} \tanh(\beta_1\vu)$.

We define the function
$$
\bar\vt(\vu,\mu)=\frac{\alpha_0\vu}{\sqrt{\mu}}+\frac{1}{\beta_1}\left(\alpha_3+\alpha_0L_0\right)\left[\log\cosh(\beta_1\vu)-\frac{1}{2}\log\mu\right]+\alpha_0L(\vu),
$$
where the constant $L_0$ and the function $L(u)$ can be computed explicitly and only depend on the terms of order three of $\bar{f},\bar{g}$ and $\bar{h}$
(see~\rmkL for an explicit  formula of them).

Then, there exist constants $\mathcal{C}_1^*$, $\mathcal{C}_2^*$ in such a way that, given $T_0>0$, for all $\vu\in[-T_0,T_0]$ and $\theta\in\mathbb{S}^1$,
the following holds:
\begin{enumerate}
 \item 
 In the conservative case, which corresponds to $\beta_1=1$ and $\nu=0$, as $\mu\to 0^+$,
\begin{align*}
 \bar D^{\uns,\sta}(\vu,\theta,\mu)=
&\sqrt{\frac{\gamma_2}{2}}\frac{e^{-\frac{\alpha_0\pi}{2\sqrt\mu}}}{(\sqrt\mu)^3}\cosh^{3}(\vu)\Bigg[\mathcal{C}_1^*\cos\Big(\theta+\bar\vt(\vu,\mu)\Big)\\
&+\mathcal{C}_2^*\sin\Big(\theta+\bar\vt(\vu,\mu)\Big)+\mathcal{O}\left(\frac{1}{|\log\mu|}\right)\Bigg].
\end{align*}

\item 
In the dissipative case, there exists a function $\nu=\nu_0(\mu)=\mathcal{O}(\mu)$, such that,
as $\mu\to 0^+$,\begin{align*}
 \bar D^{\uns,\sta}(\vu,\theta,\,\mu,\nu_0(\mu))=&\sqrt{\frac{\gamma_2}{\beta_1+1}}\cosh^{1+\frac{2}{\beta_1}}(\beta_1\vu)
\frac{e^{-\frac{\alpha_0\pi}{2\beta_1\sqrt\mu}}}{(\sqrt\mu)^{1+\frac{2}{\beta_1}}}\Bigg[\mathcal{C}_1^*\cos\Big(\theta+\bar\vt(\vu,\mu)\Big)\\
&+\mathcal{C}_2^*\sin\Big(\theta+\bar\vt(\vu,\mu)\Big)+\mathcal{O}\left(\frac{1}{|\log(\mu)|}\right)\Bigg].
\end{align*}
\end{enumerate}\end{theorem}

\begin{remark}
The constants $\mathcal{C}_i^*$, which are usually called Stokes constants (see \cite{St64,St02}),
depend on the full jet of the original vector field $X^*$ and therefore, up to now, they can only be computed numerically.
This computation is not trivial, and is not the goal of the present paper.
For the one-dimensional case, it has been done for particular examples in \cite{larrealseara}.
A detailed and accurate numerical computation of the distance in the one- and two-dimensional cases in many examples (in conservative and non-conservative settings)
has been done in \cite{diks}.
\end{remark}
\begin{remark}
In the dissipative case, a more general result is indeed proven: for given $a_1,a_2\in\mathbb{R}$ and $a_3>0$, there exists a function $\nu=\nu(\mu)$
(depending on $a_1,a_2$ and $a_3$) satisfying $\nu(\mu)-\nu_0(\mu) = \mathcal{O}\big (\mu^{a_2}
e^{-\frac{a_3\pi}{2\beta_1\sqrt\mu}}\big)$, such that,
as $\mu\to 0^+$,
$$
 \bar D^{\uns,\sta}(\vu,\theta,\,\mu,\nu(\mu))= \bar D^{\uns,\sta}(\vu,\theta,\mu,\nu_0(\mu)) + a_1\cosh^{1+\frac{2}{\beta_1}}(\beta_1\vu)\mu^{a_2}
e^{-\frac{a_3\pi}{2\beta_1\sqrt\mu}}\left(1+\mathcal{O}( \sqrt \mu )\right).
$$
The result in Theorem~\ref{mainthm-inner-intro} corresponds to $a_1=0$.
\end{remark}

The main ideas of the proof of Theorem~\ref{mainthm-inner-intro} are given in
Section~\ref{secSetupHeuristic2D-inner} below. Sections~\ref{sec:asymptotic}--\ref{sec:matching} are devoted to present
the technical proofs of the results in the mentioned Section~\ref{secSetupHeuristic2D-inner}.

\section{Set-up and heuristics of the proof of Theorem~\ref{mainthm-inner-intro}}\label{secSetupHeuristic2D-inner}
The main goal in this section is to present the strategy to prove Theorem~\ref{mainthm-inner-intro}.
As we will see its demonstration is involved and requires deep tools in functional
analysis as well as complex matching techniques which will be explained in detail.

We begin in Section~\ref{preliminary} by presenting the adequate setting to deal with and the precise statement of some of the results
proven in~\cite{BCS16a} that, as we said in Section~\ref{IntroHopfZero2D-inner},
still hold true in the current setting.
We also roughly present the strategy we will follow to prove Theorem~\ref{mainthm-inner-intro}.
Later, in Section~\ref{subsec-introinner}, we derive and study the \textit{inner equation} and we present the way to use this equation to prove
Theorem \ref{mainthm-inner-intro}.
This strategy is developed with more details in Sections~\ref{subsec:intromatching} and~\ref{subsec:introformuladiff}.

As a general rule in this work, we will omit the dependence of the functions with respect to its variables and parameters whenever this dependence is clear.
Moreover, we will denote by $K$ \textit{any} constant independent of the parameters which can change its value along the paper.

\subsection{Preliminary considerations and previous results}\label{preliminary}
This section is divided in three subsections which summarize the setting and results introduced in~\cite{BCS16a}.
In Section~\ref{subsec:scaling} we briefly explain the appropriate scalings and changes of variables we perform. After this,
in Section~\ref{subsec:unperturb}, we
study what we consider the unperturbed system and finally, in Section~\ref{subsec:previousresults}, we present the
results proven in~\cite{BCS16a} we will use along this work.
They are related to the existence and properties of global parameterizations of the invariant unstable and stable manifolds of the critical points in the, 
usually called, \textit{outer} domains (see \eqref{defDoutT}).

\subsubsection{Scalings and symplectic polar variables}\label{subsec:scaling}
We scale system \eqref{sistema-NForder3} as in~\cite{BCS16a}, see also~\cite{BCS13}. Indeed,
we define the new parameters $\delta=\sqrt{\mu}$, $\param=\delta^{-1}\nu$ and
we rename the coefficients $b=\gamma_2$, $c=\alpha_3$ and $\coef=\beta_1$.
We also introduce the constant $h_3$ of $\bar{h}$ given by
$$\bar{h}(0,0,\zb,0,0)= h_3 \zb^3 + \mathcal{O}(\zb^4).$$
With the new variables
$x=\delta^{-1}\xb$, $y=\delta^{-1}\yb$, $z=\delta^{-1}\zb+\delta h_3 /2$ and $t=\delta\tb$
system~\eqref{sistema-NForder3} becomes:
\begin{equation}\label{initsys-inner}
 \begin{aligned}
  \frac{dx}{dt}&= x\left(\param-\coef z\right)+\left(\frac{\alpha(\delta^2,\delta\param)}{\delta}+cz\right)y+
	\delta^{-2}f (\delta x,\delta y, \delta z, \delta,\delta\param),\\
  \frac{dy}{dt}&=-\left(\frac{\alpha(\delta^2,\delta\param)}{\delta}+cz\right)x+y\left(\param-\coef z\right)+
	\delta^{-2}g (\delta x,\delta y, \delta z, \delta,\delta\param),\\
 \frac{dz}{dt}&=-1+b(x^2+y^2)+z^2+\delta^{-2}h (\delta x, \delta y, \delta z, \delta,\delta\param),
 \end{aligned}
\end{equation}
where
$\alpha(\delta^2,\delta\param)=\alpha_0+\alpha_1\delta\param+\alpha_2\delta^2$ with $\alpha_0\neq 0$ and $f,g$ and $h$ are
the corresponding ones to $\bar{f},\bar{g}$ and $\bar{h}$.
To shorten the notation we write system~\eqref{initsys-inner} as
\begin{equation}\label{initsys-inner2DX}
\frac{d \zeta }{dt} = X(\zeta,\delta,\param)= X_{0}(\zeta,\delta,\param) + \delta^{-2} X_1(\delta \zeta, \delta, \delta \sigma),\qquad \zeta=(x,y,z).
\end{equation}
Note that $X_1(\delta \zeta,\delta,\delta \sigma)= \mathcal{O}_3(\delta \zeta,\delta, \delta \sigma)$.
As it is seen in \cite{BCS13}, the vector field $X$ defined in~\eqref{initsys-inner2DX} has two critical points
$S_\pm(\delta,\param)$ of saddle-focus type of the form:
$$
S_{\pm}(\delta,\param) = (0,0,\pm 1) + \big (\mathcal{O}(\delta^{3}), \mathcal{O}(\delta^{3}), \mathcal{O}(\delta^{2})\big ).
$$

We consider system~\eqref{initsys-inner2DX} in symplectic cylindric coordinates
\begin{equation}\label{cylindriccoordsunpertub}
x=\sqrt{2r}\cos\theta,\qquad y=\sqrt{2r}\sin\theta,\qquad z=z,
\end{equation}
and we obtain
\begin{equation}\label{syspolar}
 \begin{aligned}
   \frac{dr}{dt}&=2r(\param-\coef z)+\delta^{-2} \Fb (\delta r,\theta,\delta z,\delta,\delta\param),\\
   \frac{d\theta}{dt}&=-\frac{\alpha}{\delta}-cz+\delta^{-2}\Gb (\delta r,\theta,\delta z,\delta,\delta\param),\\
   \frac{dz}{dt}&=-1+2br+z^2+\delta^{-2} \Hb (\delta r,\theta,\delta z,\delta,\delta\param),
 \end{aligned}
\end{equation}
where $\mathbf{X}_1=( \Fb, \Gb, \Hb)$ is defined by
\begin{equation}\label{defFGHbold}
\mathbf{X}_1 (\delta r, \theta, \delta z,\delta,\delta \param) = \left (
\begin{array}{ccc} \sqrt{2r} \cos \theta & \sqrt{2r} \sin \theta & 0 \\
-\frac{1}{\sqrt{2r}} \sin \theta & \frac{1}{\sqrt{2r}} \cos \theta & 0 \\
0 & 0 & 1 \end{array}\right ) X_1(\delta \zeta,\delta, \delta \sigma)
\end{equation}
with $\delta \zeta = (\delta \sqrt{2r} \cos \theta, \delta \sqrt{2r}\sin \theta, \delta z)$ and $X_1$ is
introduced in~\eqref{initsys-inner2DX}.

\subsubsection{The unperturbed system}\label{subsec:unperturb} 
From now we will call the unperturbed system  to system~\eqref{syspolar} with $\Fb=\Gb=\Hb=0$ and $\sigma=0$
(or equivalently system~\eqref{initsys-inner} with $f=g=h=0$ and $\sigma=0$). In symplectic cylindric coordinates~\eqref{cylindriccoordsunpertub},
the unperturbed system is:
$$
\frac{dr}{dt}=-2\coef rz,\qquad
\frac{d\theta}{dt}=-\frac{\alpha}{\delta}-cz, \qquad
\frac{dz}{dt}=-1+2br+z^2.
$$
As $b>0$, it has a 2-dimensional heteroclinic manifold, $\Gamma$, connecting
$S_+(\delta,0)=(0,0,1)$ and $S_-(\delta,0)=(0,0,-1)$ given by:
$$\Gamma:=\left\{(r,z)\in\mathbb{R}^2\,:\,-1+\frac{2br}{\coef+1}+z^2=0\right\},$$
which can be parameterized with $t\in \mathbb{R}$ and $\theta_0 \in [0,2\pi)$, by the solutions of the unperturbed system:
\begin{eqnarray}
 r&=&\displaystyle\hetr(t):=\frac{(\coef+1)}{2b}\frac{1}{\cosh^2(\coef t)},\label{hetr}\medskip\\
 \theta&=&\hetth(t,\theta_0):=\theta_0-\frac{\alpha}{\delta}t-\frac{c}{\coef}\log\cosh(\coef t),
\medskip \notag\\
z&=&\hetz(t):=\tanh(\coef t).\label{hetz}
\end{eqnarray}

\subsubsection{Previous results in~\cite{BCS16a} and notation}\label{subsec:previousresults}
The results in this section deal with the existence of adequate parameterizations of the invariant manifolds and the
quantitative (non-sharp) bounds of them and their difference in suitable complex domains.

To recover the singular case in the results in~\cite{BCS16a} take $p=q-2=-2$,  see~\eqref{parameterp}.

\begin{enumerate}
\item \textit{Existence of the invariant manifolds}.
The critical point $S_{-}(\delta, \param)$ (resp. $S_{+}(\delta,\param)$) has a two dimensional unstable (resp. stable) manifold which, in symplectic
polar coordinates, can be written as
\begin{equation}\label{eq:paramunstmanifold}
r= r^{\uns,\sta}(\vu,\theta)=R_0(\vu)+r_{1}^{\uns,\sta}(\vu,\theta),\qquad z=Z_0(u).
\end{equation}
These parameterizations are well defined in $\DoutT{\uns} \times \Tout$ with
$$
\Tout=\left\{\theta\in\mathbb{C}/(2\pi\mathbb{Z})\,:\,|\im\theta|\leq\ost\right\}
$$
and, given constants $\dist, T>0$ sufficiently large and $0<\beta <\pi/2$,
\begin{equation}\label{defDoutT}
\DoutT{\uns}=\Bigg\{\vu\in\mathbb{C} \,:\, |\im\vu|\leq \frac{\pi}{2\coef}-\dist\delta-\tan\beta\re\vu,\, \re \vu \geq -T  \Bigg\},
\end{equation}
for the unstable manifold, and $\DoutT{\sta}=-\DoutT{\uns}$, for the stable one. See Figure~\ref{figDoutunsDinuns} where the domain $\DoutT{\uns}$
is included.

To avoid cumbersome notations, if there is not danger of confusion, from now on we will omit the dependence on variables $(\vu,\theta)$.

We introduce some notation used in~\cite{BCS16a}:
\begin{equation}\label{notationFGHbis}
\bar{X}_1(r)=\mathbf{X}_1(\delta (R_0(u)+r), \theta, \delta \hetz(u),\delta,\delta \param),\qquad \bar{X}_1=(F,G,H),
\end{equation}
where $\mathbf{X}_1=(\Fb,\Gb,\Hb)$ is defined in~\eqref{defFGHbold},
and the operators
\begin{align}
\Lout(r) =& \big(-\delta^{-1}\alpha-c\hetz(\vu)\big )\partial_\theta r +\partial_\vu r -2\hetz(\vu) r \label{defLout}\\
\Fout(r)=&2\param(\hetr(\vu)+r)+\delta^{-2} F(r)+\delta^{-2}\frac{\coef+1}{b}\hetz(\vu)H(r) \notag\\
&-\delta^{-2} G(r)\partial_\theta r-\left(\frac{2br+\delta^{-2} H(r)}{\coef(1-\hetz^2(\vu))}\right)\partial_\vu r.\label{defFout}
\end{align}
The result we use is:
\begin{theorem}[\thmoutloce]\label{thmoutloc-inner}
Consider the PDE:
\begin{equation}\label{PDEequal-inner}
\Lout(r_1^{\uns,\sta}) = \Fout(r_1^{\uns,\sta}).
\end{equation}
Let $0<\beta<\pi/2$ be a constant. There exist $\dist^*\geq1$, $\param^*>0$ and $\delta^*>0$, such that
for all $0<\delta<\delta^*$, if
$\dist=\dist(\delta)$ satisfies:
\begin{equation}\label{conddist-inner}
 \dist^*\delta\leq\dist\delta\leq \frac{\pi}{8\coef},
\end{equation}
and $|\param|\leq\param^*\delta$, the unstable manifold of $S_-(\delta,\param)$ and the stable manifold of $S_+(\delta,\param)$ are given respectively by:
$$
 \big (\sqrt{2r^{\uns,\sta}(\vu,\theta)}\cos\theta,\sqrt{2r^{\uns,\sta}(\vu,\theta)}\sin\theta,\hetz(\vu)),
$$
where, for $(\vu,\theta)\in \DoutT{\uns,\sta} \times\Tout$, the functions $r^{\uns,\sta}$ can be decomposed as
$$
r=r^{\uns,\sta}(\vu,\theta)=R_0(\vu)+r_1^{\uns,\sta}(\vu,\theta),
$$
with $r_1^\uns$ and $r_1^\sta$ satisfying the same equation \eqref{PDEequal-inner}.

Moreover, there exists $M>0$ such that for all $(\vu,\theta)\in \DoutT{\uns,\sta} \times\Tout$:
\begin{align*}
|r^{\uns,\sta}_1(\vu,\theta)|\leq M& \delta|\cosh(\coef\vu)|^{-3},\;\; |\partial_u r^{\uns,\sta}_1(\vu,\theta)|\leq M\delta|\cosh(\coef\vu)|^{-4},\\
&|\partial_{\theta} r^{\uns,\sta}_1(\vu,\theta)|\leq M\delta^2|\cosh(\coef\vu)|^{-4}.
\end{align*}
\end{theorem}
\item \textit{Difference of parameterizations}.
With respect to the difference between the parameterizations \eqref{eq:paramunstmanifold},
defined in
$$
\Doutinter\times\Tout:= (\DoutT{\uns}\cap \DoutT{\sta} )\times \Tout,
$$
we have:
\begin{theorem}[\thmdifpartsolinjectivee]\label{thmdifpartsolinjective}
Let $|\param|\leq\delta\param^*$.
The difference $\Delta:=r_1^{\uns}-r_1^{\sta}$ can be written as:
$$
\Delta(\vu,\theta)=\cosh^{2/\coef}(\coef\vu)(1+P_1(\vu,\theta))\tilde{k}(\xi(\vu,\theta)),
$$
where $\tilde k(\tau)$ is a $2\pi-$periodic function and  the function $\xi$ is defined as:
\begin{equation}\label{defxi}
\xi(\vu,\theta)=\theta+\delta^{-1}\alpha\vu+\coef^{-1}(c+\alpha L_0) \log\cosh(\coef\vu)+\alpha L(\vu)+\chi(\vu,\theta),
\end{equation}
being $L_0\in \mathbb{R}$ a constant. The functions $P_1, L, \chi$ are real analytic functions. In addition,
 $(\xi(\vu,\theta),\theta)$ is injective in $\Doutinter\times\Tout$ and:
\begin{enumerate}
 \item For all $(\vu,\theta) \in \Doutinter\times\Tout$:
\begin{equation}\label{boundLchi-diff}
|L(\vu)| \leq M\,\quad |L'(\vu)| \leq M, \quad |\chi(\vu,\theta)| \leq \frac{M\delta}{|\cosh (\coef \vu)|},
\end{equation}
for some constant $M$.
Moreover, $L(0)=0$ and $L(\vu)$ is defined on the limit $\vu\to i\pi/(2\coef)$.
\item There exists a constant $M$ such that for all  $(\vu,\theta)\in\Doutinter\times\Tout$:
\begin{equation}\label{boundp1-diff}
|P_1(\vu,\theta)|\leq \frac{M\delta}{|\cosh(\coef\vu)|}.
\end{equation}
\end{enumerate}
\end{theorem}

As a straightforward consequence of this result,
\begin{equation}\label{expression-delta}
\Delta(\vu,\theta)
=\cosh^{2/\coef}(\coef\vu)(1+P_1(\vu,\theta))\sum_{l\in\mathbb{Z}}\Upsilon^{[l]}e^{il\xi(\vu,\theta)},
\end{equation}
where $P_1$ and $\xi$ are given in Theorem~\ref{thmdifpartsolinjective} and $\Upsilon^{[l]}$, that are the Fourier
coefficients of the function $\tilde k(\tau)$, are unknown. Of course, they depend on $\delta$ and $\param$ although we do not write it explicitly.
\item \textit{The Fourier coefficients $\Upsilon^{[l]}$}.
Next lemma deals with the exponential smallness of $\Upsilon^{[l]}$ when $l\neq 0$:
\begin{lemma}[\lemUpsilone]\label{lemUpsilonlexpsmall}
Let $\Upsilon^{[l]}$, $l\in\mathbb{Z}$, $l\neq0$, be the coefficients appearing in expression~\eqref{expression-delta} of $\Delta$.
Take $\dist$ as in Theorem~\ref{thmoutloc-inner}.
There exists a constant $M$, independent of $\dist$ such that:
$$
\left|\Upsilon^{[\pm1]}\right|\leq M\frac{\delta^{-2-2/\coef}}{\dist^{3+2/\coef}}e^{-\frac{\alpha\pi}{2\coef\delta}+\alpha\dist},\qquad
\left|\Upsilon^{[l]}\right|\leq M\frac{\delta^{-2-2/\coef}}{\dist^{3+2/\coef}}e^{-\frac{\alpha\pi}{2\coef\delta}\frac{3|l|}{4}}, \qquad |l|\geq 2 .
$$
\end{lemma}

Let us to explain how we can get exponentially small bounds for $\Upsilon^{[l]}$ when $l\neq 0$ from expression~\eqref{expression-delta}
in the easiest case: when $\xi(\vu,\theta)= \theta + \delta^{-1} \alpha \vu$, that is to say the constants $c=L_0=0$ and the functions
$L\equiv \chi\equiv 0$. Indeed, in this case
$$
\sum_{l\in\mathbb{Z}}\Upsilon^{[l]} e^{il \delta^{-1} \alpha \vu } e^{il \theta } =
\frac{\Delta (\vu,\theta)}{\cosh ^{2/\coef}(\coef\vu)(1+P_1(\vu,\theta))}.
$$
Therefore, $\Upsilon^{[l]} e^{il \delta^{-1} \alpha \vu }$ are the Fourier coefficients of a $2\pi$-periodic in $\theta$ function and
we have that, for all $\vu\in \Doutinter$ (recall that $\Delta(\vu,\theta)$ is defined on $\Doutinter\times \Tout$):
$$
\Upsilon^{[l]} = e^{-il \delta^{-1} \alpha \vu } \frac{1}{2\pi} \int_{0}^{2\pi}
\frac{\Delta (\vu,\theta)}{\cosh ^{2/\coef}(\coef\vu)(1+P_1(\vu,\theta))} e^{-il\theta}\, d\theta.
$$
Since this equality holds true for any value of $\vu\in \Doutinter$, we take $\vu = \vu_+:=\pi/(2\delta)-\dist$ when $l>0$ and
conversely $\vu = \vu_-:=-\pi/(2\delta)+\dist$ when $l<0$. Using Theorem~\ref{thmoutloc-inner}, we have that
$$
|\Delta (\vu_{\pm},\theta)| \leq |r_{1}^{\uns}(\vu_{\pm},\theta)|+|r_1^{\sta}(\vu_{\pm},\theta)|\leq M \delta^{-2} \dist^{-3}.
$$
Therefore, using Theorem~\ref{thmdifpartsolinjective} to bound $P_1$, one has that
$$
|\Upsilon^{[l]}|\leq K \delta^{-2 -2/\coef} \dist^{-3-2/\coef} e^{-|l| \left (\frac{\alpha \pi}{2\coef \delta} +  \alpha \dist\right )}
$$
which is an exponentially small bound for $\Upsilon^{[l]}$ when $l\neq 0$.

Of course, this is an extremely easy case, but it shows that it is crucial to have quantitative bounds of the behavior of
the parameterization of the invariant manifolds at points $\mathcal{O}(\delta)-$close to the singularities $\pm i \pi/(2\coef)$.

To finish, we present the following result which deals with the average $\Upsilon^{[0]}$.

\begin{theorem}[\thmktildeave]\label{thmktilde0}
Let $\Upsilon^{[0]}$ be the average of the function $\tilde k(\tau)$ appearing in Theorem~\ref{thmdifpartsolinjective}.
\begin{enumerate}
\item
In the conservative case, for all $0\leq\delta\leq\delta_0$ one has
$\Upsilon^{[0]}=0$.
\item In the dissipative case, there exists a curve $\param=\param_*^0(\delta)=\mathcal{O}(\delta)$ such that for all
$0\leq\delta\leq\delta_0$ one has:
$$\Upsilon^{[0]}=\Upsilon^{[0]}(\delta,\param_*^0(\delta))=0.$$
In addition, given constants $a_1$, $a_2\in\mathbb{R}$ and $a_3>0$, there exists a curve $\param=\param_*(\delta)=\mathcal{O}(\delta)$
such that for all $0\leq\delta\leq\delta_0$ one has:
\begin{equation}\label{Upsilon0expsmall}
\Upsilon^{[0]}=\Upsilon^{[0]}(\delta,\param_*(\delta))=a_1\delta^{a_2}e^{-\frac{a_3\pi}{2\coef\delta}}.
\end{equation}
\end{enumerate}
\end{theorem}
\end{enumerate}

These are the results in~\cite{BCS16a} we will use along this work. They (among others) lead to the following result
\begin{theorem}[\mainthmregular]
In the dissipative case we take $\param=\param_*(\delta)$, where $\param^*$
is one of the curves defined in Theorem~\ref{thmktilde0}. Let $\Upsilon^{[0]}=\Upsilon^{[0]}(\param_*(\delta),\delta)$ be the constant
provided by this Theorem. In the conservative case recall that $\Upsilon^{[0]}=0$.
Let $\vt(\vu,\delta)=\delta^{-1}\alpha\vu+c\coef^{-1}\left[\log\cosh(\coef\vu)-\log\delta\right]$.

There exist constants $\mathcal{C}_1,\mathcal{C}_2$ such that, given $T_0>0$, for all $\vu\in[-T_0,T_0]$ and $\theta\in\mathbb{S}^1$
\begin{eqnarray*}
 \Delta(\vu,\theta)&=&\cosh^{\frac{2}{\coef}}(\coef\vu)\Upsilon^{[0]}\Big(1+\mathcal{O}(\delta)\Big)\\
&&+\delta^{-2-\frac{2}{\coef}}\cosh^{\frac{2}{\coef}}(\coef\vu)e^{-\frac{\alpha\pi}{2\coef\delta}}
\Bigg[\mathcal{C}_1\cos\Big(\theta+\vt(\vu,\delta)-\alpha \coef^{-1} L_0 \log \delta \Big)\\
&&+\mathcal{C}_2\sin\Big(\theta+\vt(\vu,\delta) -\alpha \coef^{-1} L_0 \log \delta \Big)+\mathcal{O}(1)\Bigg],
\end{eqnarray*}
where we recall that $\coef=1$ in the conservative case.
\end{theorem}

This result is not an asymptotic formula for the difference in the singular case, but it provides a (sharp) upper bound for $\Delta$.
Assuming the results in Theorems~\ref{thmoutloc-inner}, \ref{thmdifpartsolinjective} and~\ref{thmktilde0} and Lemma~\ref{lemUpsilonlexpsmall},
its proof consists
in finding $\Delta_0$, a good approximation of $\Delta$ by means of suitable approximations of $r_{1}^{\uns,\sta}$ on domains containing points
$\mathcal{O}(\delta)-$close to the singularities, at $u=\pm i \pi/(2\coef)$, of the parameterization of the heteroclinic connection
$\Gamma$ in~\eqref{hetr}, \eqref{hetz}.
Roughly speaking, $\Delta_0$ comes from a Poincar\'e-Melnikov perturbation theory and only gives an asymptotic formula in the
regular case, that is, when $p=q-2>-2$.

Let us to explain why the classical perturbation theory does not work in the singular ($p=-2$) case. Indeed,
the approximations of $r_{1}^{\uns,\sta}$, which we called $r_{10}^{\uns,\sta}$ in~\cite{BCS16a}, satisfy that (see \thmoutloce):
$$
\big (\hetr(\vu) \big )^{-1} r_{10}^{\uns,\sta}(\vu,\theta)= \mathcal{O}(\delta^{p+2}),\qquad
\big (\hetr(\vu) \big )^{-1} (r_1^{\uns,\sta} (\vu,\theta)-r_{10}^{\uns,\sta}(\vu,\theta))=\mathcal{O}(\delta^{2(p+2)}).
$$
Therefore, for $p=-2$ they are of $\mathcal{O}(1)$ when $\vu=\pm i\pi/(2\coef)+\mathcal{O}(\delta)$ and consequently $r_{10}^{\uns,\sta}$
are not good approximations (in the relative error sense) of the parameterizations $r_1^{\uns,\sta}$ anymore.
In this paper, we look for suitable approximations of $r_{1}^{\uns,\sta}(\vu,\theta)$ when $\vu$ is $\mathcal{O}(\delta)$ close to the singularities
and $p=-2$ by means of the \textit{inner equation} which is introduced, and deeply studied, in the following section.

In fact, since we are in the real analytic setting, we only need to look for good approximations for $r_{1}^{\uns,\sta}(\vu,\theta)$
when $\vu$ is close to the singularity $i \pi/(2\coef)$. From now on, we will restrict ourselves to these values of $\vu$.

To finish this section we present the strategy we will follow from now on to prove Theorem~\ref{mainthm-inner-intro}:
\begin{enumerate}
\item In Section~\ref{subsec-introinner} we deal with the inner equation. First we derive it, through suitable changes of variables
$(\vu,\theta)\to (\vs,\theta)$. After that we find two particular solutions of the inner
equation which we call
$\psiin^\uns(\vs,\theta)$ and $\psiin^\sta(\vs,\theta)$, satisfying some asymptotic conditions.
Finally, we shall find an asymptotic formula for their difference
$\Delta\psiin(\vs,\theta):=\psiin^\uns(\vs,\theta)-\psiin^\sta(\vs,\theta)$.
\item In Section~\ref{subsec:intromatching} we shall see that in suitable complex domains and after appropriate scaling and changes of variables,
$\psiin^\uns(\vs,\theta)$ and $\psiin^\sta(\vs,\theta)$ are
respectively good approximations of the functions $r_1^\uns(\vu,\theta)$ and $r_1^\sta(\vu,\theta)$ defined in Theorem \ref{thmoutloc-inner}.
\item Finally, in Section \ref{subsec:introformuladiff}, using the previous results, one can find an asymptotic
formula of $\Delta(\vu,\theta)$ in terms of $\Delta\psiin(\vs,\theta)$.
\end{enumerate}

From now on, in the dissipative case, we assume that $\param$ lies on one of the curves $\param_*(\delta)$ given by Theorem \ref{thmktilde0}.
In particular $\sigma= \mathcal{O}(\delta)$.
\subsection{The inner equation}\label{subsec-introinner}
The inner equation is an independent of $\delta$ and $\param$ equation having the dominant quantitative behavior of the partial differential
equation~\eqref{PDEequal-inner} when $\vu$ is $\mathcal{O}(\delta)$ close to the singularity $ i\pi/(2\coef)$.
In Section~\ref{subsec:derivationinner} we explain how to derive it. After, in Section~\ref{subsec:introinnereq} we deal with the existence of suitable
solutions, $\psiin^{\uns,\sta}$ of this inner equation and finally in Section~\ref{subsec:introdiffinner} we study the form of the difference between them.

\subsubsection{Derivation of the inner equation}\label{subsec:derivationinner}
We consider the following change of variables:
\begin{equation}\label{change-s}
\vs=\vs(\vu)=\frac{1}{\delta\hetz(\vu)}, \qquad  \vu=\vu(\vs)=\hetz^{-1}\left ( \frac{1}{\delta \vs}\right ),
\end{equation}
where we recall that $\hetz(\vu) = \tanh(\coef \vu)$.
\begin{remark}\label{rmkchange-s}
The change~\eqref{change-s} is well-defined for $\vu$ belonging to some sufficiently small neighborhoods of $i\pi/(2\coef)$.
For instance if $\vu\in \DoutT{\uns,\sta}$ and $\im \vu\geq \pi/(4\coef)$.

The usual inner change of variables would be $\delta \vs = \coef \big (\vu - i\pi/(2\coef))$,
see~\cite{BaldomaInner, BFGS12}. However we notice that, when
we are close to $i \pi/(2\coef)$, we have that
$$
s(u)=
\frac{\coef}{\delta}\left (\vu-\frac{i\pi}{2\coef} \right ) +  \mathcal{O}\left (\frac{(\coef \vu -i\pi/2)^3}{\delta}\right )
$$
and then, both changes of variables are close to each other  close to the singularity $i\pi/(2\coef)$.
\end{remark}

We consider now $r^{\uns,\sta}(\vu(\vs),\theta)= \hetr(\vu(\vs))+ r_1^{\uns,\sta}(\vu(\vs),\theta)$ where $r_1^{\uns,\sta}$ are the functions given in
Theorem~\ref{thmoutloc-inner}. It is clear that 
\begin{equation}\label{expR0us}
\begin{aligned}
\hetr(\vu(\vs))&=\frac{(\coef+1)}{2b}\frac{1}{\cosh^2(\coef \vu(\vs))}=\frac{\coef+1}{2b }\left(1-\frac{1}{\delta^2\vs^2}\right), \\
r_1^{\uns,\sta}(\vu(\vs),\theta) &= \mathcal{O}\left (\frac{1}{\delta^2 s^{3}}\right).
\end{aligned}
\end{equation}
Consequently, it is convenient to introduce, for a function $r_1$ defined either on $\DoutT{\uns}$ or $\DoutT{\sta}$, the new function:
\begin{equation}\label{defpsi}
\psi(\vs,\theta)=\delta^2r_1\left(\vu(\vs),\theta\right).
\end{equation}

If there is not danger of confusion, we will omit the dependence on variables $\vu,\vs, \theta$ assuming that $\psi$
is always evaluated in $(\vs,\theta)$.

Assume now that $r_1$ is a solution of $\Lout(r_1)=\Fout(r_1)$, where $\Lout$ and $\Fout$ were defined in~\eqref{defLout} and~\eqref{defFout}.
Then $\psi$ defined by~\eqref{defpsi} is a solution of the partial differential equation:
\begin{equation}\label{PDE-psi-new}
-\alpha\partial_\theta\psi+\coef\partial_\vs\psi-2\vs^{-1}\psi = c\vs^{-1}\partial_\theta\psi
+\coef \delta^2\vs^2\partial_\vs\psi+
\delta^3 \Fout(\delta^{-2} \psi).
\end{equation}

We restate Theorem~\ref{thmoutloc-inner} in the new variables ($s$ and $\psi$). Notice that the resulting functions, $\psi^{\uns,\sta}$, are defined
in the restricted domain $\DoutTinvars{\uns}\times\Tout$, with:
\begin{equation}\label{defDoutTuns-innervars}
\DoutTinvars{\uns,\sta}=\left\{\vs\in\mathbb{C} \,:\, \vs=\vs(\vu)=\frac{1}{\delta\hetz(\vu)},\,\vu\in\DoutT{\uns,\sta}\cap
\left\{\im\vu\geq\frac{\pi}{4\coef}\right\}\right\}.
\end{equation}
\begin{theorem}\label{thmoutloc-innervariables}
Consider the functions $r_1^{\uns,\sta}$, given by Theorem~\ref{thmoutloc-inner}, and define:
$$
\psi^{\uns,\sta}(\vs,\theta)=\delta^{2}r_1^{\uns,\sta}\left (\vu(\vs),\theta\right ),\qquad\qquad(\vs,\theta)\in\DoutTinvars{\uns,\sta}\times\Tout
$$
satisfying~\eqref{PDE-psi-new}. There exists $M>0$ such that, for $(\vs,\theta)\in\DoutTinvars{\uns,\sta}\times\Tout$,
$$
|\psi^{\uns,\sta}(\vs,\theta)|\leq \frac{M}{|\vs|^{3}},\qquad |\partial_{\vs} \psi^{\uns,\sta}(\vs,\theta)|\leq \frac{M}{|\vs|^{4}},\qquad
|\partial_{\theta} \psi^{\uns,\sta}(\vs,\theta)|\leq \frac{M}{|\vs|^{4}}.
$$
\end{theorem}
\begin{remark}
To check this result we use that, $|\delta \vs|\leq K$ if $\vs \in \DoutTinvars{\uns,\sta}$.
\end{remark}

The inner equation is the equation~\eqref{PDE-psi-new} by taking $\delta=0$.
In order to see that $\delta^3 \Fout(\delta^{-2} \psi)$ is well
defined for $\delta=0$ we need to make some computations. Note that, since we want to apply the results
of Theorem~\ref{thmoutloc-innervariables}, for $|\sigma| \leq \sigma^* \delta$, we will also have that $\sigma=0$.
We define
\begin{equation}\label{defrho}
\rho(\psi,\vs,\delta)=\delta \sqrt{2 (\hetr(\vu(\vs)) + r_{1}(\vu(\vs)))}=\sqrt{\frac{\coef+1}{b}\left(-\vs^{-2}+\delta^2\right)+2\psi}.
\end{equation}
One of the terms in $\delta^3 \Fout(\delta^{-2}\psi)$ (see~\eqref{defFout}) is $\delta F(\delta^{-2} \psi)$ with
$F$ defined by~\eqref{notationFGHbis}. Let us to compute it:
\begin{align*}
\delta F(\delta^{-2} \psi) &= \delta \Fb(\delta(\hetr(\vu(\vs)) + r_1(\vu(\vs),\theta), \theta, \delta \hetz(\vu(\vs)), \delta,\delta \param) \\
&=\rho(\psi,\vs,\delta) \big [\cos \theta f(\xi(\psi,\vs,\delta)) + \sin \theta g(\xi(\psi,\vs,\delta))\big ]
\end{align*}
with $\xi(\psi,\vs,\delta) =(\rho(\psi,\vs,\delta)\cos \theta, \rho(\psi,\vs,\delta)\sin \theta,\delta \hetz(u(s)), \delta,\delta \param)$. Note that,  
applying the change of variables~\eqref{change-s},
$$
\xi (\psi,\vs,\delta)=
(\rho(\psi,\vs,\delta)\cos \theta, \rho(\psi,\vs,\delta)\sin \theta,\vs^{-1}, \delta,\delta \param).
$$
Analogously one can see that, defining
\begin{align}\label{notationFGHhat}
 \hat F(\psi,\delta)&:= \delta F(\delta^{-2}\psi)=\rho(\psi,\vs,\delta) \big [\cos \theta f(\xi(\psi,\vs,\delta)) + \sin \theta g(\xi(\psi,\vs,\delta))\big ] \notag\\
 \hat G(\psi,\delta)&:=
\delta^{-1} G(\delta^{-2}\psi)= \frac{1}{\rho(\psi,\vs,\delta)} \big [\cos \theta g(\xi(\psi,\vs,\delta)) - \sin \theta f(\xi(\psi,\vs,\delta))\big ] \notag\\
 \hat H(\psi,\delta)&:=H(\delta^{-2}\psi) = h(\xi(\psi,\vs,\delta)),
\end{align}
then
\begin{equation}\label{expgoodFout}
\begin{aligned}
\delta^{3} \Fout (\delta^{-2} \psi) = &\param\rho^{2}(\psi,\vs,\delta)
 +\hat F(\psi,\delta)+\frac{\coef+1}{b}\vs^{-1}\hat H(\psi,\delta)\\
&-\hat G(\psi,\delta)\partial_\theta\psi+\vs^2\left(2b\psi+\hat H(\psi,\delta)\right)\partial_\vs\psi.
\end{aligned}
\end{equation}
It is clear then that $\delta^3 \Fout (\delta^{-2}\psi)$ is well defined for $\delta=0$. 
So the inner equation is well defined.

To finish this section, we define the operators:
\begin{align}
\Linner(\psi)=&-\alpha\partial_\theta\psi+\coef\partial_\vs\psi-2\vs^{-1}\psi, \label{defopLinner} \\
\Minner(\psi,\delta)=&c\vs^{-1}\partial_\theta\psi+\coef\delta^2\vs^2\partial_\vs\psi + \param\rho^{2}(\psi,\vs,\delta) +\hat F(\psi,\delta)
+\frac{\coef+1}{b}\vs^{-1}\hat H(\psi,\delta)\notag\\
&-\hat G(\psi,\delta)\partial_\theta\psi
+\vs^2\left(2b\psi+\hat H(\psi,\delta)\right)\partial_\vs\psi. \label{defopMinner}
\end{align}
Then equation~\eqref{PDE-psi-new}, see~\eqref{expgoodFout} for the expression of $\delta^3 \Fout (\delta^{-2}\psi)$, can be written as:
\begin{equation}\label{PDE-psishort}
 \Linner(\psi)=\Minner(\psi,\delta).
\end{equation}
The inner equation is obtained by taking $\delta=0$ in~\eqref{PDE-psishort}:
\begin{equation}\label{PDE-innershort}
 \Linner(\psiin)=\Minner(\psiin,0)
\end{equation}
or equivalently:
\begin{align}\label{PDE-inner}
 -\alpha\partial_\theta\psiin+&\coef\partial_\vs\psiin-2\vs^{-1}\psiin=c\vs^{-1}\partial_\theta\psiin+\hat F(\psiin,0)
 +\frac{\coef+1}{b}\vs^{-1}\hat H(\psiin,0)\nonumber \\&-\hat G(\psiin,0)\partial_\theta\psiin
+\vs^2\left(2b\psiin+\hat H(\psiin,0)\right)\partial_\vs\psiin.
\end{align}
\begin{remark} As it was remarked to the authors by V. Gelfreich, the inner equation is very related with the initial Hopf-zero singularity $X^\ast$. 
In fact with system~\eqref{sistema-NForder3} taking the parameters $\nu=\mu=0$, which is the vector field $X^\ast$ after the normal form procedure 
of order two.
This is due to that, besides the change $\vs=\vs(\vu)$, to find the inner equation we undo the scalings performed in Section~\ref{subsec:scaling}
and put $\delta=0$. Indeed, let us give more details. 
We perform to system $X^{\ast}$, the change of coordinates~\eqref{cylindriccoordsunpertub}:
$$
\bar x=\sqrt{2 R} \cos \theta ,\qquad \bar y=\sqrt{2R}\sin \theta,\qquad \bar z = \frac{1}{s}.
$$
It turns out that, taking $\bar f=\bar g =\bar h=0$, the new system has a solution parameterizated by
$R=\tilde{R_0}(\vs)=-(1+\beta_1)/(2\gamma_2 s^2)$. Notice that, by expression~\eqref{expR0us} of $\hetr(\vu(\vs))$ and recalling that $\beta_1=\coef$ and $\gamma_2 =b$,
$\tilde R_0(\vs) = \big[\delta^2 R_0(\vs)\big ]_{| \delta =0}$. 
Therefore, if we look for a solution of the new system parameterizated by $(\theta, s)$ of the form
$$
R= -\frac{\beta_1 +1}{2\gamma_2 s^2} + \varphi(s,\theta) =-\frac{\coef +1}{2b s^2}+ \varphi(s,\theta)
$$
with $\varphi(\vs,\theta) \to 0$ as $\vs \to \infty$, the partial differential equation that $\varphi$ satisfies is exactly the inner equation.
\end{remark} 
\subsubsection{Study of the inner equation}\label{subsec:introinnereq}
First, we introduce the complex domains in which equation~\eqref{PDE-innershort} will be solved.
Given $\beta_0,\distin>0$, we define (see Figure \ref{figDinuns}):
\begin{figure}
	\centering
	\includegraphics[width=6cm]{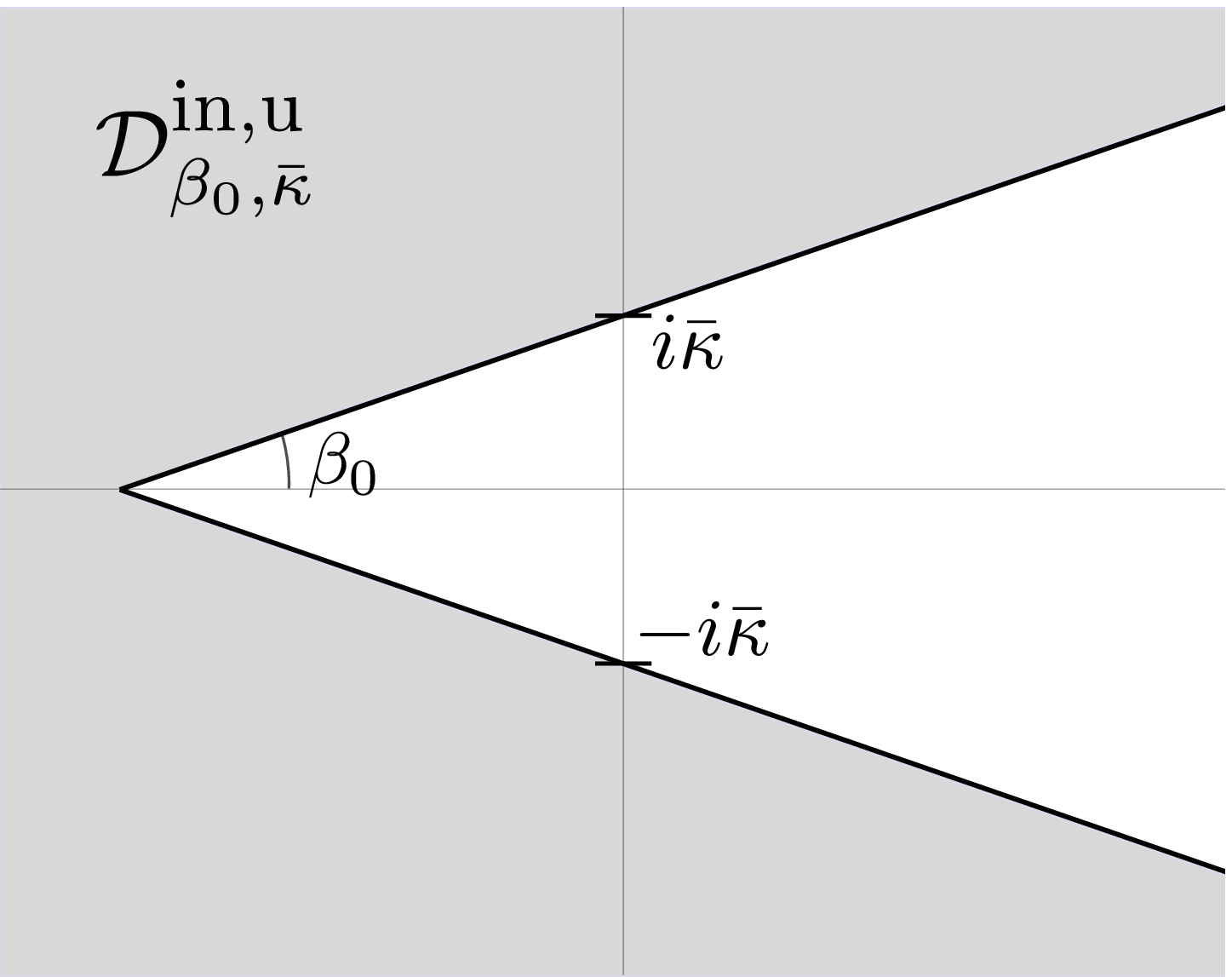}
        \caption{The domain $\Din{\uns}$ in the $s$ plane.}\label{figDinuns}
\end{figure}
\begin{equation}\label{defdinus-inner2D}
\Din{\uns}=\{\vs\in\mathbb{C}\,:\, |\im \vs|\geq\tan\beta_0\re \vs+\distin\},\qquad \Din{\sta}=-\Din{\uns}.
\end{equation}
We also define analogous domains to $\Din{\uns,\sta}$, in terms of the outer variables $\vu$:
\begin{equation}\label{defdin-outervars}
\Dinu{\uns}=\left \{\vu\in\mathbb{C}\,:\, \left |\im \left (\vu-i\frac{\pi}{2\coef}\right )\right|\geq\tan\beta_0\re \vu+\distin\delta\right\}
\end{equation}
and $\Dinu{\sta}=-\Dinu{\uns}$.
It is easy to check that taking $\distin=\dist/2$, where $\dist$ is the parameter defining the domains $\DoutT{\uns}$ introduced in~\eqref{defDoutT}, and choosing an adequate $T>0$, then for all $0<\beta_0,\beta<\pi/2$ and for $\delta$ small enough one has that $\DoutT{\uns}\subset\Dinu{\uns}$ (see Figure \ref{figDoutunsDinuns}). Analogously, we also have that $\DoutT{\sta}\subset\Dinu{\sta}$.
\begin{figure}
	\centering
	\includegraphics[width=6cm]{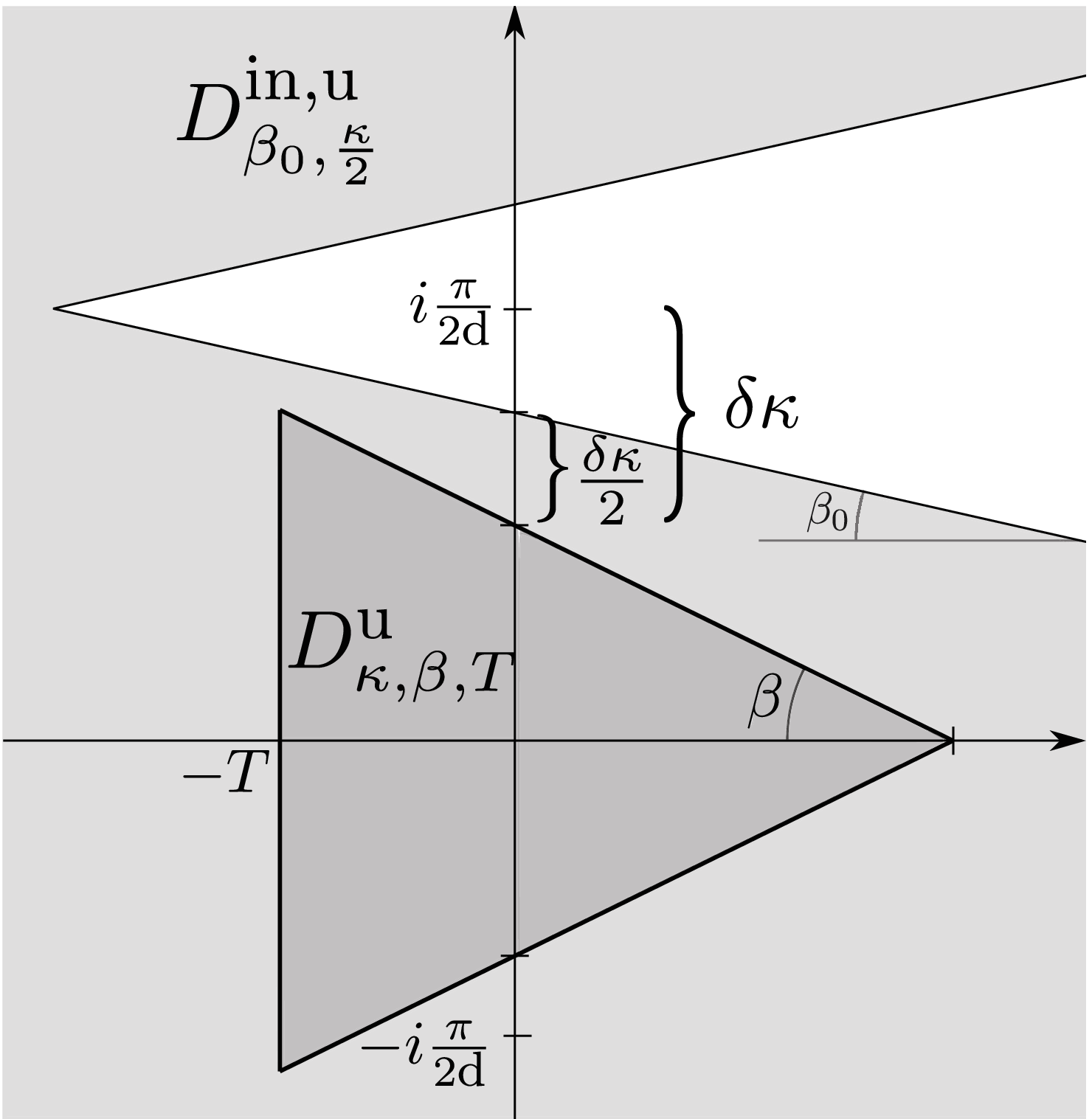}
        \caption{The domains $\DoutT{\uns}$ and $\Dinu{\uns}$ with $\bar\dist=\dist/2$ in the $u$ plane.}\label{figDoutunsDinuns}
\end{figure}

We will look for particular solutions $\psiin^{\uns}$ and $\psiin^{\sta}$ of equation \eqref{PDE-innershort} satisfying:
\begin{equation}\label{condpsi0uns}
\lim_{\re \vs\to\mp\infty}|\psiin^{\uns,\sta}(\vs,\theta)|\to0,\qquad (\vs,\theta)\in \Din{\uns,\sta}\times\Tout.
\end{equation}
Where we take the sign $-$ for $\uns$ and $+$ for $\sta$.
We will find these solutions by means of a suitable right inverse of the operator $\Linner$ defined in \eqref{defopLinner}. More precisely, assume $\Ginner$ is such that $\Linner\circ\Ginner=\mathrm{Id}$. If $\psiin$ satisfies the implicit equation:
\begin{equation}\label{innerfixedpoint-intro}
 \psiin=\Ginner(\Minner(\psiin,0)),
\end{equation}
then clearly $\psiin$ is a solution of equation \eqref{PDE-innershort}. In other words, $\Ginner$ allows us to write \eqref{PDE-innershort} as the fixed point equation \eqref{innerfixedpoint-intro}. We now introduce the right inverse used in each case: $\Ginner^{\uns,\sta}$, which will allow us to prove the existence of the functions
$\psiin^{\uns,\sta}$ satisfying \eqref{condpsi0uns}. We shall refer to each case as the ``unstable'' one and the ``stable'' one because we shall see, later on in Theorem~\ref{thmmatching},
that each one approximates respectively the unstable or stable manifold (more precisely, $\psi^\uns$ and $\psi^\sta$ defined in Theorem~\ref{thmoutloc-innervariables})
in some bounded subdomains of $\Din{\uns}$ and $\Din{\sta}$ respectively.

Given a function $\phi(\vs,\theta)$, $2\pi-$periodic in $\theta$, we define $\Ginner^{\uns,\sta}$ as:
\begin{equation}\label{defGuns}
 \Ginner^\us(\phi)(\vs,\theta)=\sum_{l\in\mathbb{Z}}{\Ginner^{\us}}^{[l]}(\phi)(\vs)e^{il\theta},\qquad  \us=\uns,\sta
\end{equation}
where the Fourier coefficients ${\Ginner^{\us}}^{[l]}(\phi)$ are defined as:
$$
{\Ginner^{\us}}^{[l]}(\phi)(\vs)=\vs^{\frac{2}{\coef}}\int_{\mp\infty}^\vs\frac{e^{-\frac{il\alpha}{\coef}(w-\vs)}}{w^{\frac{2}{\coef}}}\phi^{[l]}(w)dw,,\qquad  \us=\uns,\sta
$$
where, we take $-\infty$ for the unstable case and $+\infty$ for the stable one.
Here $\phi^{[l]}$ stands for the $l-$th Fourier coefficient of $\phi$, and $\int_{\mp\infty}^s$ means the integral over any path included
in $\Din{\us}$ such that $\re s\to\mp\infty$.

One can easily check that:
\begin{equation}\label{eqinverses}
\Linner\circ\Ginner^\uns=\Linner\circ\Ginner^\sta=\textrm{Id}.
\end{equation}
The following theorem states the existence of both functions $\psiin^\uns$ and $\psiin^\sta$. Its proof can be found in Section \ref{sec:inner}.

\begin{theorem}\label{thminner}
Let $\beta_0>0$ and $\distin$ be large enough. Then equation \eqref{PDE-innershort} has two solutions $\psiin^\uns$ and $\psiin^{\sta}$, defined respectively in
$\Din{\uns}$ and $\Din{\sta}$, such that there exists $M>0$:
$$|\psiin^{\uns,\sta}(\vs,\theta)|\leq \frac{M}{|\vs|^{3}},\qquad |\partial_{\vs} \psiin^{\uns,\sta}(\vs,\theta)|\leq \frac{M}{|\vs|^{4}},\qquad
|\partial_{\theta} \psiin^{\uns,\sta}(\vs,\theta)|\leq \frac{M}{|\vs|^{4}}.
$$
for all $(\vs,\theta)\in\Din{\uns,\sta}\times \Tout$.
Moreover:
$$|\psiin^{\uns,\sta}(\vs,\theta)-\Ginner^{\uns,\sta}(\Minner(0,0))(\vs,\theta)|\leq \frac{M}{|\vs|^{4}},\qquad  (\vs,\theta)\in\Din{\uns,\sta} \times \Tout.$$
\end{theorem}

\subsubsection{Study of the difference $\Delta\psiin=\psiin^\uns-\psiin^\sta$}\label{subsec:introdiffinner}
Once the existence of these two particular solutions of the inner equation~\eqref{PDE-innershort}, $\psiin^{\uns}$ and $\psiin^{\sta}$, is established,
one can look for an asymptotic expression of their difference $\Delta\psiin=\psiin^\uns-\psiin^\sta$.
We will study this difference in $\Ein\times\Tout$, where $\Ein$ is the domain (see Figure \ref{figEin}):
\begin{figure}
	\centering
	\includegraphics[width=6cm]{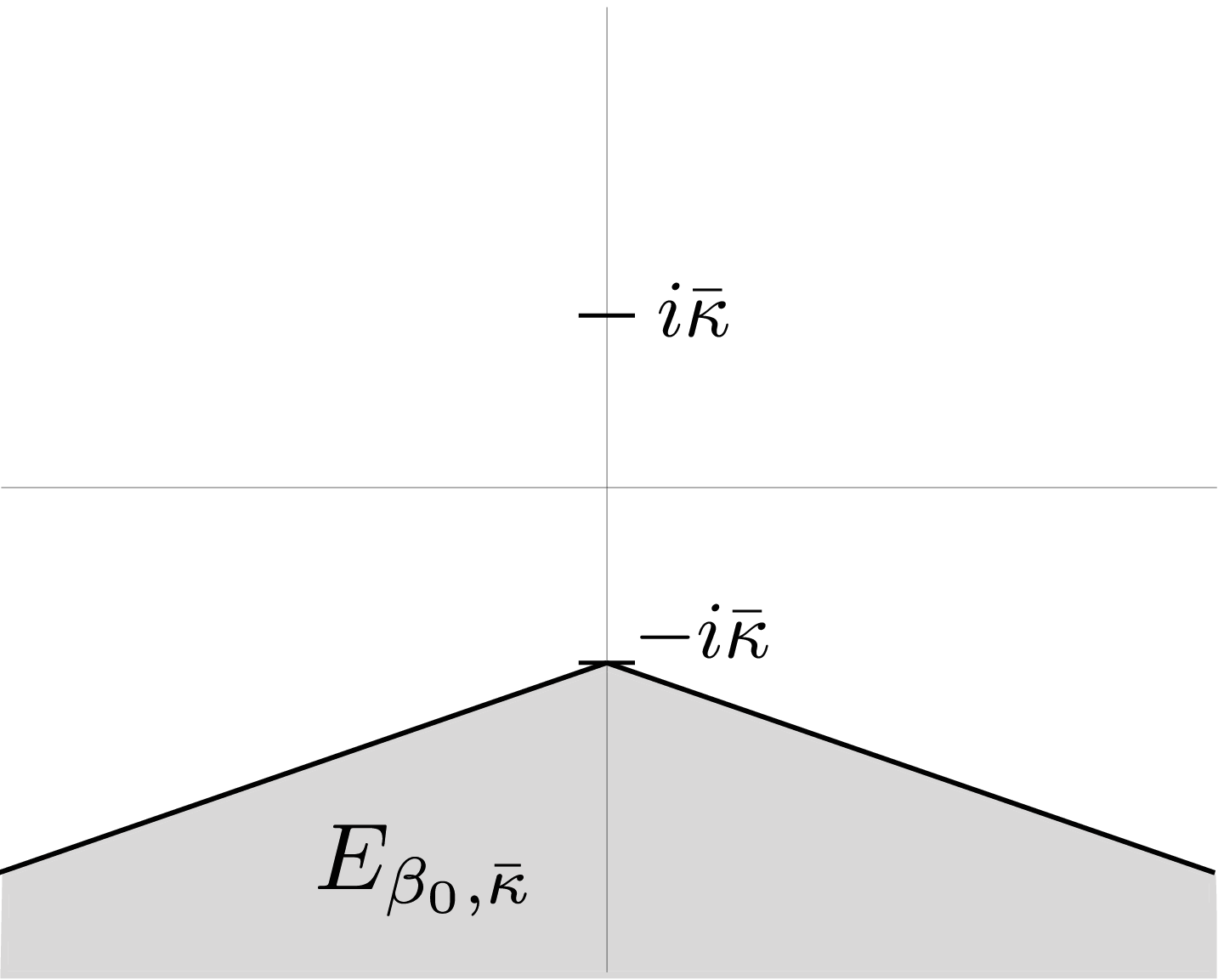}
        \caption{The domain $\Ein$.}\label{figEin}
\end{figure}
$$
\Ein=\Din{\uns}\cap\Din{\sta}\cap\{\vs\in\mathbb{C}\,:\,\im\vs<0\}.
$$

Subtracting equations \eqref{PDE-inner} for $\psiin^\uns$ and $\psiin^\sta$ and using the mean value theorem, one obtains a linear equation for $\Delta\psiin$ of the following form:
\begin{align}\label{PDE-difference-intro}
-\alpha\partial_\theta\Delta\psiin&+\coef\partial_\vs\Delta\psiin-2\vs^{-1}\Delta\psiin\\
&=a_1(\vs,\theta)\Delta\psiin+a_2(\vs,\theta)\partial_\vs\Delta\psiin+(c\vs^{-1}+a_3(\vs,\theta))\partial_\theta\Delta\psiin, \nonumber
\end{align}
for certain ``small'' (as $|s|\to \infty$) functions $a_1, a_2$ and $a_3$, which we will specify in Section~\ref{sec:diffinner}.
Of course, $a_i$, $i=1,2,3$, depend on $\psiin^\uns$ and $\psiin^\sta$, which now are known functions.
Since $\Delta\psiin$ is a solution of~\eqref{PDE-difference-intro}, we first study the form that all solutions of this equation have.
Next we give the main ideas of how this can be done, which basically are the same as in \secDifference.

Let $\Pin$ be a particular solution of~\eqref{PDE-difference-intro} such that $\Pin(\vs,\theta)\neq0$ for $(\vs,\theta)\in\Ein\times\Tout$.
Then, every solution of~\eqref{PDE-difference-intro} can be written as:
$$
\Delta\psiin(\vs,\theta)=\Pin(\vs,\theta)\kin(\vs,\theta),
$$
where $\kin(\vs,\theta)$ is a solution of the homogeneous equation:
\begin{equation}\label{PDE-k-intro}
-\alpha\partial_\theta \kin+\coef\partial_\vs \kin=a_2(\vs,\theta)\partial_\vs \kin+(c\vs^{-1}+a_3(\vs,\theta))\partial_\theta \kin.
\end{equation}

First let us describe how we can find a suitable particular solution $\Pin$ of equation \eqref{PDE-difference-intro}. Since the functions $a_1, a_2$ and
$a_3$ are ``small'' and $c\vs^{-1}$ is also small if we take $|\vs|$ sufficiently large,
equation~\eqref{PDE-difference-intro} can be regarded as a small perturbation of:
$$
-\alpha\partial_\theta\Delta\psiin+\coef\partial_\vs\Delta\psiin-2\vs^{-1}\Delta\psiin=0.
$$
This equation has a trivial solution given by $\Pin_0(\vs,\theta)=\vs^{2/\coef}$.
Thus, we will look for a solution of the form $\Pin(\vs,\theta)=\vs^{2/\coef}(1+\Pinsmall(\vs,\theta))$
where $\Pinsmall$ will be a ``small'' function. Note that, being $\Pinsmall$ small, we will ensure that
$\Pin(\vs,\theta)\neq0$ for $(\vs,\theta)\in\Ein\times\Tout$. The rigorous statement of this result can be found in
Proposition~\ref{propP1} of Section~\ref{sec:diffinner}.

Now we shall sketch the study of equation~\eqref{PDE-k-intro}. In fact, equations of the form~\eqref{PDE-k-intro} have been studied in previous works.
One of its main features is that if $\xi(\vs,\theta)$ is a solution of \eqref{PDE-k-intro} such that $(\xi(\vs,\theta),\theta)$ is injective in $\Ein\times\Tout$,
then any solution $\kin(\vs,\theta)$ of this equation can be written as:
$$\kin(\vs,\theta)=\kintilde(\xi(\vs,\theta)),$$
for some function $\kintilde(\tau)$ which has to be $2\pi$-periodic in $\tau$. To find $\xi$, one could proceed as we did with $\Pin$. Indeed, since the functions
$a_2(\vs,\theta)$ and $c\vs^{-1}+a_3(\vs,\theta)$ are ``small'', one could derive that the dominant part of equation~\eqref{PDE-k-intro} is given by:
\begin{equation}\label{dominantpart-PDE-k}
-\alpha\partial_\theta \kin+\coef\partial_\vs \kin=0.
\end{equation}
A trivial solution of \eqref{dominantpart-PDE-k} is given by:
$\xi_0(\vs,\theta)=\theta+\coef^{-1}\alpha\vs$,
and thus we could expect to find a suitable solution of~\eqref{PDE-k-intro} given by $\xi=\xi_0+\xi_1$,
where $\xi_1$ is supposed to be a ``small'' function.

However, as we shall see in Section~\ref{sec:diffinner}, this is not quite accurate.
Nevertheless, to some extent it summarizes the underlying idea of the proof. In fact, we will see that the dominant part of equation \eqref{PDE-k-intro} is:
\begin{equation}\label{dominantpart-PDE-k-true}
-\alpha\partial_\theta \kin+\coef\partial_\vs \kin=\coef L_0\vs^{-1}\partial_\vs \kin+c\vs^{-1}\partial_\theta \kin,
\end{equation}
where $L_0$ is the constant given in Theorem~\ref{thmdifpartsolinjective}.
This constant is closely related to the function $a_2(\vs,\theta)$.
Indeed, in Lemma~\ref{lema123}, we will see that:
$$a_2(\vs,\theta)=\frac{\tilde a_2(\theta)}{\vs}+\mathcal{O}(\vs^{-2}),$$
and $\coef L_0$ is the average of $\tilde a_2(\theta)$. Note that the function $\xi_0$ defined as:
$$\xi_0(\vs,\theta)=\theta+\coef^{-1}\alpha\vs+\coef^{-1}(c+\alpha L_0)\log\vs,$$
solves \eqref{dominantpart-PDE-k-true} up to terms of order $\vs^{-2}$.
Thus, the particular solution $\xi_{\inn}$ that we will use is $\xi_{\inn}=\xi_0 + \varphi$
where $\varphi$ is a function that is ``small'' when $|s| \to \infty$.
This result is contained in Proposition \ref{propk} of Section \ref{sec:diffinner}.

All these considerations lead to the following result.
\begin{theorem}\label{thmdiffinner}
Consider the difference:
$\Delta\psiin=\psiin^\uns-\psiin^\sta$ defined in $\in\Ein\times\Tout$. It can be written of the form
$$
\Delta \psiin(\vs,\theta) = \Pin(\vs,\theta) \kintilde(\xi_{\inn}(\vs,\theta),\theta)
$$
where the function $\kintilde(\tau)$ is a $2\pi-$periodic function, $\Pin$ is a particular
solution of the PDE~\eqref{PDE-difference-intro} and $\xi_\inn$ is a solution of the homogeneous PDE
equation~\eqref{PDE-k-intro} such that $(\xi_\inn(\vs,\theta),\theta)$ is injective in $\Ein\times\Tout$. They are of the form
\begin{align*}
\kintilde(\tau)&=\sum_{l<0}\Upsilon_\inn^{[l]}e^{il\tau}, \qquad
\Pin(\vs,\theta) = s^{2/d} (1+\Pin_1(\vs,\theta)), \\
\xi_\inn(\vs,\theta) &= \theta+\coef^{-1}\alpha\vs+\coef^{-1}(c+\alpha L_0)\log\vs+\varphi(\vs,\theta)
\end{align*}
with $\Upsilon_\inn^{[l]}$, $\Pin_1$ and $\varphi$ satisfying that there exists $M>0$ such that 
$\left|\Upsilon_\inn^{[l]}\right|\leq M$ and, for $(\vs,\theta)\in\Ein\times\Tout$:
\begin{equation}\label{P1inntheorem}
|\varphi(\vs,\theta)|,\,|\partial_\vs \varphi(\vs,\theta)|\leq \frac{M}{|\vs|},\qquad 
|\Pinsmall(\vs,\theta)|\leq \frac{M}{|\vs|},\qquad |\partial_\vs\Pinsmall(\vs,\theta)|\leq \frac{M}{|\vs|^2}.
\end{equation}
\end{theorem}
The proof of this Theorem is given in Section \ref{sec:diffinner}. Let us to notice here that, as an straightforward consequence
of this result:
\begin{equation}\label{expression-deltapsiinn}
\Delta \psiin(\vs,\theta)=s^{2/\coef} (1+\Pin_1(\vs,\theta))\sum_{l<0}\Upsilon_\inn^{[l]}e^{il\xi_\inn(\vs,\theta)}.
\end{equation}
\begin{corollary}\label{bounddiffinnerexp}
There exists a constant $M$ such that, if $\theta \in \mathbb{S}^1$ and $\vs\in \Ein$, then
$$
\left | \Delta \psiin(\vs,\theta)\right |,\, \left | \partial_\vs \Delta \psiin(\vs,\theta)\right |  \leq M |\vs |^{2/\coef} e^{\im \xi_\inn(\vs,\theta)}.
$$
\end{corollary}
\begin{proof}
We only need to note that $\im \xi_\inn(\vs,\theta) <0$ for $(\vs,\theta)\in \Ein\times\mathbb{S}^1$ if $\distin$ is sufficiently large. 
Then we use formula~\eqref{expression-deltapsiinn} that $\left|\Upsilon_\inn^{[l]}\right|\leq M$ and bounds~\eqref{P1inntheorem} as well.
\end{proof}
\subsection{Study of the matching errors $\psi_1^\uns$ and $\psi_1^\sta$}\label{subsec:intromatching}
In this section we shall show that the functions $\psiin^\uns$ and $\psiin^\sta$ found in Theorem \ref{thminner}
approximate the functions $\psi^\uns$ and $\psi^\sta$, defined in Theorem \ref{thmoutloc-innervariables}, in some
complex subdomains of $\Din{\uns}$ and $\Din{\sta}$ respectively, which, when written in the $(\vu,\theta)$ variables,
correspond to domains close to the singularity $ i\pi/(2\coef)$. Let us first define these domains.
Recall that, if $\us= \uns,\sta$, $\psi^{\us}(\vs,\theta)=\delta^2r_1^{\us}(\vu(\vs),\theta)$ (see~\eqref{change-s}
for the definition of the change $\vu(\vs)$) and that $r_1^\us$ are defined in the domains $\DoutT{\us}\times\Tout$
(see~\eqref{defDoutT} for the definition of these domains), where $\dist$ satisfies condition~\eqref{conddist-inner}
and $\beta$ is some fixed constant.
Take $\beta_1,\beta_2$ two constants independent of $\delta$ and $\param$, such that:
$$
0<\beta_1<\beta<\beta_2<\pi/2.
$$
Fix also a constant $\gamma\in(0,1)$. We define the points $\vu_j\in\mathbb{C}$, $j=1,2$ to be those satisfying (see Figure~\ref{figDmatch}):
\begin{equation}\label{defuj}
\begin{aligned}
\im \vu_j=-\tan\beta_j\re &\vu_j+\frac{\pi}{2\coef}-\delta\dist,\quad
\re \vu_1<0,\; \re\vu_2>0\\
&\left|\vu_j-i\left(\frac{\pi}{2\coef}-\delta\dist\right)\right|=\delta^\gamma.
\end{aligned}
\end{equation}
\begin{figure}
	\centering
	\begin{subfigure}[b]{0.45\textwidth}
	  \centering
	  \includegraphics[width=6.1cm]{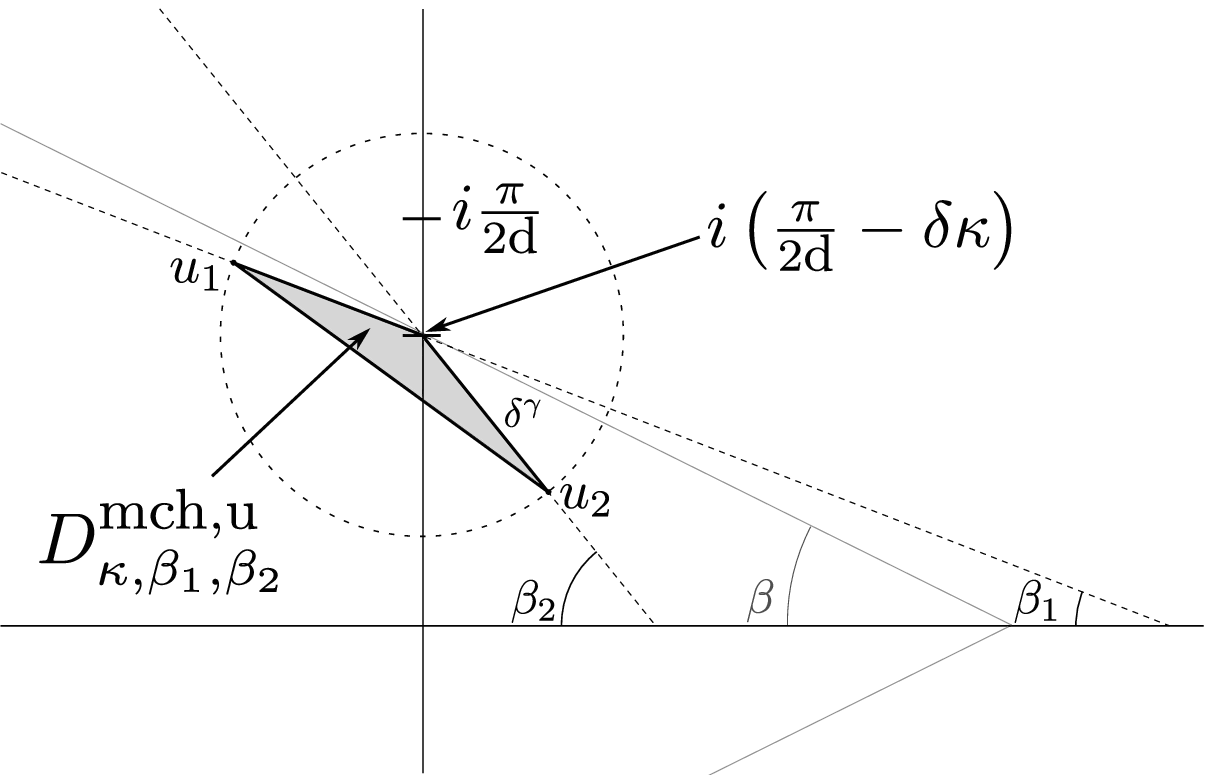}
	  \caption{The domain $\Dmchout{\uns}$.}
	\end{subfigure}
	\begin{subfigure}[b]{0.45\textwidth}
	  \centering
	  \includegraphics[width=4.5cm]{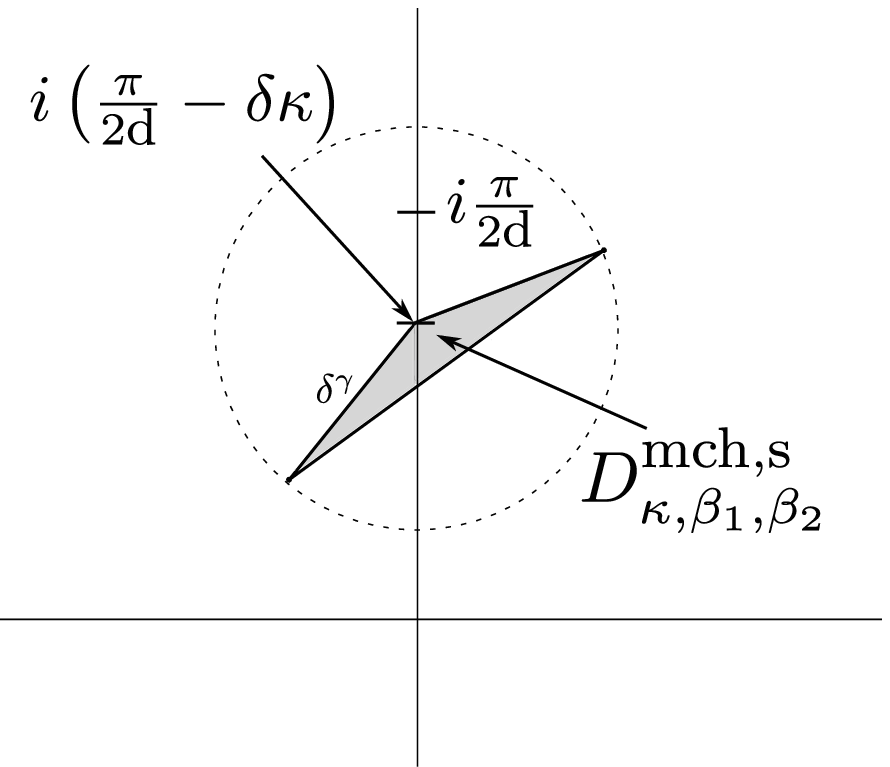}
	  \caption{The domain $\Dmchout{\sta}$.}
	\end{subfigure}
	\caption{The domains $\Dmchout{\uns}$ and $\Dmchout{\sta}$.}\label{figDmatch}
\end{figure}
Then we define the following domain (see Figure~\ref{figDmatch}):
\begin{align*}
\Dmchout{\uns}=\Big\{\vu\in\mathbb{C}\,:\, &\im u\leq -\tan\beta_1\re \vu+\frac{\pi}{2\coef}-\delta\dist,
\\ & \im \vu\leq -\tan\beta_2\re \vu+\frac{\pi}{2\coef}-\delta\dist,\\
&\im \vu\geq \im u_1-\tan\left(\frac{\beta_1+\beta_2}{2}\right)(\re \vu-\re \vu_1)\Big\}.
\end{align*} 
Note that $\Dmchout{\uns}$ is a triangular domain, having its vertices at $\vu_1$, $\vu_2$ and $i(\pi/(2\coef)-\delta\dist)$. We also define:
$$\Dmchout{\sta}=\{\vu\in\mathbb{C}\,:\, -\overline{\vu}\in\Dmchout{\uns}\}.$$
One can easily see  that, taking $\distin=\dist/2$, one has (see Figures \ref{figDoutunsDinuns} and \ref{figDmatch}):
$$\Dmchout{\uns}\subset\DoutT{\uns}\subset\Dinu{\uns},\quad \Dmchout{\sta}\subset\DoutT{\sta}\subset\Dinu{\sta},$$
with $\DoutT{\uns}$ defined in~\eqref{defDoutT},  $\DoutT{\sta}=-\DoutT{\uns}$ and $\Dinu{\uns}$ and $\Dinu{\sta}$ were defined in \eqref{defdin-outervars}.

Finally we define the following domains in terms of the inner variable $\vs$:
\begin{equation}\label{Dmchindef}
\Dmchin{\us}=\{\vs\in\mathbb{C}\,:\,i\frac{\pi}{2\coef}+\delta\frac{\vs}{\coef}\in\Dmchout{\us}\},\qquad \us = \uns,\sta.
\end{equation}
One also has that for $\distin=\dist/2$ and taking $\delta$ sufficiently small:
$$\Dmchin{\uns}\subset\Din{\uns},\qquad\Dmchin{\sta}\subset\Din{\sta},$$
where $\Din{\uns}$ and $\Din{\sta}$ were defined in \eqref{defdinus-inner2D}.

We will also denote:
\begin{equation}\label{etsj-inner2D}
 \vs_j=\frac{\coef}{\delta}\left (\vu_j-i\frac{\pi}{2\coef}\right ),\quad j=1,2,
\end{equation}
with $\vu_1,\vu_2$ defined in~\eqref{defuj}.
It is clear that, there exist constants $K_1$, $K_2$ such that:
\begin{equation}\label{fitasj-inner2D}
 K_1\delta^{\gamma-1}\leq|\vs_j|\leq K_2\delta^{\gamma-1},\qquad j=1,2,
\end{equation}
and that for all $\vs\in\Dmchin{\us}$, $\us=\uns,\sta$, we have:
\begin{equation}\label{upperlowerboundss}
\dist\coef\cos\beta_1\leq|\vs|\leq K_2\delta^{\gamma-1}.
\end{equation}

Next goal is to see how well the functions $\psiin^{\uns,\sta}$ approximate $\psi^{\uns,\sta}$ in the matching domains $\Dmchin{\uns,\sta}\times\Tout$.
To that aim, we recall the definition of the matching error:
\begin{equation}\label{defpsi1us-intro}
 \psi_1^\uns(\vs,\theta):=\psi^\uns(\vs,\theta)-\psiin^\uns(\vs,\theta),\qquad\psi_1^\sta(\vs,\theta):=\psi^\sta(\vs,\theta)-\psiin^\sta(\vs,\theta).
\end{equation}
We stress that Theorems~\ref{thmoutloc-innervariables} and~\ref{thminner} yield directly the existence of $\psi_1^{\uns,\sta}$.
On the one hand, $\psi^\uns(\vs,\theta)$ is defined for $\vs\in\DoutTinvars{\uns}$ (see~\eqref{defDoutTuns-innervars} for its definition) and
$\psiin^{\uns}(\vs,\theta)$ is defined for $\vs\in\Din{\uns}$ (see \eqref{defdinus-inner2D}).
Then, since
$$
\Dmchin{\uns}\subset\DoutTinvars{\uns}\subset\Din{\uns},
$$
one has that
$\psi_1^\uns=\psi^\uns-\psiin^\uns$ is defined in $\Dmchin{\uns}$. On the other hand,
Theorems~\ref{thmoutloc-innervariables} and~\ref{thminner} also provide us with a non-sharp upper bound for these functions
of order $\mathcal{O}(|s|^{-3})$.
In the following result we prove that, restricting $\psi_1^{\uns}$ and $\psi_1^\sta$ to the smaller
domains $\Dmchin{\uns}$ and $\Dmchin{\uns}$ respectively, we can get better upper bounds for them.

\begin{theorem}\label{thmmatching}
Let $\psi_1^\uns$ and $\psi_1^\sta$ be defined in \eqref{defpsi1us-intro}. There exists $M>0$ such that:
$$|\psi_1^{\us}(\vs,\theta)|\leq M\delta^{1-\gamma}|\vs|^{-2}\qquad \vs\in\Dmchin{\us}\qquad \us =\uns,\sta
$$
\end{theorem}
The proof of this Theorem can be found in Section~\ref{sec:matching}.

\subsection{An asymptotic formula for the difference $\Delta=r_1^\uns-r_1^\sta$}\label{subsec:introformuladiff}
We shall use the information obtained in the previous subsections to find an asymptotic formula
for $\Delta=r_1^\uns-r_1^\sta$ by means of the $\Delta\psi_\inn=\psi_\inn^{\uns}-\psi_\inn^{\sta}$. As we pointed out in~\eqref{expression-delta},
\begin{equation}\label{expression-deltabis}
\Delta(\vu,\theta)=\cosh^{2/\coef}(\coef\vu)(1+P_1(\vu,\theta))\sum_{l\in\mathbb{Z}}\Upsilon^{[l]}e^{il\xi(\vu,\theta)},
\end{equation}
with $P_1$ and $\xi$ the functions given in Theorem~\ref{thmdifpartsolinjective} and $\Upsilon^{[l]}$, the Fourier
coefficients of the function $\tilde k(\tau)$, which are unknown.
They depend on $\delta$ and $\param$.
Recall that in the dissipative case, we assume that $\param$ lies on one of the curves $\param_*(\delta)$ given by Theorem~\ref{thmktilde0}.

We shall look for a first asymptotic order of
$\Delta$ of the form:
\begin{equation}\label{defDelta0inn}
\Delta_0(\vu,\theta)= \cosh^{2/\coef}(\coef\vu)(1+P_1(\vu,\theta))\left (\Upsilon^{[0]}
+ \sum_{l\neq 0} \Upsilon_0^{[l]}e^{il\xi(\vu,\theta)} \right ),
\end{equation}
for certain $\Upsilon_0^{[l]}$ to be determined. Some remarks about the choice of $\Delta_0$ we have made:
\begin{itemize}
\item 
We stress that $\Upsilon^{[0]}$ is the same coefficient (depending on $\delta$ and  $\param$) appearing
in~\eqref{expression-delta}. From Theorem~\ref{thmktilde0} we know that in the conservative case the coefficient
$\Upsilon^{[0]}$ is zero, while in the dissipative case it can be made zero (or exponentially small) with the right choice of the parameter
$\param$ as a function of $\delta$ (see expression~\eqref{Upsilon0expsmall}).
\item
Notice that by Lemma~\ref{lemUpsilonlexpsmall} the coefficients $\Upsilon^{[l]}$ are exponentially small with respect to $\delta$.
This result hints that, disregarding the coefficient $\Upsilon^{[0]}$, the dominant term of $\Delta(\vu,\theta)$
is determined essentially by the coefficients $\Upsilon^{[1]}$ and $\Upsilon^{[-1]}$ in
expression~\eqref{expression-deltabis}. 
As we will see after we choose the appropriate $\Upsilon_0^{[l]}$, the same will happen to $\Delta_0$ which will have 
the coefficients $\Upsilon_0^{[1]}$ and $\Upsilon_0^{[-1]}$ as dominant terms in \eqref{defDelta0inn}.
However, just for technical reasons, we have preferred to include all the Fourier coefficients in the asymptotic first order
$\Delta_0$.
\item 
It is proven in Lemma~\ref{lemUpsilonlexpsmall} that the coefficients 
$\Upsilon^{[l]}=\mathcal{O}\left(\delta^{-2-2/\coef}e^{-\frac{\alpha\pi}{2\coef\delta}}\right)$ 
when $l\neq 0$. In fact smaller if $|l|\ge 2$.
For this reason,  when $l\neq 0$, we expect
$\Upsilon_0^{[l]}=\mathcal{O}\left(\delta^{-2-2/\coef}e^{-\frac{\alpha\pi}{2\coef\delta}}\right)$ at least.
\item 
We only need to deal with $\Upsilon_0^{[l]}$ with $l<0$ and then define $\Upsilon_0^{[l]}=\overline{\Upsilon_0^{[-l]}}$, for $l>0$, 
to make $\Delta_0$ a real analytic function.
\item 
As we pointed out in item 3 in Section~\ref{subsec:previousresults}, to obtain exponentially small bounds for $\Delta(u,\theta)-\Delta _0(u,\theta)$ 
for real values of $u$,
we will need to evaluate our different functions at $\vu_+=i(\pi/(2\coef)-\dist\delta)$. Roughly speaking,
what we need is to guarantee that, with the appropriate choice of $\Upsilon^{[l]}_0$,
$$
\Delta(\vu_+,\theta) - \Delta_0(\vu_+,\theta) = o\left (\Delta_0(\vu_+,\theta)\right).
$$
\end{itemize}

In what follows we will give some heuristic ideas  to explain how to choose the coefficients $\Upsilon_0^{[l]}$ with $l<0$.
For that we assume that the function $P_1\equiv 0$ and that the function
$\xi$ in~\eqref{defxi} is
\begin{equation}\label{defxisimple}
\xi(\vu,\theta) = \theta + \delta^{-1} \alpha \vu + \coef^{-1} \big (c + \alpha L_0\big )\log \cosh(\coef \vu) +L(\vu).
\end{equation}
That is the function $\chi \equiv 0$.
\begin{remark}
We want emphasize here that there is no attempt to be rigorous but to give some intuition about why our choice of
$\Upsilon_0^{[l]}$ could be the good one.
Later, in the proof of Proposition~\ref{properrorUpsilonsinner} in Section~\ref{sec:asymptotic}, we will see that the choice
of $\Upsilon_0^{[l]}$ we shall make provides indeed an asymptotic first order $\Delta_0(\vu,\theta)$ of the difference $\Delta(\vu,\theta)$.
\end{remark}
Note that $\vu_+\in\Dmchout{\uns}\cap\Dmchout{\sta}$.
In this region, as a consequence of Theorem~\ref{thmmatching}, the functions $r_1^\uns$ and $r_1^\sta$ are well
approximated by $\psiin^\uns$ and $\psiin^\sta$, respectively, the solutions of the inner equation \eqref{PDE-inner} given by Theorem~\ref{thminner}.
More precisely:
$$
r_1^{\uns,\sta} (\vu,\theta)\approx\delta^{-2}\psiin^{\uns,\sta}(s(\vu),\theta), \qquad \vu\in\Dmchout{\uns}\cap\Dmchout{\sta},
$$
where the change $\vs(\vu) =1/\big [\delta \tanh (\coef \vu)\big ]$ is defined in~\eqref{change-s}.
Since $\Delta(\vu,\theta)=r_1^{\uns}(\vu,\theta)-r_1^{\sta}(\vu,\theta)$,
one expects that, for  $\vu\in\Dmchout{\uns}\cap\Dmchout{\sta}$
\begin{equation}\label{DeltaApproxDeltain}
\Delta(\vu,\theta)\approx \Delta_\inn(\vu,\theta):=\delta^{-2}\psiin^\uns(s(\vu),\theta)-\delta^{-2}\psiin^\sta(s(\vu),\theta).
\end{equation}
Now, by expression~\eqref{expression-deltapsiinn} of $\Delta \psiin=\psiin^\uns-\psiin^\sta$ we have:
\begin{equation}\label{expression-deltain}
\Delta_\inn(\vu,\theta)=\delta^{-2} \vs^{2/\coef} (\vu)
\left(1+\Pinsmall\left(\vs(\vu),\theta\right)\right)\sum_{l<0}\Upsilon_\inn^{[l]}e^{il\xiin(\vs(\vu),\theta)},
\end{equation}
where $\Upsilon_\inn^{[l]}$ are constants independent of $\delta$ and $\param$ and $\xi_\inn$ is defined in Theorem~\ref{thmdiffinner}.
In this heuristic approach we make the same simplifications for $\Delta_\inn$ as the ones for $\Delta$. 
That is, we will assume that $\Pinsmall\equiv 0$ and that
\begin{equation}\label{defxiinnsimple}
\xi_\inn(\vs(\vu),\theta)=\theta+\coef^{-1}\alpha \vs(\vu)+\coef^{-1}(c+\alpha L_0)\log \left( \frac{\cosh(\coef \vu)}{\delta\sinh(\coef \vu)}\right ).
\end{equation}
A necessary condition to guarantee that approximation~\eqref{DeltaApproxDeltain} holds true at $\vu=\vu_+$ is:
$$
\delta^{-2} \vs^{2/\coef} (\vu_+)  \Upsilon_\inn^{[l]}e^{il \xiin(\vs(\vu_+),\theta)} \approx
\cosh^{2/\coef} (\coef \vu_+) \Upsilon^{[l]} e^{il\xi(\vu_+,\theta)}, \qquad l<0
$$
or equivalently
\begin{equation}\label{firstUpsilon0}
\Upsilon^{[l]} \approx \delta^{-2} \vs^{2/\coef} (\vu_+)  \cosh^{-2/\coef}(\coef \vu_+) \Upsilon_\inn^{[l]}
e^{il \big (\xiin(\vs(\vu_+),\theta)-\xi(\vu_+,\theta)\big )}, \qquad l<0.
\end{equation}
Now we estimate the right-hand side of~\eqref{firstUpsilon0}. First we notice that by Remark~\ref{rmkchange-s}
\begin{equation}\label{su+}
\vs(\vu_+) = -i\coef \dist + \mathcal{O}(\delta^2 \dist^3).
\end{equation}
Moreover subtracting expressions~\eqref{defxiinnsimple} and~\eqref{defxisimple} of $\xi_\inn(\vs(\vu),\theta)$ and $\xi(\vu,\theta)$ respectively, 
evaluated at $\vu=\vu_+$ and using~\eqref{su+} for $\vs(\vu_+)$, we have that
$$
\xi_\inn(\vs(\vu_+),\theta) - \xi(\vu_+,\theta)= - i\frac{\alpha  \pi}{2 \coef \delta } -\coef^{-1} (c+\alpha L_0) \log( \delta \sinh(\coef \vu_+))-
\alpha L(\vu_+) + \mathcal{O}(\delta^ 2 \kappa^3)
$$
Secondly we observe that, on the one hand,
\begin{equation}\label{coshsinhu+}
\begin{aligned}
\cosh (\coef \vu_+) &= \coef \delta  \dist + \mathcal{O}(\delta^3 \dist^3) \\
\log\sinh(\coef\vu_+)&=\log|\sinh(\coef\vu_+)|+i\frac{\pi}{2}=i\frac{\pi}{2}+\mathcal{O}\left(\delta^2\dist^2\right),
\end{aligned}
\end{equation}
and on the other hand, by Theorem~\ref{thmdifpartsolinjective}, the following limit is well-defined:
\begin{equation}\label{defL+}
L_+:=\lim_{\vu\to i\frac{\pi}{2\coef}}L(\vu).
\end{equation}
Therefore, since by \eqref{boundLchi-diff} $|L'(u)|$ is bounded, from~\eqref{firstUpsilon0} and dismissing the high order terms:
$$
\Upsilon^{[l]} \approx \delta^{-2 -2/\coef} (-i)^{2/\coef} \Upsilon^{[l]}_\inn
e^{\coef^{-1} (c+\alpha L_0)\left (-il\log  \delta +l\frac{\pi}{2} \right ) -il\alpha L_+}e^{l\frac{\alpha \pi}{2\coef \delta}},\qquad l<0.
$$
This concludes the heuristic approach and now we define:
\begin{equation}\label{defUpsilon0lneg}
\Upsilon_0^{[l]}=:\delta^{-2 -2/\coef} (-i)^{2/\coef} \Upsilon^{[l]}_\inn
e^{\coef^{-1} (c+\alpha L_0)\left (-il\log  \delta +l\frac{\pi}{2} \right ) -il\alpha L_+}e^{l\frac{\alpha \pi}{2\coef \delta}},\qquad l<0.
\end{equation}
Since $\Delta(\vu,\theta)$ is real analytic, $\Upsilon^{[l]}=\overline{\Upsilon^{[-l]}}$. Thus, we define:
\begin{equation}\label{defUpsilon0lpos}
\Upsilon_0^{[l]}:=\overline{\Upsilon_0^{[-l]}}=\delta^{-2 -2/\coef} i^{2/\coef} \overline{\Upsilon^{[-l]}_\inn}
e^{\coef^{-1} (c+\alpha L_0)\left (-il\log  \delta -l\frac{\pi}{2} \right ) -il\alpha L_+}e^{-l\frac{\alpha \pi}{2\coef \delta}}\qquad l>0.
\end{equation}

Of course, to prove that $\Delta_0$ in~\eqref{defDelta0inn} with the choice of $\Upsilon_0^{[l]}$ provided in~\eqref{defUpsilon0lneg} and~\eqref{defUpsilon0lpos}
is the first order of $\Delta$, we need to see that:
$\Delta_1 :=\Delta-\Delta_0$ is smaller than $\Delta_0$. To this end, let us  write
\begin{equation}\label{defDelta1}
\Delta_1(\vu,\theta)=
\cosh^{2/\coef}(\coef\vu)(1+P_1(\vu,\theta))\sum_{|l|\geq 1}\left (\Upsilon^{[l]}-\Upsilon^{[l]}_0\right )e^{il\xi(\vu,\theta)}.
\end{equation}
By Lemma \ref{lemUpsilonlexpsmall} and formulae~\eqref{defUpsilon0lneg} and~\eqref{defUpsilon0lpos} is clear that, generically, 
the terms involving $\Upsilon^{[l]},\Upsilon^{[l]}_0$
with $|l|\geq2$ are smaller than $\Delta_0$ (observe that $\Delta_0$ contains the biggest terms $\Upsilon_0^{[\pm 1]}$). The following result states that the terms $\Upsilon^{[\pm1]}-\Upsilon_0^{[\pm1]}$ are also small. Its proof can be found in Section \ref{sec:asymptotic}.

\begin{proposition}\label{properrorUpsilonsinner}
Let $\dist=\dist_0\log(1/\delta)$, with $\dist_0>0$ any constant such that $1-\gamma>\alpha\dist_0$. Let $\Upsilon_0^{[\pm1]}$ be defined
as~\eqref{defUpsilon0lneg} with $l=-1$ and~\eqref{defUpsilon0lpos} with $l=1$ respectively. There exists a constant $M$ such that:
$$
\left|\Upsilon^{[\pm1]}-\Upsilon_0^{[\pm1]}\right|\leq \frac{M}{\dist}\delta^{-2-\frac{2}{\coef}}e^{-\frac{\alpha\pi}{2\coef\delta}},
$$
where we assume that, in the dissipative case, $\param=\param_*(\delta)$ is one of the curves defined in Theorem \ref{thmktilde0}.
Recall that $\coef=1$ in the conservative case.
\end{proposition}

Using this result and formulae \eqref{defUpsilon0lneg} and \eqref{defUpsilon0lpos} of $\Upsilon_0^{[l]}$ we can prove:
\begin{theorem}\label{mainthm-inner}
Define:
$$
\vt(\vu,\delta)=\delta^{-1}\alpha\vu+\coef^{-1}(c+\alpha L_0)\left[\log\cosh(\coef\vu)-\log\delta\right]+\alpha L(\vu),
$$
where $L_0$ and $L(\vu)$ are given in Theorem \ref{thmdifpartsolinjective}. 
Define:
$$
\mathcal{C}^*=\mathcal{C}_1^*+i\mathcal{C}_2^*:=2(-i)^{\frac{2}{\coef}}\Upsilon_\inn^{[-1]}e^{-\coef^{-1}(c+\alpha L_0)\frac{\pi}{2}+i\alpha L_+}.
$$
where $\Upsilon_\inn^{[-1]}$ appears in Theorem \ref{thmdiffinner} and $L$ is  defined in~\eqref{defL+}.
Then there exist $T_0>0$ and $\delta_0>0$ such that for all $\vu\in[-T_0,T_0]$, $\theta\in\mathbb{S}^1$ and $0<\delta<\delta_0$ the following holds:
\begin{align*}
\Delta(\vu,\theta)=&\cosh^{2/\coef}(\coef\vu)\Upsilon^{[0]}\left(1+\mathcal{O}(\delta)\right)\\
&+\delta^{-2-\frac{2}{\coef}}\cosh^{2/\coef}(\coef\vu)e^{-\frac{\alpha\pi}{2\coef\delta}}
\Bigg[\mathcal{C}_1^*\cos\Big(\theta+\vt(\vu,\delta)\Big)\\
&+\mathcal{C}_2^*\sin\Big(\theta+\vt(\vu,\delta)\Big)+\mathcal{O}\left(\frac{1}{\log(1/\delta)}\right)\Bigg],
\end{align*}
where, in the dissipative case, $\param=\param_*(\delta)$ is one of the curves defined in Theorem \ref{thmktilde0} meanwhile
$\Upsilon^{[0]}=0$ and
$\coef=1$ in the conservative case.
\end{theorem}
\begin{proof}
Let $\dist=\dist_0\log(1/\delta)$, with $\dist_0>0$ any constant such that $1-\gamma>\alpha\dist_0$.
On the one hand, using Lemma~\ref{lemUpsilonlexpsmall} to bound $|\Upsilon^{[l]}|$ for $l\geq2$, formulae~\eqref{defUpsilon0lneg} and~\eqref{defUpsilon0lpos}
for $\Upsilon^{[l]}_0$, Proposition~\ref{properrorUpsilonsinner}
to bound $|\Upsilon^{[\pm1]}-\Upsilon_0^{[\pm1]}|$ and the fact that $\xi(\vu,\theta)\in\mathbb{R}$ for
$(\vu,\theta)\in[-T_0,T_0]\times\mathbb{S}^1$, it is easy to see that, from expression~\eqref{defDelta1} of $\Delta_1 = \Delta-\Delta_0$,
$$|\Delta_1(\vu,\theta)|\leq K\cosh^{2/\coef}(\coef\vu)|1+P_1(\vu,\theta)|\frac{\delta^{-2-2/\coef}}{\dist}e^{-\frac{\alpha\pi}{2\coef\delta}}.$$
Since, by bound~\eqref{boundp1-diff} of $P_1$, we have:
\begin{equation}\label{boundP1-ureal}
|P_1(\vu,\theta)|\leq \frac{K\delta}{|\cosh(\coef\vu)|}\leq K\delta, \qquad \text{for }\vu\in[-T_0,T_0]
\end{equation}
and this yields (renaming $K$):
\begin{eqnarray}\label{boundDelta1-inner}
 |\Delta_1(\vu,\theta)|\leq K\cosh^{2/\coef}(\coef\vu)\frac{\delta^{-2-2/\coef}}{\log(1/\delta)}e^{-\frac{\alpha\pi}{2\coef\delta}}.
\end{eqnarray}
On the other hand, by definition~\eqref{defDelta0inn} of $\Delta_0$ and since $\Upsilon^{[l]}_0 = \overline{\Upsilon^{[-l]}_0}$,
$$
\Delta_0(\vu,\theta)=\cosh^{2/\coef}(\coef\vu)(1+P_1(\vu,\theta))\left[\Upsilon^{[0]}+2 \re \left (\Upsilon_0^{[-1]}e^{-i\xi(\vu,\theta)}\right)
+\mathcal{O}\left (\delta^{-2 -2/\coef} e^{-\frac{\alpha \pi}{\coef \delta}}\right )\right ],
$$
where $\Upsilon_0^{[-1]}$ is given in~\eqref{defUpsilon0lneg} with $l=-1$,
and $\xi$ is defined in \eqref{defxi}. Then one has:
$$
\Upsilon_0^{[-1]}e^{-i\xi(\vu,\theta)}=\delta^{-2-\frac{2}{\coef}}
e^{-\frac{\alpha\pi}{2\coef\delta}}\frac{\mathcal{C}^*}{2}e^{-i(\theta+\vt(\vu,\delta)+\chi(\vu,\theta))},
$$
with:
$$
\mathcal{C}^*=2(-i)^{\frac{2}{\coef}}\Upsilon_\inn^{[-1]}e^{-\coef^{-1}(c+\alpha L_0)\frac{\pi}{2}+i\alpha L_+}.
$$
Using Theorem~\ref{thmdifpartsolinjective}, one has that, if $\vu\in[-T_0,T_0]$ then
$|\chi(\vu,\theta)|\leq K\delta$, so that:
$$
\Upsilon_0^{[-1]}e^{-i\xi(\vu,\theta)}=\delta^{-2-\frac{2}{\coef}}
e^{-\frac{\alpha\pi}{2\coef\delta}}\frac{\mathcal{C}^*}{2}e^{-i(\theta+\vt(\vu,\delta))}(1+\mathcal{O}(\delta)).
$$
Then, using also bound~\eqref{boundP1-ureal} and the fact that $\Upsilon^{[1]}=\overline{\Upsilon^{-[1]}}$, we obtain:
\begin{align}\label{asymptotic-Delta0inner}
 \Delta_0&(\vu,\theta)=\cosh^{2/\coef}(\coef\vu)\Upsilon^{[0]}(1+\mathcal{O}(\delta))\\
&+\cosh^{2/\coef}(\coef\vu)\delta^{-2-\frac{2}{\coef}}e^{-\frac{\alpha\pi}{2\coef\delta}}
\left[\frac{\overline{\mathcal{C}^*}}{2}e^{i(\theta+\vt(\vu,\delta))}+\frac{\mathcal{C}^*}{2}e^{-i(\theta+\vt(\vu,\delta))}\right](1+\mathcal{O}(\delta)).\nonumber
\end{align}
Finally we only have to note that:
$$
\frac{\overline{\mathcal{C}^*}}{2}e^{i(\theta+\vt(\vu,\delta))}+\frac{\mathcal{C}^*}{2}e^{-i(\theta+\vt(\vu,\delta))}=\re \mathcal{C}^*\cos(\theta+\vt(\vu,\delta))+\im \mathcal{C}^*\sin(\theta+\vt(\vu,\delta)),
$$
so that using bound~\eqref{boundDelta1-inner} of $\Delta_1$, expression~\eqref{asymptotic-Delta0inner} of $\Delta_0$ and the fact that
$\Delta=\Delta_0 + \Delta_1$ we obtain the claim of the theorem.
In the conservative case we take into account that $\coef=1$ and $\Upsilon^{[0]}=0$ by Theorem \ref{thmktilde0}.
\end{proof}

\begin{remark}
Theorem \ref{mainthm-inner} yields straightforwardly Theorem~\ref{mainthm-inner-intro}.
\end{remark}

The remaining part of this work is devoted to provide the proofs of the results in this section.
We present first, in Section~\ref{sec:asymptotic}, the proof related to the exponentially small behavior of $\Delta(\vu,\theta)$ 
in Proposition \ref{properrorUpsilonsinner}, assuming
that all the previous results in the present section hold true. After that we deal with the results related to the inner equation in Section~\ref{sec:innergeneral}.
Indeed, first, in Section~\ref{sec:inner}, we deal with the existence and properties of the solutions $\psiin^{\uns,\sta}$ of the inner equation
and secondly, in Section~\ref{sec:diffinner}, we prove the
asymptotic expression for the difference $\Delta \psiin=\psiin^{\uns}-\psiin^{\sta}$ stated in Theorem~\ref{thmdiffinner}.
Then, in Section~\ref{sec:matching} we measure the matching errors $\psi_1^{\uns,\sta}(\vs,\theta)$ for $\vs$ belonging to the matching domains $\vs\in\Dmchin{\uns,\sta}$.

All the constants that appear in the statements of the following results might depend on $\delta^*$, $\delta_0$, $\param^*$, $\dist^*$ and $\dist_0$
but never on $\delta$, $\param$ and $\dist$. 
We assume that $\delta^*,\delta_0$ and $\param^*$ are sufficiently small,
and $\dist^*, \dist_0$ are big enough
satisfying condition~\eqref{conddist-inner}. These conventions are valid for all the sections of this work.
As in the previous work~\cite{BCS16a}, we shall skip the proofs that do not provide any interesting insight.
For these proofs we refer the reader to \cite{CastejonPhDThesis}.

\section{Exponentially small behavior. Proposition~\ref{properrorUpsilonsinner}}\label{sec:asymptotic}
This section is devoted to prove the asymptotic for $\Upsilon^{[\pm 1]}$ given in Proposition~\ref{properrorUpsilonsinner}.
To prove this result, we will assume all the results in Sections~\ref{preliminary}, \ref{subsec-introinner} and~\ref{subsec:intromatching}.

We first begin with a result which relates the functions $\xi(\vu,\theta)$ and $\xi_\inn(\vs(\vu),\theta)$, given in Theorems \ref{thmdifpartsolinjective}
and~\ref{thmdiffinner} respectively, when $\vu$ is close to the singularity $i\pi/(2\coef)$.
\begin{lemma}\label{lemxixiin}
Let $\dist$ be sufficiently large, $L_+$ the constant in~\eqref{defL+},
$\vu_+=i\left(\frac{\pi}{2\coef}-\delta\dist\right)$
and $\vs(\vu)=[\delta \hetz(\vu)]^{-1}$ the function defined in \eqref{change-s}.
There exists two functions $\varrho(\theta)$ and $\eta(\theta)$ and a constant $M$ satisfying:
$$\sup_{\theta\in\mathbb{S}^1}|\varrho(\theta)|,\,  \sup_{\theta\in\mathbb{S}^1}|\eta(\theta)|\leq \frac{M}{\dist},$$
such that:
\begin{enumerate}
\item The function $\xi(\vu,\theta)$ in Theorem~\ref{thmdifpartsolinjective} satisfies
$$
\xi(\vu_+,\theta) -\theta = i\frac{\alpha \pi}{2 \coef \delta} - i\alpha \dist +\coef^{-1} (c+\alpha L_0) \log (\delta \dist \coef) +\alpha L_+
+\varrho(\theta).
$$
\item The function $\xi_\inn(\vs(\vu),\theta)$, given in Theorem~\ref{thmdiffinner} is related to $\xi(\vu,\theta)$ by:
$$
\xiin(\vs(\vu_+),\theta)=\xi(\vu_+,\theta)-i\frac{\alpha\pi}{2\coef\delta}-\coef^{-1}(c+\alpha L_0)\left(\log\delta+i\frac{\pi}{2}\right)-\alpha L_++\eta(\theta).
$$
\end{enumerate}
\end{lemma}
\begin{proof}
By definition of $\xi$ in~\eqref{defxi} and remaining that $u_+= i(\frac{\pi}{2d}-|delta \kappa)$,
$$
\xi(\vu_+,\theta)-\theta = i\frac{\alpha \pi}{2 \coef \delta} - i \alpha \kappa +\coef^{-1} (c + \alpha L_0) \log (\cosh(\coef \vu_+)) + \alpha L(\vu_+)
+\chi(\vu_+,\theta).
$$
By Theorem \ref{thmdifpartsolinjective}, for all $\vu\in\Doutinter$, $|L'(\vu)|\leq K$ and $|\chi(\vu,\theta)|\leq K \delta |\cosh (\coef \vu)|^{-1}$.
Then, by definition \eqref{defL+} of $L_+$ and using expression~\eqref{coshsinhu+} of $\cosh(\coef \vu_+)$:
\begin{equation}\label{boundLL+}
|L(\vu_+)-L_+|\leq K\delta\dist,\qquad |\chi(\vu,\theta)|\leq \frac{K}{\dist}.
\end{equation}
The first item follows using again~\eqref{coshsinhu+}.

Now we prove the second item.
By using definition of $\xi_\inn $ in Theorem~\ref{thmdiffinner} as well as
that $\vs(\vu) = 1/[\delta \tanh(\coef \vu)]$ one obtains
\begin{equation}\label{difxixinn}
\begin{aligned}
\xi_\inn(\vs(\vu),\theta) - \xi(\vu,\theta) = &\alpha\big (\coef^{-1}  \vs(\vu) - \delta^{-1} \vu\big ) -
\coef^{-1} (c+\alpha L_0) \log (\delta \sinh (\coef \vu)) \\ &-\alpha L(\vu) + \varphi(\vs (\vu),\theta) -\chi(\vu,\theta).
\end{aligned}
\end{equation}
We evaluate~\eqref{difxixinn} at $\vu=\vu_+$. As we pointed out in~\eqref{su+} and~\eqref{coshsinhu+}:
$$
\coef^{-1}  \vs(\vu_+) - \delta^{-1} \vu_+ = -i\frac{\alpha \pi}{2\coef \delta} +\mathcal{O}(\delta^2 \kappa^3),\qquad
\log (\delta \sinh (\coef \vu)) = \log \delta + i\frac{\pi}{2}.
$$
In addition, by Theorem~\ref{thmdiffinner} and using also expression~\eqref{su+} of $\vs(\vu_+)$, we have:
$$|\varphi\left(\vs(\vu_+),\theta\right)|\leq \frac{K}{|\vs(\vu_+)|}\leq \frac{K}{\dist}.$$
Then, by~\eqref{difxixinn} and using bounds in~\eqref{boundLL+}, we obtain readily:
$$
\xiin(\vs(\vu_+),\theta)=\xi(\vu_+,\theta)-i\frac{\alpha\pi}{2\coef\delta}-\coef^{-1}(c+\alpha L_0)\left(\log\delta+i\frac{\pi}{2}\right)-\alpha L_++\eta(\theta),
$$
with:
$$
\eta(\theta)=\varphi\left(\vs(\vu_+),\theta\right)+L(\vu_+)-L_+-\chi(\vu_+,\theta)+\mathcal{O}\left(\delta^2\dist^3\right).
$$
Clearly, $|\eta(\theta)|\leq K/\dist$ for some constant $K$.
\end{proof}

\begin{proof}[Proof of Proposition \ref{properrorUpsilonsinner}]
Since $\Delta(\vu,\theta)$, $\Delta_0(\vu,\theta)$ are real analytic, we just need to prove the result for $\Upsilon^{[-1]}-\Upsilon^{[-1]}_0$.

Rewriting expression~\eqref{defDelta1} of $\Delta_1=\Delta-\Delta_0$ one has 
\begin{equation}\label{rewdefDelta1}
\frac{\Delta_1(\vu,\theta)}{\cosh^{2/\coef}(\coef\vu)(1+P_1(\vu,\theta))}=
\sum_{l\neq 0 }\big (\Upsilon^{[l]}-\Upsilon^{[l]}_0\big )e^{il\xi(\vu,\theta)},
\end{equation}
with $\xi(\vu,\theta)$ defined in~\eqref{defxi}. We introduce the function
$$
F(\vu,\theta)= \delta \alpha^{-1} (\xi(\vu,\theta)-\theta).
$$
By Theorem \ref{thmdifpartsolinjective},
$(\xi(\vu,\theta),\theta)$ is injective in $\Doutinter\times\Tout$ then $(F(\vu,\theta),\theta)$ is also injective in the same domain. 
In particular, for all $(\vu,\theta)\in\Doutinter\times\mathbb{S}^1$, the change $(w,\theta)=(F(u,\theta),\theta)$ is a diffeomorphism between 
$\Doutinter\times\mathbb{S}^1$ and its image $\Doutintertilde\times\mathbb{S}^1$, with inverse $(\vu,\theta)=(G(w,\theta),\theta)$. 
Then, if we define the function:
$$
\Diffw(w,\theta)= \sum_{|l|\geq 1 }\big (\Upsilon^{[l]}-\Upsilon^{[l]}_0\big )e^{il(\theta+\delta^{-1}\alpha w)},
$$
by~\eqref{rewdefDelta1}, one has that:
\begin{equation}\label{DeltaDeltatilde-inner}
 \Diffw(w,\theta)=\frac{\Delta_1(G(w,\theta),\theta)}{\cosh^{2/\coef}(\coef G(w,\theta))(1+P_1(G(w,\theta),\theta))}.
\end{equation}
Note that $\Diffw(w,\theta)$ is $2\pi-$periodic in $\theta$, and its Fourier coefficient $\Diffw^{[-1]}(w)$ is:
$$
\Diffw^{[-1]}(w)=\big (\Upsilon^{[-1]}-\Upsilon^{[-1]}_0\big )e^{-i\delta^{-1}\alpha w}.
$$
Hence we know that for all $w\in \Doutintertilde$:
\begin{equation}\label{boundinterupsilons-inner}
\left|\Upsilon^{[-1]}-\Upsilon^{[-1]}_0\right|=\frac{1}{2\pi}\left|e^{i\delta^{-1}\alpha w}
\int_0^{2\pi}\Diffw (w,\theta)e^{i\theta}d\theta\right|
\leq\left|e^{i\delta^{-1}\alpha w}\right|\sup_{\theta\in\mathbb{S}^1}\left|\Diffw(w,\theta)\right|.
\end{equation}
For any $\theta_0 \in \mathbb{S}^1$,
we take $w= w_+:=F(\vu_+,\theta_0) = \delta\alpha^{-1} (\xi(\vu_+,\theta_0)-\theta_0)\in \Doutintertilde$ with $\vu_+= i\left(\frac{\pi}{2\coef}-\dist\delta\right)$
in \eqref{boundinterupsilons-inner}. Then~\eqref{boundinterupsilons-inner} yields:
\begin{equation}\label{boundinterupsilons-inner-v2}
\left|\Upsilon^{[-1]}-\Upsilon^{[-1]}_0\right|\leq \left |e^{i(\xi(\vu_+,\theta_0)-\theta_0)}\right |
\sup_{\theta\in\mathbb{S}^1}\left|\Diffw(w_+,\theta)\right|.
\end{equation}
Since $(F(\vu,\theta),\theta)$ is the inverse of $(G(w,\theta),\theta)$, from \eqref{DeltaDeltatilde-inner} we obtain:
$$
\Diffw(w_+,\theta)=\frac{\Delta_1(\vu_+,\theta)}{\cosh^{2/\coef}(\coef\vu_+)(1+P_1(\vu_+,\theta))}.
$$
Thus, using bound~\eqref{boundp1-diff} for $P_1$, that $|\cosh (\coef \vu_+) |\geq K \delta \dist$, and taking $\dist$ sufficiently large,
bound~\eqref{boundinterupsilons-inner-v2} writes out as:
\begin{equation}\label{difUpsilon1-supDelta1}
\left|\Upsilon^{[-1]}-\Upsilon^{[-1]}_0\right|\leq \frac{K}{\delta^{\frac{2}{\coef}}\dist^{\frac{2}{\coef}}}
e^{-\im \xi(\vu_+,\theta_0)}\sup_{\theta\in\mathbb{S}^1}\left|\Delta_1(\vu_+,\theta)\right|.
\end{equation}

By item 1 in Lemma~\ref{lemxixiin} and since $L_0\in\mathbb{R}$ (see Theorem \ref{thmdifpartsolinjective}) we have:
\begin{equation}\label{imxiu+}
\im \xi(\vu_+,\theta_0)= \frac{\alpha \pi}{2\coef \delta}-\alpha \dist  + \mathcal{O}(1)
\end{equation}
Therefore, since $\dist=\dist_0\log(1/\delta)$, bound~\eqref{difUpsilon1-supDelta1} writes out as:
\begin{equation} \label{boundUpsilons-Deltarough}
\left|\Upsilon^{[-1]}-\Upsilon_0^{[-1]}\right|\leq \frac{K}{\delta^{\frac{2}{\coef}+\alpha\dist_0}\log^{\frac{2}{\coef}}(1/\delta)}
e^{-\frac{\alpha\pi}{2\coef\delta}}\sup_{\theta\in\mathbb{S}^1}\left|\Delta_1(\vu_+,\theta)\right|.
\end{equation}
Now we claim that there exists a constant $K$ such that for all $\theta\in\mathbb{S}^1$:
\begin{equation}\label{boundDelta1rough-inner}
 |\Delta_1(\vu_+,\theta)|\leq K\frac{\delta^{-2+\alpha\dist_0}}{\dist^{1-\frac{2}{\coef}}}.
\end{equation}
Clearly, using~\eqref{boundDelta1rough-inner} in~\eqref{boundUpsilons-Deltarough} and recalling that $\dist=\dist_0\log(1/\delta)$ we obtain the claim of
the proposition (in the conservative case we just need to take $\coef=1$). Hence, the rest of the proof is devoted to prove bound~\eqref{boundDelta1rough-inner}.

To prove \eqref{boundDelta1rough-inner} we rewrite $\Delta_1(\vu_+,\theta)=\Delta(\vu_+,\theta)-\Delta_0(\vu_+,\theta)$ in the following way:
\begin{equation}\label{rewriteDelta1}
\Delta_1(\vu_+,\theta)=\Delta(\vu_+,\theta)-\Delta_\inn(\vu_+,\theta)+\Delta_\inn(\vu_+,\theta)-\Delta_0(\vu_+,\theta).
\end{equation}
First we bound $\Delta(\vu_+,\theta)-\Delta_\inn(\vu_+,\theta)$. We have:
\begin{align*}
\Delta(\vu,\theta)-\Delta_\inn(\vu,\theta)=&r_1^\uns(\vu,\theta)-r_1^\sta(\vu,\theta)-\delta^{-2}\left[\psiin^\uns(\vs(\vu),\theta)-\psiin^\sta(\vs(\vu),\theta)\right]\\
=&\delta^{-2}\left[\psi_1^\uns(\vs(\vu),\theta)-\psi_1^\sta(\vs(\vu),\theta)\right],
\end{align*}
where we have used that by definition $r_1^{\uns,\sta}(\vu,\theta)=\delta^{-2}\psi^{\uns,\sta}(\vs(\vu),\theta)$ and
$\psi_1^{\uns,\sta}=\psi^{\uns,\sta}-\psiin^{\uns,\sta}$. Thus, using that $|\psi_1^{\uns,\sta}(\vs,\theta)|\leq K \delta^{1-\gamma}|\vs|^{-2}$ by
Theorem~\ref{thmmatching}, it is clear that:
\begin{equation}\label{boundpartmatch}
\left|\Delta(\vu_+,\theta)-\Delta_\inn(\vu_+,\theta)\right|\leq \frac{K}{|\vs(\vu_+)|^2}\delta^{-1-\gamma}\leq\frac{K}{\dist^2}\delta^{-1-\gamma},
\end{equation}
where we have used that, by~\eqref{su+}
$
|\vs(\vu_+)|=\coef\dist+\mathcal{O}(\delta^2\dist^3).
$

Now we shall bound $\Delta_\inn(\vu_+,\theta)-\Delta_0(\vu_+,\theta)$. 
Recall definition~\eqref{defDelta0inn} of $\Delta_0$ and expression~\eqref{expression-deltain} of $\Delta_\inn$:
\begin{align*}
\Delta_0(\vu,\theta) &= \Delta_0^{\geq 0}(\vu,\theta) + \Delta_0^{<0}(\vu,\theta):=P(\vu,\theta) 
\left (\Upsilon^{[0]}+\sum_{l>0} \Upsilon_0^{[l]}e^{il\xi(\vu,\theta)}+\sum_{l<0} \Upsilon_0^{[l]}e^{il\xi(\vu,\theta)} \right ), \\
\Delta_\inn(\vu,\theta)& = \delta^{-2} P_\inn(\vs(\vu),\theta)\sum_{l<0} \Upsilon_\inn^{[l]}e^{il\xi_\inn(\vs(\vu),\theta)}
\end{align*}
with
\begin{equation}\label{defP}
P(\vu,\theta) = \cosh^{2/\coef} (\coef \vu) (1+P_1(\vu,\theta)),\qquad P_{\inn}(s,\theta) = s^{2/\coef} (1+P_1^{\inn}(s,\theta)).
\end{equation}
We introduce the following notation:
\begin{align*}
\hat F(\vu,\theta) &= \delta \alpha^{-1} \left (\xi(\vu,\theta)-\theta + \coef^{-1} (c+\alpha L_0)\left (-\log \delta -i \frac{\pi}{2}\right )- \alpha L_+\right ),\\
F_{\inn}(\vu,\theta) &=\delta \alpha^{-1} \left (\xi_\inn(\vs(\vu),\theta) - \theta + i\frac{\alpha \pi}{2 \coef \delta} \right ).
\end{align*}
Then, by definition of~\eqref{defUpsilon0lneg} of $\Upsilon_0^{[l]}$, if $l<0$, one has that
\begin{equation}\label{relationDelta0Deltainn_1}
\Upsilon_0^{[l]} e^{il\xi(\vu,\theta)} = \frac{ (-i)^{2/\coef}}{\delta^{2+2/\coef}} \Upsilon_\inn^{[l]} e^{l\frac{\alpha \pi}{2 \coef \delta}}
e^{il(\theta + \alpha \delta^{-1} \hat F(\vu,\theta))},\qquad l<0
\end{equation}
and
\begin{equation}\label{relationDelta0Deltainn_2}
\Upsilon_\inn^{[l]} e^{il\xi_\inn(\vs(\vu),\theta)} = \Upsilon_\inn^{[l]} e^{l\frac{\alpha \pi}{2 \coef \delta}} 
e^{il(\theta + \alpha \delta^{-1} F_\inn(\vu,\theta))},\qquad l<0.
\end{equation}
Therefore, from~\eqref{relationDelta0Deltainn_1} and~\eqref{relationDelta0Deltainn_2} we can rewrite $\Delta_0^{<0}$ and $\Delta_\inn$ as:
\begin{equation*}
\begin{aligned}
\Delta_0^{<0}(\vu,\theta) &=   (-i)^{2/\coef}\delta^{-2-2/\coef}P(\vu,\theta) 
\sum_{l<0}
\Upsilon_\inn^{[l]} e^{l\frac{\alpha \pi}{2 \coef \delta}} e^{il(\theta + \alpha \delta^{-1} \hat F(\vu,\theta))},\\
\Delta_\inn(\vu,\theta) &=\delta^{-2}P_\inn(\vs(\vu),\theta) \sum_{l<0}\Upsilon_\inn^{[l]} e^{l\frac{\alpha \pi}{2 \coef \delta}} e^{il(\theta + 
\alpha \delta^{-1} F_\inn(\vu,\theta))}.
\end{aligned}
\end{equation*}

We will not bound $|\Delta_0^{<0}(\vu,\theta)-\Delta_\inn(\vu,\theta)|$ directly. First, we will study the difference 
$|\Delta_0^{<0}(\vu,\theta)-\Delta_\inn(f(\vu,\theta),\theta)|$, where the function $f(\vu,\theta)$ is defined through:
\begin{equation}\label{functionf}
F_\inn(f(\vu,\theta),\theta)=\hat F(\vu,\theta).
\end{equation}
To see that $f(\vu,\theta)$ is well defined we proceed as follows. 
Since $(\xi_\inn(\vs,\theta),\theta)$ is injective in $\Ein\times\Tout$, one has that $(F_\inn(\vu,\theta),\theta)$ is also invertible
in $s^{-1} \left (\Ein\right)\times \Tout \subset \Doutinter\times\Tout$, choosing $\bar{\dist}$ and $\dist$ adequately. 
Let $(G_\inn(w,\theta),\theta)$ be the inverse of $(F_\inn(\vu,\theta),\theta)$. Then the function
$$
f(\vu,\theta)= G_\inn(\hat F(\vu,\theta),\theta),
$$
clearly satisfies equation \eqref{functionf}.
We emphasize that, by~\eqref{relationDelta0Deltainn_1} and~\eqref{relationDelta0Deltainn_2}:
\begin{align*}
\Delta_\inn (f(\vu,\theta),\theta) =  &
\delta^{-2}P_\inn(\vs(f(\vu,\theta)),\theta)\sum_{l<0}\Upsilon_\inn^{[l]} e^{l\frac{\alpha \pi}{2 \coef \delta}} e^{il(\theta + \alpha \delta^{-1} \hat F(\vu,\theta))}
\\ =& \frac{\delta^{2/\coef}P_\inn(\vs(f(\vu,\theta)),\theta)}{(-i)^{2/\coef}P(\vu,\theta) } \Delta_0^{<0}(\vu,\theta).
\end{align*}
All these considerations yield to the following decomposition of $\Delta_\inn-\Delta_0$:
\begin{equation}\label{Deltainn-Delta0}
\begin{aligned}
\Delta_\inn(\vu,\theta) - \Delta_0(\vu,\theta)=& \Delta_\inn(\vu,\theta)- \Delta_\inn (f(\vu,\theta),\theta)  -\Delta_0^{\geq 0}(\vu,\theta) \\
&+\Delta_\inn (f(\vu,\theta),\theta) \left (1 - \frac{(-i)^{2/\coef}P(\vu,\theta)}{\delta^{2/\coef}P_\inn(\vs(f(\vu,\theta)),\theta)}\right ).
\end{aligned}
\end{equation}
Now we proceed to bound each term in the above equality for $\vu=\vu_+$. 
For that we need a more precise knowledge about 
$f(\vu_+,\theta) = G_{\inn}(\hat F(\vu_+,\theta),\theta)$.  
First we compute $\hat F(\vu_+,\theta)$
which, straightforwardly from Lemma~\ref{lemxixiin}, is:
$$
\hat F(\vu_+,\theta) = i \frac{\pi}{2\coef} - i\delta \dist + \mathcal{O}(\delta \log \dist) = u_+ + \mathcal{O}(\delta \log \dist).
$$
Now we deal with $G_\inn$. By Remark~\ref{rmkchange-s} about $\vs(\vu)$ and using definition of $\xi_\inn(\vs,\theta)$ in Theorem~\ref{thmdiffinner}
$$
F_\inn(\vu,\theta) = u+ \alpha^{-1} \delta \coef^{-1} (c+\alpha L_0)\log (\vs(\vu)) + \alpha^{-1} \delta \varphi(\vs(\vu),\theta) +
\mathcal{O}\left (\big(\coef u- i\pi/2\big )^3\right ).
$$
consequently, since by bound~\eqref{P1inntheorem} of $|\varphi(\vs,\theta)| \leq K |\vs|^{-1}$, we have that
$$
F_\inn(\vu,\theta) = u + \delta \mathcal{O}(\log (\vs(\vu))) + \mathcal{O}\left (\big(\coef u- i\pi/2\big )^3\right ).
$$
Therefore, the inverse $(G_\inn(w,\theta),\theta)$ of $(F_\inn(\vu,\theta),\theta)$ also satisfies
$$
G_\inn(w,\theta) = w + \delta \mathcal{O}(\log (\vs(\vw))) + \mathcal{O}\left (\big(\coef w- i\pi/2\big )^3\right ).
$$
Recall that $\vs(\vu_+)\approx -i\coef \dist$ (see~\eqref{su+}). 
It is clear from the above considerations that
\begin{equation}\label{f-u+}
|f(\vu_+,\theta) -u_+| \leq K \delta \log \dist.
\end{equation}
In fact we have a more sharp bound of $|f(\vu_+,\theta) -u_+|$. 
Indeed, on the one hand, by using item 2 of Lemma~\ref{lemxixiin}
$$
|\hat F(\vu_+,\theta) - F_\inn(\vu_+,\theta)|=\frac{\delta}{\alpha} |\eta (\vu_+,\theta) |\leq K \delta \dist^{-1}
$$
and on the other hand, using~\eqref{f-u+} and the mean's value theorem as well, one has that
$$
|F_\inn(f(\vu_+,\theta),\theta)-F_\inn(\vu_+,\theta)| \geq \frac{1}{2} |f(\vu_+,\theta) -u_+|
$$
if $\dist$ is sufficiently large. 
Therefore, since $F_\inn(f(\vu,\theta),\theta) = \hat F(\vu,\theta)$:
\begin{equation}\label{f-u+sharp}
|f(\vu_+,\theta) -u_+| \leq 2|\hat F(\vu_+,\theta) - F_\inn(\vu_+,\theta)| \leq M \delta \dist^{-1}.
\end{equation}

Once we have bound $f(\vu_+,\theta) -u_+$ we proceed to bound  $\Delta_\inn(\vu_+,\theta)- \Delta_\inn (f(\vu_+,\theta),\theta)$.
We claim that, for $\vu_\lambda= \vu_+ +\lambda(f(\vu_+,\theta)-\vu_+)$, $\lambda\in [0,1]$, 
\begin{equation}\label{boundpartialuDeltainn}
\left |\partial_{\vu} \Delta_\inn(\vu_\lambda,\theta) \right | \leq K \delta^{-3} \dist^{2/\coef} e^{-\alpha \dist}.
\end{equation}
Indeed, by definition~\eqref{DeltaApproxDeltain} of $\Delta_\inn$, using Corollary~\ref{bounddiffinnerexp} 
and that $\partial_{\vu} s(\vu)=-\delta\frac{s^2(\vu)}{\cosh^{2}(\coef \vu)}$ one obtains:
\begin{equation}\label{exprpartialuDeltainn}
\left |\partial_{\vu} \Delta_\inn(\vu,\theta) \right |= \delta^{-2} \left |\partial_\vs \Delta \psiin (\vs(\vu),\theta) \partial_\vu \vs(\vu)\right |
\leq K \delta^{-1}e^{\im \xi_\inn(\vs(\vu),\theta)}  \frac{|\vs(\vu)|^{2+2/\coef}}{|\cosh^{2}(\coef \vu)|}.
\end{equation}
We have to control the terms in the above inequality when $\vu=\vu_\lambda$. By~\eqref{f-u+sharp}
we have that $\vu_\lambda= \vu_+ +\mathcal{O}(\delta \dist^{-1})$. 
Therefore, by expressions~\eqref{su+} and~\eqref{coshsinhu+} of $\vs(\vu_+)$ and $\cosh(\coef \vu_+)$ one obtains:
$$
|\vs(\vu_\lambda)| =\mathcal{O}(\dist),\qquad |\cosh (\vu_\lambda)| =\mathcal{O}(\dist \delta).
$$
Moreover, 
$$
\left | \im \xi_\inn(\vs(\vu),\theta) - \im \xi_\inn(\vu_+,\theta)\right |\leq 
\int_{0}^1 |\partial_\vs \im \xi_\inn(\vs(\vu_\lambda),\theta) ||\partial_{\vu}\vs(\vu_\lambda)|\, d\lambda \leq \frac{K}{\dist}.
$$
Then, by item 2 of Lemma~\ref{lemxixiin} which relates $\xi_\inn(\vs(\vu_+),\theta)$ with $\xi(\vu_+,\theta)$ 
and expression~\eqref{imxiu+} of $\im \xi(\vu_+,\theta)$ one gets: 
\begin{equation}\label{imxiinnu+}
\begin{aligned}
\im \xi_\inn(\vs(\vu_+),\theta) &=-i\alpha \kappa - \im \left [\coef^{1-} (c+\alpha L_0) \left (\log \delta +i\frac{\pi}{2}\right )\right ]+
\mathcal{O}(1) \\ &= 
-i\alpha \kappa + \mathcal{O}(1),
\end{aligned}
\end{equation}
where in the last equality we have used that $L_0\in \mathbb{R}$ (see Theorem~\ref{thmdifpartsolinjective}). Bound~\eqref{boundpartialuDeltainn}
follows from~\eqref{exprpartialuDeltainn} and previous considerations.

By the mean's value theorem and using bounds~\eqref{boundpartialuDeltainn} and~\eqref{f-u+sharp}
\begin{equation}\label{boundDelta1_1}
\left | \Delta_\inn(\vu_+,\theta)- \Delta_\inn (f(\vu_+,\theta),\theta)\right | \leq K \frac{\delta^{-2}}{\kappa^{1-2/\coef}}e^{-\alpha \dist} =
K \frac{\delta^{-2 + \alpha \kappa_0}}{\kappa^{1-2/\coef}}
\end{equation}
where we have used again that $\kappa=\kappa_0\log(1/\delta)$. 

Now we deal with $\Delta^{\geq 0}_0$, the second term in the decomposition~\eqref{Deltainn-Delta0} of $\Delta_\inn-\Delta_0$. 
On the one hand, we recall that in the conservative case $\Upsilon^{[0]}=0$, and in the dissipative one we take $\param=\param_*(\delta)$, so that:
$|\Upsilon^{[0]}|= |a_1|\delta^{a_2}e^{-\frac{a_3\pi}{2\coef\delta}}$,
for some $a_1$, $a_2\in\mathbb{R}$ and $a_3>0$. On the other hand, from the definition~\eqref{defUpsilon0lpos} of $\Upsilon_0^{[l]}$ and
the expression of $\xi(\vu_+,\theta)$ in Lemma~\ref{lemxixiin} one has that:
$$
\left|\Upsilon_0^{[l]} e^{il\xi(\vu_+,\theta)}\right|\leq K\delta^{-2-2/\coef}e^{-l\left (\frac{\alpha\pi}{\coef\delta} - \alpha \kappa -M\right )}
$$
for certain constant (independent of $\kappa$) $M$.
Hence, using definition~\eqref{defP} of $P$,
that $|\cosh(\coef\vu_+)|\leq K\delta\dist$, bound~\eqref{boundp1-diff} of $P_1$ and that $\dist = \dist_0 \log (1/\delta)$, we obtain:
\begin{equation}\label{boundDelta1_2}
|\Delta_0^{\geq 0}(\vu_+,\theta)| \leq K(\delta\dist)^{\frac{2}{\coef}}\left(|a_1|\delta^{a_2}e^{-\frac{a_3\pi}{2\coef\delta}}+
\delta^{-2-2/\coef-\alpha\dist_0}e^{-\frac{\alpha\pi}{\coef\delta}}\right),
\end{equation}
where we recall that $a_3>0$. 

We finally deal with the third term in~\eqref{Deltainn-Delta0}. Using definitions~\eqref{defP} of $P$ and $P_\inn$ and expressions~\eqref{su+}
and~\eqref{coshsinhu+} of $\vs(\vu_+)$ and $\cosh(\coef \vu_+)$ one readily obtains: 
$$
\left | 1 - \frac{(-i)^{2/\coef}P(\vu,\theta)}{\delta^{2/\coef}P_\inn(\vs(f(\vu,\theta)),\theta)}\right |\leq \frac{K}{\dist}.
$$
In addition, by Corollary~\ref{bounddiffinnerexp}
$$
|\Delta_\inn(\vu_+,\theta)| =\delta^{-2} |\Delta\psiin(\vs(\vu_+),\theta)| \leq K |\vs(\vu_+)|^{2/\coef} e^{\im \xi_\inn(\vs(\vu_+))} 
\leq K \delta^{-2} \dist^{2/\coef} e^{-\alpha \dist},
$$
where in the last inequality we have used expression~\eqref{imxiinnu+} of $\im \xi_\inn(\vs(\vu_+),\theta)$.
Therefore, since $\kappa=\kappa_0\log(1/\delta)$ and using also bound~\eqref{boundDelta1_1}
\begin{equation}\label{boundDelta1_3}
|\Delta_\inn (f(\vu_+,\theta),\theta)| \left | 1 - \frac{(-i)^{2/\coef}P(\vu,\theta)}{\delta^{2/\coef}P_\inn(\vs(f(\vu,\theta)),\theta)}\right |
\leq K \frac{\delta^{-2}}{\dist^{1-2/\coef}} e^{-\alpha \dist}=  K \frac{\delta^{-2+\alpha \kappa_0}}{\dist^{1-2/\coef}}.
\end{equation}

In conclusion, using bounds \eqref{boundDelta1_1}, \eqref{boundDelta1_2} and \eqref{boundDelta1_3} in \eqref{Deltainn-Delta0} we obtain:
$$|\Delta_\inn(\vu_+,\theta)-\Delta_0(\vu_+,\theta)|\leq K\frac{\delta^{-2+\alpha\dist_0}}{\dist^{1-\frac{2}{\coef}}}.$$
Using this bound and bound \eqref{boundpartmatch} in \eqref{rewriteDelta1} we obtain:
$$|\Delta_1(\vu_+,\theta)|\leq \frac{K}{\dist^2}\delta^{-1-\gamma}+K\frac{\delta^{-2+\alpha\dist_0}}{\dist^{1-\frac{2}{\coef}}}.$$
Then we just need to recall that $1-\gamma>\alpha\dist_0$ by hypothesis, so that $\delta^{-1-\gamma}<\delta^{-2+\alpha\dist_0}$, and we obtain 
bound \eqref{boundDelta1rough-inner}.
\end{proof}

\section{The inner equation. Theorems~\ref{thminner} and~\ref{thmdiffinner}}\label{sec:innergeneral}
In this section we present an exhaustive sketch of the proofs of Theorems~\ref{thminner}, in Section~\ref{sec:inner} and Theorem~\ref{thmdiffinner}
in Section~\ref{sec:diffinner}. We refer to the interested reader to~\cite{CastejonPhDThesis} where all the details are provided.

The inner equation was introduced in Section~\ref{subsec:derivationinner} in~\eqref{PDE-innershort} as $\Linner(\psi_\inn)=\Minner(\psi_\inn,0)$.
\subsection{Existence and properties of $\psi_{\inn}^{\uns,\sta}$}\label{sec:inner}
In this section we will prove Theorem~\ref{thminner}.
As we pointed out in equation~\eqref{eqinverses}, the operator $\Ginner^\uns$ defined in~\eqref{defGuns} is a right inverse of the linear operator $\Linner$
(see~\eqref{defopLinner} for its definition). Thus, the inner equation~\eqref{PDE-innershort} can be written as the following fixed point equation:
\begin{equation}\label{innerfixedpoint}
\psiin^\uns=\tilde\Minner^\uns(\psiin^\uns),
\end{equation}
where:
\begin{equation}\label{defopMinnertilde}
 \tilde\Minner^\uns(\phi)=\Ginner^\uns\circ\Minner(\phi,0),
\end{equation}
and $\Minner$ is defined in \eqref{defopMinner}. The proof of Theorem \ref{thminner} relies on proving that the operator $\tilde\Minner^\uns$
has a fixed point in a suitable Banach space.

In Section~\ref{subsec:Bspaces} we define such Banach space and provide some technical properties of,
the operator $\Ginner^{\uns}$ and the functions $\hat F$, $\hat G$, $\hat H$
(defined in~\eqref{notationFGHhat}) and their derivatives that will be used also in
Sections~\ref{subsec:statementDiffin} and~\ref{subsec:mmatchtilde}.

\subsubsection{Banach spaces and technical lemmas}\label{subsec:Bspaces}
For $\phi:\Din{\uns}\times\Tout\to\mathbb{C}$, writing $\phi(\vs,\theta)=\sum_{l\in\mathbb{Z}}\phi^{[l]}(\vs)e^{il\theta}$, we define the norms:
$$
\|\phi^{[l]}\|_n^\uns:=\sup_{\vs\in\Din{\uns}}|\vs^n\phi(\vs)|,\qquad
 \|\phi\|_{n,\ost}^\uns:=\sum_{l\in\mathbb{Z}}\|\phi^{[l]}\|_n^\uns e^{|l|\ost},
$$
and the Banach space $\Bsin{\uns}{n}$:
$$
 \Bsin{\uns}{n}:=\{\phi:\Din{\uns}\times\Tout\to\mathbb{C}\,:\, \phi\textrm{ is analytic,}\,\|\phi\|_{n,\ost}^\uns<\infty\}.
$$
We also consider the following norm:
$$
\llfloor\phi\rrfloor_{n,\ost}^\uns:=\|\phi\|_{n,\ost}^\uns+\|\partial_\vs\phi\|_{n+1,\ost}^\uns+\|\partial_\theta\phi\|_{n+1,\ost}^\uns,
$$
and the corresponding Banach space:
$$
 \Bsinfloor{\uns}{n}:=\{\phi:\Din{\uns}\times\Tout\to\mathbb{C}\,:\, \phi\textrm{ is analytic,}\,\llfloor\phi\rrfloor_{n,\ost}^\uns<\infty\}.
$$
For the stable case, we define analogous norms $\|.\|_{n,\ost}^\sta$ and $\llfloor.\rrfloor_{n,\ost}^\sta$ and Banach spaces $\Bsin{\sta}{n}$ and $\Bsinfloor{\sta}{n}$, just replacing the domain $\Din{\uns}$ by $\Din{\sta}$.

The following result has Theorem~\ref{thminner} as an obvious corollary.
\begin{proposition}\label{propinner}
Let $\beta_0>0$ and $\distin>0$ be large enough. Equation~\eqref{innerfixedpoint} has two solutions $\psiin^\uns\in \Bsinfloor{\uns}{3}$ and
$\psiin^{\sta}\in \Bsinfloor{\sta}{3}$ and there exists $M>0$ such that:
$$\llfloor\psiin^{\uns,\sta}\rrfloor_{3,\ost}^{\uns,\sta}\leq M,\qquad \llfloor\psiin^{\uns,\sta}-\tilde\Minner^{\uns,\sta}(0)\rrfloor_{4,\ost}^{\uns,\sta}\leq M.$$
\end{proposition}
The rest of the section is devoted to prove this proposition for the unstable case. The proof for the stable one is completely analogous.

We present (see~\cite{CastejonPhDThesis} for their proofs) some technical results.
\begin{enumerate}
\item \label{itemproposnorm-inner}\textit{Banach spaces}. Let $n_1,n_2 \geq 0$. There exists $M>0$ such that
\begin{enumerate}
\item if $n_1 \leq n_2$, then $\Bsin{\uns}{n_2}\subset\Bsin{\uns}{n_1}$ and
$$
\|\phi\|_{n_1,\ost}^{\uns}\leq \frac{M}{\distin^{n_2-n_1}}\|\phi\|_{n_2,\ost}^{\uns}.
$$
\item If $\phi_1\in\Bsin{\uns}{n_1}$, $\phi_2\in\Bsin{\uns}{n_2}$, then $\phi_1\phi_2\in\Bsin{\uns}{n_1+n_2}$ and
$$
\|\phi_1\phi_2\|_{n_1+n_2,\ost}^{\uns}\leq M\|\phi_1\|_{n_1,\ost}^{\uns}\|\phi_2\|_{n_2,\ost}^{\uns}.
$$
\end{enumerate}
\item \label{itempropsGinner}\textit{The operator $\Ginner^{\uns}$}. Let $n\geq1$ and $\phi\in\Bsin{\uns}{n}$. There exists a constant $M$ such that:
$$
\|\Ginner^\uns(\phi)\|_{n-1,\ost}^\uns\leq M\|\phi\|_{n,\ost}^\uns,\qquad \llfloor\Ginner^\uns(\phi)\rrfloor_{n-1,\ost}^\uns\leq M\|\phi\|_{n,\ost}^\uns.
$$
In addition, if $\phi^{[0]}(\vs)=0$, then $\|\Ginner^\uns(\phi)\|_{n,\ost}^\uns\leq M\|\phi\|_{n,\ost}^\uns$.
\item \textit{The nonlinear terms, $\hat{F}, \hat{G}$ and $\hat{H}$}. Recall that these functions are defined in~\eqref{notationFGHhat}.
Let $C$ be any constant. Then:
\begin{enumerate}
\item \label{itemFGHhats} If $\phi\in\Bsin{\uns}{3}$ with $\|\phi\|_{3,\ost}^\uns\leq C$, there exists $M>0$ such that
$$
\|\hat F(\phi,0)\|_{4,\ost}^\uns, \;\|\hat G(\phi,0)\|_{2,\ost}^\uns,\;\|\hat H(\phi,0)\|_{3,\ost}^\uns\leq M.
$$
\item
If $\phi\in\Bsin{\uns}{3}$ with $\|\phi\|_{3,\ost}^\uns\leq C$ and $\distin$ is sufficiently large, there exists $M>0$ and:
$$
\|D_\phi\hat F(\phi,0)\|_{2,\ost}^\uns,\; \|D_\phi\hat G(\phi,0)\|_{0,\ost}^\uns,\; \|D_\phi\hat H(\phi,0)\|_{1,\ost}^\uns\leq M.
$$
\item \label{itemdifFGHhats} If $\phi_1,\phi_2\in\Bsin{\uns}{3}$ is such that $\|\phi_i\|_{3,\ost}^\uns\leq C$, for $i=1,2$, there exists $M>0$ such that:
\begin{align*}
&\|\hat F(\phi_1,0)- \hat F(\phi_2,0)\|_{5,\ost}^\uns\leq M\|\phi_1-\phi_2\|_{3,\ost}^\uns,\\
&\|\hat G(\phi_1,0)- \hat G(\phi_2,0)\|_{3,\ost}^\uns\leq M\|\phi_1-\phi_2\|_{3,\ost}^\uns,\\
&\|\hat H(\phi_1,0)- \hat H(\phi_2,0)\|_{4,\ost}^\uns\leq M\|\phi_1-\phi_2\|_{3,\ost}^\uns.
\end{align*}
\end{enumerate}
\end{enumerate}

\subsubsection{The fixed point equation}
Finally we can proceed to prove the existence of a fixed point of the operator $\tilde\Minner^\uns$, given in \eqref{defopMinnertilde}, in the Banach space
$\Bsinfloor{\uns}{3}$. We first begin by studying the \textit{independent term} $\tilde\Minner^\uns(0)$.

\begin{lemma}\label{lemfirstitpsiin}
Let $\Minner^{\uns}$ be the operator defined in~\eqref{defopMinnertilde}.
There exists a constant $M$ such that:
$\llfloor\tilde\Minner^\uns(0)\rrfloor_{3,\ost}^\uns\leq M$.
\end{lemma}
\begin{proof}
Noting that:
$$\Minner(0,0)=\hat F(0,0)+\frac{\coef+1}{b}\vs^{-1}\hat H(0,0),$$
by item~\ref{itemFGHhats} in Section~\ref{subsec:Bspaces} it is clear that $\|\Minner(0,0)\|_{4,\ost}^\uns\leq K$ for some constant $K$.
Then, since $\tilde\Minner^\uns(0)=\Ginner^\uns\circ\Minner(0,0)$, one just needs to use item~\ref{itempropsGinner} in
Section~\ref{subsec:Bspaces} to obtain the claim of the lemma.
\end{proof}

The next step is to find a Lipschitz constant of the operator $\tilde\Minner^\uns$.
\begin{lemma}\label{lemlipconstinner}
Let $\phi_1,\phi_2\in\Bsinfloor{\uns}{3}$ such that $\llfloor\phi_i\rrfloor_{3,\ost}^\uns\leq C$, $i=1,2$, for some constant $C$. Then, there exists a constant $M$ such that:
$$\llfloor\tilde\Minner^\uns(\phi_1)-\tilde\Minner^\uns(\phi_2)\rrfloor_{4,\ost}^\uns\leq M\llfloor\phi_1-\phi_2\rrfloor_{3,\ost}^\uns.$$
\end{lemma}
\begin{proof}
First we note that since $\Ginner^\uns$ is linear:
 $$\tilde\Minner^\uns(\phi_1)-\tilde\Minner^\uns(\phi_2)=\Ginner^\uns(\Minner(\phi_1,0)-\Minner(\phi_2,0)).$$
Hence, by item~\ref{itempropsGinner} in Section~\ref{subsec:Bspaces}, we just need to prove:
\begin{equation}\label{difMinnerphi12}
\|\Minner(\phi_1,0)-\Minner(\phi_2,0)\|_{5,\ost}^\uns\leq K\llfloor\phi_1-\phi_2\rrfloor_{3,\ost}^\uns.
\end{equation}
Now, by definition \eqref{defopMinner} of $\Minner$, we decompose:
 \begin{align*}
\Minner(\phi_1,0)-&\Minner(\phi_2,0)=c\vs^{-1}\partial_\theta(\phi_1-\phi_2)+\hat F(\phi_1,0)- \hat F(\phi_2,0)\\
&+\frac{\coef+1}{b}\vs^{-1} \left[\hat H(\phi_1,0)- \hat H(\phi_2,0)\right]-\left[\hat G(\phi_1,0)- \hat G(\phi_2,0)\right]\partial_\theta\phi_1\\
&- \hat G(\phi_2,0)\partial_\theta(\phi_1-\phi_2)+\vs^2\left[2b\phi_2+\hat H(\phi_2,0)\right]\partial_\vs(\phi_1-\phi_2)\\
&+\vs^2\left[2b(\phi_1-\phi_2)+\hat H(\phi_1,0)- \hat H(\phi_2,0)\right]\partial_\vs\phi_1.
 \end{align*}
One just needs to use the properties in Section~\ref{subsec:Bspaces} exposed in item~\ref{itemproposnorm-inner},
items~\ref{itemFGHhats} and~\ref{itemdifFGHhats}, and take into account that:
 $$\|\phi_1-\phi_2\|_{3,\ost}^\uns\leq \llfloor\phi_1-\phi_2\rrfloor_{3,\ost}^\uns,\qquad \|\partial_\theta(\phi_1-\phi_2)\|_{4,\ost}^\uns\leq \llfloor\phi_1-\phi_2\rrfloor_{3,\ost}^\uns,$$
 $$\|\partial_\vs(\phi_1-\phi_2)\|_{4,\ost}^\uns\leq \llfloor\phi_1-\phi_2\rrfloor_{3,\ost}^\uns,$$
and then \eqref{difMinnerphi12} is obtained easily.
\end{proof}

\begin{proof}[End of the proof of Proposition \ref{propinner}]
Let $\phi_1,\phi_2\in B(2\llfloor\tilde\Minner^\uns(0)\rrfloor_{3,\ost}^\uns)\subset\Bsinfloor{\uns}{3}$. By using item~\ref{itemproposnorm-inner}
in Section~\ref{subsec:Bspaces} and Lemma~\ref{lemlipconstinner} we obtain:
$$
\llfloor\tilde\Minner^\uns(\phi_1)-\tilde\Minner^\uns(\phi_2)\rrfloor_{3,\ost}^\uns\leq\frac{K}{\distin}\llfloor\tilde\Minner^\uns(\phi_1)
-\tilde\Minner^\uns(\phi_2)\rrfloor_{4,\ost}^\uns\leq \frac{K}{\distin}\llfloor\phi_1-\phi_2\rrfloor_{3,\ost}^\uns.
$$
Hence, for $\distin$ sufficiently large, $\Minner^\uns$ is contractive and:
$$
\tilde\Minner^\uns:B(2\llfloor\tilde\Minner^\uns(0)\rrfloor_{3,\ost}^\uns)\rightarrow B(2\llfloor\tilde\Minner^\uns(0)\rrfloor_{3,\ost}^\uns),
$$
so that it has a unique fixed point $\psiin^\uns\in B(2\llfloor\tilde\Minner^\uns(0)\rrfloor_{3,\ost}^\uns)$. In other words, $\psiin^\uns$ satisfies
equation~\eqref{innerfixedpoint}, and $\llfloor\psiin^\uns\rrfloor_{3,\ost}^\uns\leq2\llfloor\tilde\Minner^\uns(0)\rrfloor_{3,\ost}^\uns\leq K$ by Lemma~\ref{lemfirstitpsiin}.
To finish the proof, using Lemma~\ref{lemlipconstinner} again, we conclude:
$$
\llfloor\psiin^\uns-\tilde\Minner^\uns(0)\rrfloor_{4,\ost}^\uns=
\llfloor\tilde\Minner^\uns(\psiin^\uns)-\tilde\Minner^\uns(0)\rrfloor_{4,\ost}^\uns\leq K\llfloor\psiin^\uns\rrfloor_{3,\ost}^\uns\leq K.
$$
\end{proof}

\subsection{The difference $\Delta\psiin$}\label{sec:diffinner}
In this section we provide the proof of Theorem~\ref{thmdiffinner} which deals with the form of $\Delta \psiin=\psiin^{\uns}-\psiin^{\sta}$.
As in Section~\ref{sec:inner} we refer to the reader to~\cite{CastejonPhDThesis} for the details.

\subsubsection{Preliminary considerations}\label{subsec:diffinner-pre}
As we explained in Section~\ref{subsec:introdiffinner}, since $\psiin^\uns$ and $\psiin^\sta$ are solutions of the same
equation~\eqref{PDE-inner}, subtracting $\psiin^\uns$ and $\psiin^\sta$ and using the mean value theorem, one obtains that
$\Delta\psiin=\psiin^\uns-\psiin^\sta$ satisfies equation~\eqref{PDE-difference-intro}:
\begin{align}\label{PDE-difference}
-\alpha\partial_\theta\Delta\psiin&+\coef\partial_\vs\Delta\psiin-2\vs^{-1}\Delta\psiin\nonumber\\
&=a_1(\vs,\theta)\Delta\psiin+a_2(\vs,\theta)\partial_\vs\Delta\psiin+(c\vs^{-1}+a_3(\vs,\theta))\partial_\theta\Delta\psiin.
\end{align}
Denoting $\psi_\lambda=(\psiin^\uns+\psiin^\sta)/2+\lambda(\psiin^\uns-\psiin^\sta)/2$, the functions $a_i$ are:
\begin{eqnarray}
a_1(\vs,\theta)&=&\frac{1}{2}\int_{-1}^1\partial_\psi\hat F(\psi_\lambda,0)d\lambda+
\frac{\coef+1}{2b}\vs^{-1}\int_{-1}^1\partial_\psi\hat H(\psi_\lambda,0)d\lambda\nonumber\\
&&-\frac{1}{2}\int_{-1}^1\partial_\psi\hat G(\psi_\lambda,0)\partial_\theta\psi_\lambda d\lambda+
b\vs^2(\partial_\vs\psiin^\uns+\partial_\vs\psiin^\sta)\nonumber\\
&&+\frac{1}{2}\vs^2\int_{-1}^1\partial_\psi\hat H(\psi_\lambda,0)\partial_\vs\psi_\lambda d\lambda,\label{defa1}\\
a_2(\vs,\theta)&=&b\vs^2(\psiin^\uns+\psiin^\sta)+\frac{1}{2}\vs^2\int_{-1}^1\hat H(\psi_\lambda,0)d\lambda\label{defa2}\\
a_3(\vs,\theta)&=&-\frac{1}{2}\int_{-1}^1\hat G(\psi_\lambda,0)d\lambda\label{defa3}.
\end{eqnarray}
We recall that $\hat F$, $\hat G$ and $\hat H$ are defined in~\eqref{notationFGHhat} and that the difference $\Delta\psiin$ is defined
for $\vs\in\Ein=\Din{\uns}\cap\Din{\sta}$ (see Figure~\ref{figEin}) and $\theta\in\Tout$.

We already argued in Section~\ref{subsec:introdiffinner} that $\Delta\psiin$ can be written as:
$$
\Delta \psiin(\vs,\theta) = \Pin(\vs,\theta) \kintilde(\xi_{\inn}(\vs,\theta),\theta),
$$
being $\kintilde(\tau)$ a $2\pi-$periodic function, $\Pin$ a particular
solution of the equation~\eqref{PDE-difference} and $\xi_\inn$ is a solution of the homogeneous PDE
equation~\eqref{PDE-k-intro}:
\begin{equation}\label{PDE-k}
-\alpha\partial_\theta k+\coef\partial_\vs k=a_2(s,\theta)\partial_\vs k +(c\vs^{-1}+a_3(\vs,\theta))
\partial_\theta k
\end{equation}
such that $(\xi_\inn(\vs,\theta),\theta)$ is injective in $\Ein\times\Tout$.

To take advantage of the perturbative setting, we look for $\Pin$ and $\xi_\inn$
of the form:
\begin{align}
\Pin(\vs,\theta) &= s^{2/d} (1+\Pin_1(\vs,\theta)),\label{formaPin}\\
\xi_\inn(\vs,\theta) &= \theta+\coef^{-1}\alpha\vs+\coef^{-1}(c+\alpha L_0)\log\vs+\varphi(\vs,\theta)\label{formaxinn}.
\end{align}
It can be easily checked that $\Pin_1$ has to be a solution of:
\begin{equation}\label{PDE-P1}
\begin{aligned}
 -\alpha\partial_\theta \Pinsmall+\coef\partial_\vs \Pinsmall=&(a_1+2\coef^{-1}\vs^{-1}a_2)(1+\Pinsmall)+a_2\partial_\vs \Pinsmall \\&+(c\vs^{-1}+a_3)\partial_\theta \Pinsmall
\end{aligned}
\end{equation}
and, denoting,
\begin{equation}\label{defa2bar}
\bar a_2(\vs,\theta)=a_2(\vs,\theta)-\coef\vs^{-1}L_0,
\end{equation}
$\varphi$ has to be a solution of:
\begin{align}\label{PDE-phi}
-\alpha\partial_\theta \varphi+\coef\partial_\vs\varphi=&\coef^{-1}\alpha\bar a_2+\coef^{-1}(c+\alpha L_0)\vs^{-1}a_2+a_3
+a_2\partial_\vs\varphi \nonumber\\ &+(c\vs^{-1}+a_3)\partial_\theta\varphi.
\end{align}
First, note that in the left-hand side of equations \eqref{PDE-phi} and \eqref{PDE-P1} we have the same linear operator, namely:
$$
\Ldiffin(\phi)=-\alpha\partial_\theta\phi+\coef\partial_\vs\phi.
$$
Moreover, $\varphi$ and $\Pinsmall$ are  defined in the same domain $\Ein\times\Tout$. Now, to solve equation \eqref{PDE-phi} we consider the operator:
\begin{align*}
\oprhsphi(\phi)=&\coef^{-1}\alpha\bar a_2(\vs,\theta)+\coef^{-1}(c+\alpha L_0)\vs^{-1}a_2(\vs,\theta)+a_3(\vs,\theta)\\
&+a_2(\vs,\theta)\partial_\vs\phi+(c\vs^{-1}+a_3(\vs,\theta))\partial_\theta\phi,
\end{align*}
and to solve equation~\eqref{PDE-P1}:
\begin{equation}\label{defoprhsP}
\begin{aligned}
\oprhsP(\phi)=&(a_1(\vs,\theta)+2\coef^{-1}\vs^{-1}a_2(\vs,\theta))(1+\phi)+a_2(\vs,\theta)\partial_\vs \phi\\&+(c\vs^{-1}+a_3(\vs,\theta))\partial_\theta \phi.
\end{aligned}
\end{equation}
Then equation~\eqref{PDE-phi} and~\eqref{PDE-P1} can be written respectively as:
\begin{equation}\label{PDE-phiP1short}
 \Ldiffin(\varphi)=\oprhsphi(\varphi),  \qquad \Ldiffin(\Pinsmall)=\oprhsP(\Pinsmall).
\end{equation}

Note that both equations in~\eqref{PDE-phiP1short} can be rewritten as fixed point equations using a suitable right inverse of the operator
$\Ldiffin$.

Let us denote $s_0=-i \distin$.
Then we define the following right inverse of $\Ldiffin$, which shall denote by $\Gdiffin$, as the operator acting on functions $\phi$ given by:
$$
\Gdiffin(\phi)(\vs,\theta)=\sum_{l\in\mathbb{Z}}\Gdiffin^{[l]}(\phi)(\vs)e^{il\theta},
$$
where:
\begin{align*}
 \Gdiffin^{[l]}(\phi)(\vs)&=\int_{\vs_0}^\vs e^{-il\alpha(w-\vs)}\phi^{[l]}(w)dw, &\qquad \textrm{if } l>0,\\
 \Gdiffin^{[l]}(\phi)(\vs)&=\int_{-i\infty}^\vs  e^{-il\alpha(w-\vs)}\phi^{[l]}(w)dw, &\qquad \textrm{if } l\leq 0.
\end{align*}
One can easily see that $\Linner\circ\Gdiffin=\textrm{Id}$.

\subsubsection{Banach Spaces, properties of $a_1,a_2, a_3$ and of the linear operator $\Gdiffin$}
Now we shall introduce the Banach spaces in which we will solve equations~\eqref{PDE-phiP1short}.
These spaces and norms are basically the same as in Section~\ref{subsec:Bspaces}, but restricted to $\Ein=\Din{\uns}\cap\Din{\sta}$.

For $\phi:\Ein\times\Tout\to\mathbb{C}$, writing $\phi(\vs,\theta)=\sum_{l\in\mathbb{Z}}\phi^{[l]}(\vs)e^{il\theta}$, we define the norms:
\begin{align*}
 &\|\phi^{[l]}\|_n:=\sup_{\vs\in\Ein}\,|\vs^n\phi(\vs)|,\qquad \|\phi\|_{n,\ost}:=\sum_{l\in\mathbb{Z}}\|\phi^{[l]}\|_n e^{|l|\ost}, \\
&\llfloor\phi\rrfloor_{n,\ost}:=\|\phi\|_{n,\ost}+\|\partial_\vs\phi\|_{n+1,\ost}+\|\partial_\theta\phi\|_{n+1,\ost},
\end{align*}
Then we define the Banach spaces:
\begin{align*}
 \Bsdiffin{n}&:=\{\phi:\Ein\times\Tout\to\mathbb{C}\,:\, \phi\textrm{ is analytic,}\,\|\phi\|_{n,\ost}<\infty\},\\
 \Bsdiffinfloor{n}&:=\{\phi:\Ein\times\Tout\to\mathbb{C}\,:\, \phi\textrm{ is analytic,}\,\llfloor\phi\rrfloor_{n,\ost}<\infty\}.
\end{align*}
These Banach spaces and norms satisfy the same properties stated in item~\ref{itemproposnorm-inner} in Section~\ref{subsec:Bspaces}.
We will use them without mention.
\begin{lemma}\label{lema123}
Consider the functions $a_i(\vs,\theta)$, $i=1,2,3$ defined respectively in~\eqref{defa1}, \eqref{defa2} and~\eqref{defa3}, and the function $\bar a_2(\vs,\theta)$ defined in~\eqref{defa2bar}. There exists a constant $M$ such that:
$$
\|a_1\|_{2,\ost}\leq M,\quad \|a_2\|_{1,\ost}\leq M \quad \|\bar a_2^{[0]}\|_2\leq M,\quad \|a_3\|_{2,\ost}\leq M.
$$
\end{lemma}
\begin{proof}
Recalling that by Proposition \ref{propinner} we have that $\psiin^\uns\in\Bsinfloor{\uns}{3}$, using bounds of $\hat{F}, \hat{G}$ and $\hat{H}$
in Section~\ref{subsec:Banach-match}, it is straightforward to prove the bounds for $a_1,a_2$ and $a_3$.

The bound for $\bar a_2^{[0]}$, the
mean of $\bar a_2$ is more involved. We will give an sketch of the proof, the details are in~\cite{CastejonPhDThesis}. Let us denote:
$$
a_0=\lim_{\substack{\vs\in\Ein\\|\vs|\to\infty}}\vs a_2^{[0]}(\vs).
$$
From the definition~\eqref{defa2} of $a_2$ and Theorem~\ref{thminner} (which gives some properties of the functions
$\psiin^\uns$ and $\psiin^\sta$) one obtains that, for all $(\vs,\theta) \in \Ein\times\Tout$,
\begin{equation*}
a_2(\vs,\theta)=b\vs^2\left[\Ginner^\uns(\Minner(0,0))(\vs,\theta)+\Ginner^\sta(\Minner(0,0))(\vs,\theta)\right]+\vs^2\hat H(0,0)+\mathcal{O}(\vs^{-2}),
\end{equation*}
where $\Ginner^{\uns,\sta}$ are defined in~\eqref{defGuns} and $\Minner$ in~\eqref{defopMinner}. From this expression,
one can see that $a_0$ is well defined. From definitions \eqref{notationFGHhat} of $\hat F$ and $\hat H$ one checks that, for some constant $K$:
$$
|a_2^{[0]}(\vs)-\vs^{-1}a_0|\leq \frac{K}{|\vs|^2},\qquad (\vs,\theta) \in \Ein\times\Tout.
$$

Since $\bar a_2(\vs,\theta) = a_2(\vs,\theta) - \coef \vs^{-1}L_0$, if we check that $a_0=\coef L_0$ the bound for $\bar a_2^{[0]}$
will hold true.
We state now the definition of $L_0$ introduced in~\cite{BCS16a}:
$$
L_0=\lim_{\vu\to i\frac{\pi}{2\coef}} \lim_{\delta\to0}\delta^{-1}l_2^{[0]}(\vu)\tanh^{-1}(\coef\vu),
$$
being $l_2^{0]}$ the mean of
\begin{equation}\label{defl2}
l_2(\vu,\theta)=-\frac{b}{\coef(1-\hetz^2(\vu))}(r_1^\uns+r_1^\sta)-\frac{\delta^p}{2\coef(1-\hetz^2(\vu))}\int_{-1}^1H(r_\lambda)d\lambda,
\end{equation}
where $H$ is defined in~\eqref{notationFGHbis}.

Let us consider $u(s)=d^{-1}\textrm{arctanh}(1/\delta s)$, as in~\eqref{change-s}.
From the definitions~\eqref{defl2} of $l_2(\vu,\theta)$ and~\eqref{defa2} of $a_2(\vs,\theta)$, recalling that:
$$
\psi^{\uns,\sta}(\vs,\theta)=\delta^2r_1^{\uns,\sta}(\vu(\vs),\theta),
$$
the fact that $\psi^{\uns,\sta}=\psiin^{\uns,\sta}+\psi_1^{\uns,\sta}$, using the bounds provided in Theorems~\ref{thmoutloc-innervariables}
and~\ref{thmmatching} for $\psi^{\uns,\sta}$ and $\psi_1^{\uns,\sta}$ respectively, and using formula~\eqref{notationFGHhat}
(which relates $H$ and $\hat H$) one can see, after some straightforward but tedious computations, that for $\vs\in\Dmchin{\uns}\cap\Dmchin{\sta}$:
$$
l_2(\vu(\vs),\theta)=\coef^{-1}a_2(\vs,\theta)+\mathcal{O}(\delta^{1-\gamma}).
$$
In particular, taking $|\vs|\leq K\delta^{(\gamma-1)/2}$, $\vs\in\Dmchin{\uns}\cap\Dmchin{\sta}$, one has:
$$
l_2(\vu(\vs),\theta)=\coef^{-1}a_2(\vs,\theta)+\mathcal{O}(\vs^{-2}),
$$
so that:
$$
l_2^{[0]}(\vu(\vs))=\coef^{-1}a_2^{[0]}(\vs)+\mathcal{O}(\vs^{-2}).
$$
Since
$|a_2^{[0]}(\vs)-\vs^{-1}a_0|\leq K|\vs|^{-2}$,
this yields:
$$
l_2^{[0]}(\vu(\vs))=\coef^{-1}a_0\vs^{-1}+\mathcal{O}(\vs^{-2}).
$$
Taking $\vs=1/(\delta\tanh(\coef\vu))$ (with $|\vu-i\pi/(2\coef)|\leq \delta^{(1+\gamma)/2}$ so that $|\vs|\leq K\delta^{(\gamma-1)/2}$) we obtain:
$$l_2^{[0]}(\vu)=\delta\coef^{-1}a_0\tanh(\coef\vu)+\mathcal{O}(\delta^2\tanh^2(\coef\vu)).$$
Thus:
$$L_0=\lim_{\vu\to i\frac{\pi}{2\coef}}\lim_{\delta\to0}\delta^{-1}l_2^{[0]}(\vu)\tanh^{-1}(\coef\vu)=\coef^{-1}a_0,$$
and the claim is proved.
\end{proof}

Next lemma deals with the operator $\Gdiffin$ and its Fourier coefficients.

\begin{lemma}\label{lempropsGdiffin}
Let $n\geq1$ and $\phi\in\Bsin{\uns}{n}$. There exists a constant $M$ such that:
\begin{enumerate}
\item If $l\neq0$, $\|\Gdiffin^{[l]}(\phi)\|_{n}\leq \displaystyle\frac{M}{|l|}\|\phi^{[l]}\|_{n}$. Moreover
$\|\partial_\theta\Gdiffin(\phi)\|_{n,\ost},\|\partial_\vs\Gdiffin(\phi)\|_{n,\ost}\leq M\|\phi\|_{n}.$
\item If $n>1$, then $\|\Gdiffin^{[0]}(\phi)\|_{n-1}\leq M\|\phi^{[0]}\|_{n}$ and $\|\Gdiffin(\phi)\|_{n-1,\ost}\leq M\|\phi\|_{n,\ost}$. In addition,
 $\llfloor\Gdiffin(\phi)\rrfloor_{n-1,\ost}\leq M\|\phi\|_{n,\ost}$.
\end{enumerate}
\end{lemma}

\subsubsection{Existence and properties of $\varphi$}\label{subsec:statementDiffin}

Now we are going to prove the statements of Theorem~\ref{thmdiffinner} related to the existence and properties of $\varphi$.
Concretely we prove:
\begin{proposition}\label{propk}
Equation~\eqref{PDE-phiP1short} has a solution $\varphi$ which is $2\pi$-periodic in $\theta$ and
\begin{equation}\label{boundphi}
\|\varphi\|_{1,\ost}\leq M,\qquad \|\partial_\vs \varphi\|_{1,\ost}\leq M,\qquad \|\partial_\theta \varphi\|_{1,\ost}\leq M.
\end{equation}
In addition, $(\xi_\inn(\vs,\theta),\theta)$, is injective in $\Ein\times\Tout$, with $\xi_\inn$ defined by~\eqref{formaxinn}.
\end{proposition}
\begin{proof}
By means of $\Gdiffin$, we rewrite the equation for $\varphi$ in~\eqref{PDE-phiP1short} as a fixed point equation:
$$
\varphi=\Gdiffin\circ\oprhsphi(\varphi)=:\tilde\oprhsphi(\varphi).
$$
We will prove that $\tilde\oprhsphi$ is a contraction in a certain ball. We first claim that
\begin{equation}\label{indepB}
\|\tilde\oprhsphi(0)\|_{1,\ost}\leq M,\qquad \|\partial_\theta\tilde\oprhsphi(0)\|_{1,\ost}\leq M,\qquad\|\partial_\vs\tilde\oprhsphi(0)\|_{1,\ost}\leq M,
\end{equation}
and hence $\llfloor\tilde\oprhsphi(0)\rrfloor_{0,\ost}\leq M$.
\begin{remark}
Note that, although we can bound $\tilde\oprhsphi(0)$ using the norm $\|.\|_{1,\ost}$, we have to take the norm $\llfloor.\rrfloor_{0,\ost}$ since the bounds of the derivatives are with the norm $\|.\|_{1,\ost}$ too.
\end{remark}
Indeed, by Lemma~\ref{lema123} it is straightforward to see that
$\|\oprhsphi(0)\|_{1,\ost}\leq K$, which in particular, implies that, for all $l\in\mathbb{Z}$,
$\|\oprhsphi^{[l]}(0)\|_{1}\leq K$. Therefore, using Lemma~\ref{lempropsGdiffin}.
\begin{equation}\label{boundGoprhsphil}
\|\Gdiffin^{[l]}(\oprhsphi(0))\|_{1}\leq \frac{K}{|l|}\|\oprhsphi^{[l]}(0)\|_{1}\leq \frac{K}{|l|},\qquad l\neq 0
\end{equation}
However, again using Lemma~\ref{lema123}, for the zeroth Fourier coefficient, one has that
$\|\oprhsphi^{[0]}(0)\|_{2}\leq K$
which,  using item 2 of Lemma~\ref{lempropsGdiffin}, yields
\begin{equation}\label{boundGoprhsphi0}
\|\Gdiffin^{[0]}(\oprhsphi(0))\|_{1}\leq K\|\oprhsphi^{[0]}(0)\|_{2} \leq K.
\end{equation}
Since $\tilde\oprhsphi=\Gdiffin\circ\oprhsphi$, bounds~\eqref{boundGoprhsphil} and~\eqref{boundGoprhsphi0} imply that
$\|\tilde\oprhsphi(0)\|_{1,\ost}\leq K$.

Finally using item 4 of Lemma \ref{lempropsGdiffin} with bounds \eqref{boundGoprhsphil} we obtain that:
$$\|\partial_\theta\tilde\oprhsphi(0)\|_{1,\ost}\leq K,\qquad\|\partial_\vs\tilde\oprhsphi(0)\|_{1,\ost}\leq K.$$

Using the same tools, one can check that, if $\phi_1,\phi_2\in\Bsdiffinfloor{0}$ are such that $\llfloor\phi_i\rrfloor_{0,\ost}\leq C$ for some constant $C$, then there exists a constant $M$ such that:
\begin{equation}\label{prop-lipconst-B}
\llfloor\tilde\oprhsphi(\phi_1)-\tilde\oprhsphi(\phi_2)\rrfloor_{1,\ost}\leq M\llfloor\phi_1-\phi_2\rrfloor_{0,\ost}.
\end{equation}
In particular $\llfloor\tilde\oprhsphi(\phi_1)-\tilde\oprhsphi(\phi_2)\rrfloor_{0,\ost}\leq M\distin^{-1}\llfloor\phi_1-\phi_2\rrfloor_{0,\ost}$.

Bound~\eqref{prop-lipconst-B} and~\eqref{indepB} imply that, taking $\distin$ is sufficiently large, the operator $\tilde \oprhsphi$ satisfies that
$\tilde\oprhsphi:B(2\llfloor\tilde\oprhsphi(0)\rrfloor_{0,\ost})\to B(2\llfloor\tilde\oprhsphi(0)\rrfloor_{0,\ost})$, and it has a unique fixed point:
$$\varphi\in B(2\llfloor\tilde\oprhsphi(0)\rrfloor_{0,\ost})\subset\Bsdiffinfloor{0}.$$
By construction, $\varphi$ satisfies equation \eqref{PDE-phi} and, as we pointed out in Section~\ref{subsec:diffinner-pre}, $\xi_\inn$ defined
as~\eqref{formaxinn} satisfies equation \eqref{PDE-k}.

By the definition of the norm $\llfloor.\rrfloor_{n,\ost}$ and since $\llfloor\varphi\rrfloor_{0,\ost}\leq2\llfloor\tilde\oprhsphi(0)\rrfloor_{0,\ost}\leq K$
by bounds~\eqref{indepB}, one obtains the corresponding bounds to $\partial_{\vs}$ and $\partial_{\theta}$ in~\eqref{boundphi}. We point out that this does
not imply \eqref{boundphi} directly, but it can be used to prove this bound \textit{a posteriori} with the following argument.
Indeed, since $\varphi$ is the unique fixed point of $\tilde\oprhsphi$, we can write:
\begin{equation}\label{rewritevarphi}
\varphi=\tilde\oprhsphi(\varphi)=\tilde\oprhsphi(0)+\tilde\oprhsphi(\varphi)-\tilde\oprhsphi(0).
\end{equation}
On the one hand, by~\eqref{indepB} we already know that $\|\tilde\oprhsphi(0)\|_{1,\ost}\leq K$.
On the other hand, $\llfloor\varphi\rrfloor_{0,\ost}\leq2\llfloor\tilde\oprhsphi(0)\rrfloor_{0,\ost}\leq K$ by~\eqref{prop-lipconst-B}, we have:
$$\|\tilde\oprhsphi(\varphi)-\tilde\oprhsphi(0)\|_{1,\ost}\leq\llfloor\tilde\oprhsphi(\varphi)-\tilde\oprhsphi(0)\rrfloor_{1,\ost}\leq K\llfloor\varphi\rrfloor_{0,\ost}\leq K.$$
Then, from \eqref{rewritevarphi} it is clear that $\llfloor\varphi\rrfloor_{1,\ost}\leq K$ and~\eqref{boundphi} is proven.

Now we shall prove that $(\xi_\inn(\vs,\theta),\theta)$, with $\xi_\inn$ as in~\eqref{formaxinn},
is injective in $\Ein\times\Tout$. We first note that if $\vs_1$,$\vs_2\in\Ein$, then $\vs_\lambda=\vs_1+\lambda(\vs_2-\vs_1)\in\Ein$. Thus, $|\vs_\lambda|\geq \distin$.
Assume now that $\xi_\inn(\vs_1,\theta)=\xi_\inn(\vs_2,\theta)$ for some $\vs_1,\vs_2\in\Ein$. By the mean value theorem and the bounds for $\varphi$ we have that
\begin{align*}
0 &=|\vs_1-\vs_2| \left|\coef^{-1}\alpha+\int_0^1\left[\frac{\coef^{-1}(c+\alpha L_0)}{\vs_\lambda}+\partial_\vs\varphi(\vs_\lambda,\theta)\right]d\lambda\right| \\
&\geq |\vs_1-\vs_2|\left (\coef^{-1}\alpha-\frac{K}{\distin}\right ).
\end{align*}
Taking $\distin$ is large enough one concludes that $\vs_1=\vs_2$.
\end{proof}

\subsubsection{Existence and properties of $\Pinsmall$}
As we did in the previous section with $\varphi$, we now prove the properties for $\Pinsmall$ established in Theorem~\ref{thmdiffinner}
which we summarize below:
\begin{proposition}\label{propP1}
Equation~\eqref{PDE-phiP1short} has a $2\pi-$periodic in $\theta$ solution, $\Pinsmall$, such that, for some $M>0$,
\begin{equation}\label{boundP1-inner}
\llfloor\Pinsmall\rrfloor_{1,\ost}\leq M.
\end{equation}
Therefore, $\Pin$ given by~\eqref{formaPin}, satisfies that $\Pin(\vs,\theta)\neq0$, for all $(\vs,\theta)\in\Ein\times\Tout$.
\end{proposition}
\begin{proof} Similarly as in the previous subsection, we first rewrite equation the equation for $\Pinsmall$ in~\eqref{PDE-phiP1short} as a fixed point equation:
$$
\Pinsmall=\Gdiffin\circ\oprhsP(\Pinsmall):=\tilde\oprhsP(\Pinsmall).
$$
Again, we prove that the operator $\tilde\oprhsP$ has a unique fixed point in a certain ball.
We have that, if $\phi_1,\phi_2\in\Bsdiffinfloor{1}$ such that $\llfloor\phi_i\rrfloor_{1,\ost}\leq C$ for some constant $C$, then
$$
\llfloor\tilde\oprhsP(0)\rrfloor_{1,\ost}\leq M,\qquad \llfloor\tilde\oprhsP(\phi_1)-\tilde\oprhsP(\phi_2)\rrfloor_{2,\ost}\leq M\llfloor\phi_1-\phi_2\rrfloor_{1,\ost}
$$
for some constant $M>0$. We can prove this properties using Lemma~\ref{lempropsGdiffin},
definition~\eqref{defoprhsP} of $\oprhsP$ and Lemma~\ref{lema123}.
As a consequence,
$$\llfloor\tilde\oprhsP(\phi_1)-\tilde\oprhsP(\phi_2)\rrfloor_{1,\ost}\leq \frac{K}{\distin} \llfloor\phi_1-\phi_2\rrfloor_{1,\ost}$$
Therefore, if $\distin$ sufficiently large, then $\tilde\oprhsP:B(2\llfloor\tilde\oprhsP(0)\rrfloor_{1,\ost})\to B(2\llfloor\tilde\oprhsP(0)\rrfloor_{1,\ost})$
and it has a unique fixed point:
$\Pinsmall\in B(2\llfloor\tilde\oprhsP(0)\rrfloor_{1,\ost})\subset \Bsdiffinfloor{1}$.
Since $\Pinsmall$ satisfies the fixed point equation $\Pinsmall=\tilde\oprhsP(\Pinsmall)$ it satisfies equation~\eqref{PDE-phiP1short},
and $\Pin(\vs,\theta)=\vs^{2/\coef}(1+\Pinsmall(\vs,\theta))$ satisfies equation~\eqref{PDE-difference}.

Bound \eqref{boundP1-inner} follows from the fact that,
$\llfloor \Pinsmall\rrfloor_{1,\ost}\leq2\llfloor\tilde\oprhsP(0)\rrfloor_{1,\ost}\leq K$.
\end{proof}

\subsubsection{End of the proof of Theorem~\ref{thmdiffinner}} 
We have that, for some periodic function $\kintilde$
\begin{equation}\label{formDeltapsiin}
\Delta\psiin(\vs,\theta)=\vs^{2/\coef}(1+\Pinsmall(\vs,\theta))\kintilde( \xi_\inn(\vs,\theta))
\end{equation}
with $\xi_\inn$ of the form~\eqref{formaxinn} and $\Pinsmall$ given in~\ref{propP1}.

It only remains to check that $\kintilde$ satisfies the properties stated in the theorem.
Since $\kintilde$ is $2\pi$-periodic, we can write it in its Fourier series:
$$\kintilde(\tau)=\sum_{l\in\mathbb{Z}}\Upsilon_\inn^{[l]}e^{il\tau}.$$
Now we note that by the definition of $\Delta\psiin=\psiin^{\uns}-\psiin^{\sta}$ and Theorem~\ref{thminner} we have for all $\vs\in\Ein$:
$$
|\Delta\psiin(\vs,\theta)|\leq |\psiin^\uns(\vs,\theta)|+|\psiin^\uns(\vs,\theta)|\leq \frac{K}{|\vs|^3}.
$$
In particular:
$$\lim_{\im\vs\to-\infty}\Delta\psiin(\vs,\theta)=0.$$
Since $\Delta\psiin(\vs,\theta)$ is defined for $\im\vs\to-\infty$, expression~\eqref{formDeltapsiin} of $\Delta\psiin$ implies that $\kintilde$ is
defined for $\im\tau\to-\infty$. Moreover:
$$\lim_{\im\tau\to-\infty}\kintilde(\tau)=0$$
and in particular $|\kintilde(\tau)|\leq M$.
This trivially implies that, on the one hand $\Upsilon_\inn^{[l]}=0$ if $l\geq 0$ and on the other hand 
$|\Upsilon_\inn^{[l]}\leq M$ if $l<0$.

The bounds for $\varphi$ and $\Pinsmall$ follow from the corresponding ones in Proposition~\ref{propk} and Proposition~\ref{propP1}, respectively. 
\section{The matching errors. Proof of Theorem \ref{thmmatching}}\label{sec:matching}
The main object of study of this section are the matching errors, defined in~\eqref{defpsi1us-intro} as:
\begin{equation}\label{defpsi1us}
\psi_1^\uns(\vs,\theta)=\psi^\uns(\vs,\theta)-\psiin^\uns(\vs,\theta),\qquad \psi_1^\sta(\vs,\theta)=\psi^\sta(\vs,\theta)-\psiin^\sta(\vs,\theta),
\end{equation}
where $\psi^{\uns,\sta}$ and $\psiin^{\uns,\sta}$ are given in Theorems~\ref{thmoutloc-innervariables} and~\ref{thminner} respectively.

\subsection{Banach spaces}
We present the Banach spaces we will work with, which are the same as in
Section~\ref{sec:inner}, but using the domains $\Dmchin{\uns,\sta}$ (see~\ref{Dmchindef}) instead.
For completeness we write them.
For $\phi:\Dmchin{\uns}\times\Tout\to\mathbb{C}$, with $\phi(\vs,\theta)=\sum_{l\in\mathbb{Z}}\phi^{[l]}(\vs)e^{il\theta}$, we define the norms:
\begin{align*}
&\|\phi\|_n^\uns:=\sup_{\vs\in\Dmchin{\uns}}|\vs^n\phi(\vs)|,\qquad
 \|\phi\|_{n,\ost}^\uns:=\sum_{l\in\mathbb{Z}}\|\phi^{[l]}\|_n^\uns e^{|l|\ost} \notag\\
&\llfloor\phi\rrfloor_{n,\ost}^\uns:=\|\phi\|_{n,\ost}^\uns+\|\partial_\vs\phi\|_{n+1,\ost}^\uns+\|\partial_\theta\phi\|_{n+1,\ost}^\uns, \notag
\end{align*}
and we endow the space of analytic functions with these norms:
\begin{align*}
 \Bsin{\uns}{n}&:=\{\phi:\Din{\uns}\times\Tout\to\mathbb{C}\,:\, \phi\textrm{ is analytic,}\,\|\phi\|_{n,\ost}^\uns<\infty\},\\
 \Bsinfloor{\uns}{n}&:=\{\phi:\Din{\uns}\times\Tout\to\mathbb{C}\,:\, \phi\textrm{ is analytic,}\,\llfloor\phi\rrfloor_{n,\ost}^\uns<\infty\}.
\end{align*}
For the stable case, we define analog norms $\|.\|_{n,\ost}^\sta$ and $\llfloor.\rrfloor_{n,\ost}^\sta$ and Banach spaces $\Bsin{\sta}{n}$ and $\Bsinfloor{\sta}{n}$, just replacing the domain $\Dmchin{\uns}$ by $\Dmchin{\sta}$.

We will prove the following proposition, which is equivalent to Theorem~\ref{thmmatching}.

\begin{proposition}\label{propmatching-inner2D}
Consider the functions $\psi_1^{\uns,\sta}(\vs,\theta)$ defined in~\eqref{defpsi1us}.
Then, $\psi_1^\uns\in\Bsinfloor{\uns}{2}$ and $\psi_1^\sta\in\Bsinfloor{\sta}{2}$. Moreover there exists a constant $M$ such that:
 $$\llfloor\psi_1^\uns\rrfloor_{2,\ost}^\uns\leq M\delta^{1-\gamma},\qquad\llfloor\psi_1^\sta\rrfloor_{2,\ost}^\sta\leq M\delta^{1-\gamma}.$$
\end{proposition}
The rest of this section is devoted to proving this result for the unstable case,
but the argument can be analogously done for the stable case.

\subsection{Decomposition of $\psi_1^\uns$}
Note that we already know the existence of $\psi^{\uns}_1=\psi^{\uns}-\psiin^{\uns}$ by Theorems~\ref{thmoutloc-innervariables}
and~\ref{thminner}. These theorems provide us with an \textit{a priori} bound of the matching error $\psi_1^\uns$. Indeed, it is clear that
\begin{equation}\label{prioriboundpsi1}
\psi_1^\uns \in  \Bsinfloor{\uns}{2},\qquad \llfloor\psi_1^{\uns}\rrfloor_{2,\ost}^\uns \leq K\dist^{-1}
\end{equation}
for some constant $K$. Our goal is to improve these bounds as stated in Proposition~\ref{propmatching-inner2D}.
This strategy to bound the matching error, was introduced in~\cite{BFGS12}.

Recall that $\psi^\uns(\vs,\theta)$ is defined for $\vs\in\DoutTinvars{\uns}$ (see~\eqref{defDoutTuns-innervars} for its definition) and
$\psiin^{\uns}(\vs,\theta)$ is defined for $\vs\in\Din{\uns}$ (see \eqref{defdinus-inner2D}).
Then, since
$$
\Dmchin{\uns}\subset\DoutTinvars{\uns}\subset\Din{\uns},
$$
one has that
$\psi_1^\uns$ is defined in $\Dmchin{\uns}$. We also recall that $\psi^{\uns}$ and  $\psiin^\uns$ satisfy:
$$\Linner(\psi^\uns)=\Minner(\psi^\uns,\delta),\qquad \Linner(\psiin^\uns)=\Minner(\psiin^\uns,0),$$
respectively, where $\Linner$ is the linear operator defined in~\eqref{defopLinner} and $\Minner$ is the operator defined in~\eqref{defopMinner}.
Defining the operator $\Mmatch^\uns$ as:
\begin{equation}\label{defMmatch}
\Mmatch^\uns(\psi_1^\uns)=\Minner(\psiin^\uns+\psi_1^\uns,\delta)-\Minner(\psiin^\uns,0),
\end{equation}
then $\psi_1^\uns$ satisfies:
\begin{equation}\label{PDE-matchshort}
 \Linner(\psi_1^\uns)=\Mmatch^\uns(\psi_1^\uns).
\end{equation}
For convenience, we avoid writing explicitly the dependence of $\Mmatch^\uns$ with respect to $\delta$.

We recall that $\psi_1^\uns(\vs,\theta)$ is $2\pi-$periodic in $\theta$.
Next lemma characterizes $\psi_1^\uns$ by means of the initial conditions of its Fourier coefficients in appropriate values of $\vs$.
\begin{lemma}\label{lemrewritepsi1uns}
Let $s_1$ and $s_2$ be the points defined in~\eqref{etsj-inner2D}:
$ \vs_j=\delta^{-1}(\vu_j-i\frac{\pi}{2\coef})$, for $j=1,2$ (see Figure~\ref{figDmatch} and~\eqref{defuj} for definition of $\vu_1,\vu_2$).

Then, the function $\psi_1^\uns=\psi^{\uns}-\psi^{\uns}_{\inn}$ defined in~\eqref{defpsi1us} is the unique function satisfying equation~\eqref{PDE-matchshort}
whose Fourier coefficients ${\psi_1^\uns}^{[l]}(s)$ satisfy:
\begin{equation}\label{icsFourierpsi1}
\begin{array}{rclr}{\psi_1^\uns}^{[l]}(\vs_1)&=&{\psi^\uns}^{[l]}(\vs_1)-{\psiin^\uns}^{[l]}(s_1)&\qquad\textrm{ if }l<0,\\
{\psi_1^\uns}^{[l]}(\vs_2)&=&{\psi^\uns}^{[l]}(\vs_2)-{\psiin^\uns}^{[l]}(\vs_2)&\qquad\textrm{ if }l\geq0.
\end{array}
\end{equation}
\end{lemma}

\begin{proof}
Since $\psi_1^\uns$ is $2\pi-$periodic in $\theta$, it is uniquely determined by its Fourier coefficients. Writing equation~\eqref{PDE-matchshort}
in terms of these Fourier coefficients, one easily obtains that each Fourier coefficient ${\psi_1^{\uns}}^{[l]}$ satisfies a given ODE.
Moreover, solutions of ODEs are uniquely determined by an initial condition at a given time. We choose this initial time to be $s=s_1$ for $l<0$
and $s=s_2$ for $l\geq0$. Since by definition ${\psi_1^{\uns}}^{[l]}(\vs)={\psi^{\uns}}^{[l]}(\vs)-{\psiin^{\uns}}^{[l]}(\vs)$,
we obtain precisely~\eqref{icsFourierpsi1}.
\end{proof}

The values ${\psi_1^\uns}^{[l]}(s_1)$ and ${\psi_1^\uns}^{[l]}(s_2)$ will be bounded later on using Theorem~\ref{thmoutloc-innervariables} and Theorem~\ref{thminner}.
We will denote them by $C_l^\uns$:
\begin{equation}\label{defCluns}
 C_l^\uns:=\left\{\begin{array}{rcl}{\psi^\uns}^{[l]}(s_1)-{\psiin^\uns}^{[l]}(s_1)&\qquad&\textrm{ if }l<0,\\
 {\psi^\uns}^{[l]}(s_2)-{\psiin^\uns}^{[l]}(s_2)&\qquad&\textrm{ if }l\geq0.
 \end{array}\right.
\end{equation}

Recall that our ultimate goal is to find a sharp bound of $\psi_1^\uns(\vs,\theta)$, for $(\vs,\theta)\in\Dmchin{\uns}$. To that aim we write $\psi_1^\uns$ as a function that satisfies \eqref{PDE-matchshort} and \eqref{icsFourierpsi1} by means of a solution of the homogeneous equation $\Linner(\psi)=0$ with initial conditions \eqref{icsFourierpsi1} and a suitable solution of a fixed point equation. More precisely, we shall follow three steps:

\begin{enumerate}
\item First, we construct a function $\Phi^\uns$ that satisfies:
\begin{enumerate}
\item $\Linner\circ\Phi^{\uns}=0,$
\item ${\Phi^{\uns}}^{[l]}(\vs_i)=C_l^\uns$, where we take $i=1$ if $l<0$ and $i=2$ otherwise.
\end{enumerate}
This can be trivially done, defining $\Phi^\uns$ as the function:
\begin{equation}\label{defPhiuns}
 \Phi^\uns(\vs,\theta)=\sum_{k\in\mathbb{Z}}{\Phi^{\uns}}^{[l]}(\vs)e^{il\theta},
\end{equation}
where:
\begin{equation}\label{defPhiuns-coeffs}
 \begin{array}{rclr}
  {\Phi^{\uns}}^{[l]}(\vs)&=&\displaystyle\frac{C_l^\uns}{\vs_1^{2/\coef}}\vs^{2/\coef}e^{\coef^{-1}\alpha(\vs-\vs_1)il}&\qquad\textrm{ if }l<0,\bigskip\\
  {\Phi^{\uns}}^{[l]}(\vs)&=&\displaystyle\frac{C_l^\uns}{\vs_2^{2/\coef}}\vs^{2/\coef}e^{\coef^{-1}\alpha(\vs-\vs_2)il}&\qquad\textrm{ if }l\geq0.
 \end{array}
\end{equation}

\item The second step consists in finding a right inverse $\Gmatch^\uns$ of the operator $\Linner$. We can define it via its Fourier coefficients ${\Gmatch^\uns}^{[l]}$. That is, given a function $\phi(\vs,\theta)$ we consider:
\begin{equation}\label{defGmatch}
 \Gmatch^\uns(\phi)(\vs,\theta)=\sum_{k\in\mathbb{Z}}{\Gmatch^{\uns}}^{[l]}(\phi)(\vs)e^{il\theta},
\end{equation}
and we choose ${\Gmatch^{\uns}}^{[l]}$ so that for all functions $\phi(\vs,\theta)$ the following holds:
\begin{enumerate}
\item[(c)] ${\Gmatch^\uns}^{[l]}(\phi)(\vs_1)=0,$ if $l<0$,
\item[(d)]  ${\Gmatch^\uns}^{[l]}(\phi)(\vs_2)=0,$ if $l\geq0$.
\end{enumerate}
One can easily see that if we define:
\begin{align*}
{\Gmatch^{\uns}}^{[l]}(\phi)(\vs)&=\coef^{-1}\vs^{\frac{2}{\coef}}\int_{\vs_1}^\vs\frac{e^{-\frac{il\alpha}{\coef}(w-\vs)}}{w^{\frac{2}{\coef}}}\phi^{[l]}(w)dw\qquad \textrm{ if }l<0,\notag\\
{\Gmatch^{\uns}}^{[l]}(\phi)(\vs)&=\coef^{-1}\vs^{\frac{2}{\coef}}\int_{\vs_2}^\vs\frac{e^{-\frac{il\alpha}{\coef}(w-\vs)}}{w^{\frac{2}{\coef}}}\phi^{[l]}(w)dw\qquad \textrm{ if }l\geq0,\notag
\end{align*}
then $\Gmatch^\uns$ defined as in \eqref{defGmatch} satisfies conditions (c) and (d).

\item Now we point out that items (a)--(d) above imply that the function $\phi$ defined implicitly by:
$$\phi=\Phi^{\uns}+\Gmatch^\uns(\Mmatch^\uns(\phi)),$$
satisfies \eqref{PDE-matchshort} and \eqref{icsFourierpsi1}. Since by Lemma \ref{lemrewritepsi1uns} $\psi_1^\uns$ is the only function satisfying \eqref{PDE-matchshort} and \eqref{icsFourierpsi1}, we can write:
\begin{equation}\label{rewritepsi1-intro}
\psi_1^\uns=\Phi^{\uns}+\Gmatch^\uns(\Mmatch^\uns(\psi_1^\uns)).
\end{equation}
We define the operator:
$$
\Delta {\Mmatch^\uns}(\phi):=\Gmatch^\uns(\Mmatch^\uns(\phi))-\Gmatch^\uns(\Mmatch^\uns(0)).
$$
Then we can rewrite \eqref{rewritepsi1-intro} as:
\begin{equation}\label{rewritepsi1-idG-intro}
(\textrm{Id}-\Delta {\Mmatch^\uns})(\psi_1^\uns)=\Phi^{\uns}+\Gmatch^\uns(\Mmatch^\uns(0)).
\end{equation}
Thus, we just need to see that the operator $\Delta {\Mmatch^\uns}$ has ``small'' norm in $\Bsinfloor{\uns}{2}$.
We point out that, unlike $\psi_1^\uns$, we have an explicit formula for functions $\Phi^{\uns}$ and $\Gmatch^\uns(\Mmatch^\uns(0))$, so that these functions can be bounded easily. Then \eqref{rewritepsi1-idG-intro} will allow us to bound $\psi_1^\uns$ using bounds of the functions $\Phi^{\uns}$ and
$\Gmatch^\uns(\Mmatch^\uns(0))$.
\end{enumerate}

We shall proceed as follows. First we state several technical results about $\Gmatch^{\uns}$ and $\hat{F}, \hat{G}$ and $\hat{H}$ whose proofs
can be encountered in~\cite{CastejonPhDThesis}. Next we summarize the main properties of the operator $\Gmatch^\uns$. After that, we find bounds of
$\Phi^\uns$  and  $\Gmatch^\uns(\Mmatch^\uns(0))$. This is done in Section~\ref{subsec:phi-gmatchmmatch}.
Finally, in Section~\ref{subsec:mmatchtilde} we study the operator $\Delta \Mmatch^\uns$ to see that $\textrm{Id}-\Delta \Mmatch^\uns$ is invertible,
which yields the proof of Proposition~\ref{propmatching-inner2D}.

\subsection{Preliminary properties of $\Gmatch^{\uns}$, $\hat{F}, \hat{G}$ and $\hat{H}$}\label{subsec:Banach-match}

Even when we are not going to prove the following properties, let us just to point out that, to prove these properties we have to take into account
that for $\vs\in\Dmchin{\uns}$ one has:
\begin{equation}\label{boundmch}
K_1\dist\leq|\vs|\leq K_2\delta^{\gamma-1}\qquad \textrm{and}\qquad \delta<|\vs|^{-1}.
\end{equation}

\begin{enumerate}
\item
\textit{Banach spaces}. The same properties given in item~\ref{itemproposnorm-inner} in Section~\ref{subsec:Bspaces} hold
true in this case.
\item \label{itempropsGmatch}\textit{The operator $\Gmatch^{\uns}$}. Again the same properties in item~\ref{itempropsGinner} are valid in this case.
\item \textit{The nonlinear terms, $\hat{F}, \hat{G}$ and $\hat{H}$}. The definition of these functions is given in~\eqref{notationFGHhat}.
Let $C$ be any constant. Then:
\begin{enumerate}
\item If $\phi\in\Bsin{\uns}{3}$ with $\|\phi\|_{3,\ost}^\uns\leq C$, then there exists $M>0$ such that
\begin{align}
&\|\hat F(\phi,\delta)\|_{4,\ost}^\uns, \;\|\hat G(\phi,\delta)\|_{2,\ost}^\uns,\;\|\hat H(\phi,\delta)\|_{3,\ost}^\uns\leq M \notag
\\
&\|D_\delta\hat F(\phi,\delta)\|_{3,\ost}^\uns,\; \|D_\delta\hat G(\phi,\delta)\|_{1,\ost}^\uns,\;
\|D_\delta\hat H(\phi,\delta)\|_{2,\ost}^\uns\leq M.\label{itemdifFGHhatsdelta}
\end{align}
\item
If $\phi\in\Bsin{\uns}{2}$ with $\|\phi\|_{2,\ost}^\uns\leq C/\distin$ and $\distin$ is sufficiently large, there exists $M>0$ and:
$$
\|D_\phi\hat F(\phi,\delta)\|_{2,\ost}^\uns,\; \|D_\phi\hat G(\phi,\delta)\|_{0,\ost}^\uns,\; \|D_\phi\hat H(\phi,\delta)\|_{1,\ost}^\uns\leq M.
$$
Note that, in the definition \eqref{notationFGHhat} of $\hat F$, $\hat G$ and $\hat H$, the variable $\phi$ always appears inside
the function $\rho(\phi,\vs,\delta)$ defined in~\eqref{defrho}. The condition $\|\phi\|_{2,\ost}^\uns$ small assures that $\rho(\phi,\vs,\delta)\neq0$
which is needed to prove the actual item.
\item
If $\phi\in\Bsin{\uns}{3}$ is such that $\|\phi\|_{3,\ost}^\uns\leq C$, there exists $M>0$ such that:
\begin{align*}
\|\hat F(\phi,\delta)- \hat F(\phi,&\,0)\|_{3,\ost}^\uns\leq M\delta, \quad
\|\hat G(\phi,\delta)- \hat G(\phi,0)\|_{1,\ost}^\uns\leq M\delta,\\
&\|\hat H(\phi,\delta)- \hat H(\phi,0)\|_{2,\ost}^\uns\leq M\delta.
\end{align*}
\end{enumerate}
\end{enumerate}

\subsection{The functions $\Phi^{\uns}$ and $\Gmatch^\uns(\Mmatch^\uns(0))$}\label{subsec:phi-gmatchmmatch}
Along this section we use the previous properties of $\Gmatch$, $\hat{F},\hat{G}$ and $\hat{H}$ to give suitable
properties of $\Phi^{\uns}$ and $\Gmatch^\uns(\Mmatch^\uns(0))$.
Looking at equality~\eqref{rewritepsi1-idG-intro} this is mandatory to obtain sharp bounds of the matching error $\psi_1^{\uns}$.
\begin{lemma}\label{lemboundPhi}
The function $\Phi^\uns$ defined in~\eqref{defPhiuns}--\eqref{defPhiuns-coeffs} satisfies
$\Phi^\uns\in\Bsinfloor{\uns}{2}$. Moreover, there exists a constant $M$ such that:
$\llfloor\Phi^\uns\rrfloor_{2,\ost}^\uns\leq M\delta^{1-\gamma}$.
\end{lemma}
\begin{proof}
Let us recall the definition~\eqref{defPhiuns-coeffs} of the Fourier coefficients ${\Phi^\uns}^{[l]}$:
\begin{equation}\label{defPhiunsl-proof}
{\Phi^\uns}^{[l]}(\vs)=\frac{C_l^\uns}{\vs_j^{2/\coef}}\vs^{2/\coef}e^{\coef^{-1}\alpha(\vs-\vs_j)il},
\end{equation}
where $j=1$ if $l<0$ and $j=2$ if $l\geq0$. From the definition~\eqref{defCluns} of $C_l^\uns$ and using
Theorems~\ref{thmoutloc-innervariables} and~\ref{thminner}, it is clear that:
\begin{equation}\label{boundCluns}
|C_l^\uns|\leq \left(\|{\psi^\uns}^{[l]}\|_{3}^\uns+\|{\psiin^\uns}^{[l]}\|_{3}^\uns\right)|\vs_j|^{-3}.
\end{equation}
Moreover, since $\im(s-\vs_j)l>0$, we have $|e^{\coef^{-1}\alpha(\vs-\vs_j)il}|<1$. Then:
$$\left|{\Phi^\uns}^{[l]}(\vs)\vs^2\right|\leq |\vs_j|^{-3-2/\coef}|\vs|^{2+2/\coef}\left(\|{\psi^\uns}^{[l]}\|_{3}^\uns+\|{\psiin^\uns}^{[l]}\|_{3}^\uns\right).$$
As we pointed out in~\eqref{fitasj-inner2D} and~\eqref{upperlowerboundss},
$|s_j|\geq K_1 \delta^{\gamma-1}$ and $|s|\leq  K_2\delta^{\gamma-1}$ for all $\vs\in\Dmchin{\uns}$, thus
$$\left|{\Phi^\uns}^{[l]}(\vs)\vs^2\right|\leq K|\vs_j|^{-1}\left(\|{\psi^\uns}^{[l]}\|_{3}^\uns+\|{\psiin^\uns}^{[l]}\|_{3}^\uns\right).$$
We use now the definition of the norm $\|.\|_{3,\ost}^{\uns}$, that by
Theorems~\ref{thmoutloc-innervariables} and~\ref{thminner}, $\|\psi^\uns\|_{3,\ost}^\uns+ \|\psiin^\uns\|_{3,\ost}^\uns\leq K$ and that
$|s_j|\geq K_1 \delta^{\gamma-1}$, to conclude that
\begin{equation}\label{normPhiuns}
\|\Phi^\uns\|_{2,\ost}^\uns\leq K\delta^{1-\gamma}.
\end{equation}

Now we proceed to bound $\|\partial_\theta\Phi^\uns\|_{3,\ost}$. We note that, if $\phi(\vs,\theta)$ is a $2\pi$-periodic function with Fourier coefficients
$\phi^{[l]}(\vs)$ the Fourier coefficients of $\partial_{\theta}\phi(\vs,\theta)$ are $il \phi^{[l]}(\vs)$. Using this and definition~\eqref{defCluns} of $C_l^\uns$
we have that
\begin{equation}\label{boundilCluns}
|ilC_l^\uns|\leq \left(\|\left(\partial_\theta\psi^\uns\right)^{[l]}\|_{4}^\uns+\|\left(\partial_\theta\psiin^\uns\right)^{[l]}\|_{4}^\uns\right)|\vs_j|^{-4},
\end{equation}
Then, reasoning analogously as in the previous case we reach that:
\begin{align}\label{normpartialthetaPhiuns}
\|\partial_\theta\Phi^\uns\|_{3,\ost}^\uns\leq& K\delta^{1-\gamma}\left(\|\partial_\theta\psi^\uns\|_{4,\ost}^\uns+
\|\partial_\theta\psiin^\uns\|_{4,\ost}^\uns\right)\leq  K\delta^{1-\gamma}.
\end{align}

Finally we bound $\|\partial_\vs\Phi^\uns\|_{3,\ost}$. Differentiating the Fourier coefficients of ${\Phi^{\uns}}^{[l]}$ defined in \eqref{defPhiunsl-proof} with respect to $\vs$ we obtain:
$$\frac{d}{d\vs}{\Phi^{\uns}}^{[l]}(\vs)=\frac{2C_l^\uns}{\coef\vs_j^{2/\coef}}\vs^{2/\coef-1}e^{\coef^{-1}\alpha(\vs-\vs_j)il}+\frac{C_l^\uns}{\vs_j^{2/\coef}}\vs^{2/\coef}\coef^{-1}\alpha ile^{\coef^{-1}\alpha(\vs-\vs_j)il}.$$
Using bounds~\eqref{boundCluns} and~\eqref{boundilCluns} of $C_l^{\uns}$ and $ilC_l^\uns$ respectively, we obtain:
\begin{equation}\label{normpartialsPhiuns}
\|\partial_\vs\Phi^\uns\|_{3,\ost}^\uns\leq K\delta^{1-\gamma}\left(\|\partial_{\vs}\psi\|^\uns_{4,\ost}+
\|\psiin^\uns\|_{4,\ost}^\uns\right)\leq K\delta^{1-\gamma}.
\end{equation}

Bounds \eqref{normPhiuns}, \eqref{normpartialthetaPhiuns} and \eqref{normpartialsPhiuns} yield directly the claim of the lemma.
\end{proof}

\begin{lemma}\label{lemboundGmatchMmatch0}
Let $\param=\mathcal{O}(\delta)$. Then the function $\Gmatch^\uns(\Mmatch^\uns(0))\in\Bsinfloor{\uns}{2},$ where $\Mmatch^\uns$ is defined in \eqref{defMmatch} and $\Gmatch^\uns$ is defined in \eqref{defGmatch}. Moreover, there exists a constant $M$ such that:
 $$\llfloor\Gmatch^\uns(\Mmatch^\uns(0))\rrfloor_{2,\ost}^\uns\leq M\delta.$$
\end{lemma}
\begin{proof}
By item~\ref{itempropsGmatch} in Section~\ref{subsec:Banach-match} it is enough to prove that:
\begin{equation}\label{boundnormGmatch0}
\|\Mmatch^\uns(0)\|_{3,\ost}^\uns\leq K\delta.
\end{equation}
Recall that
$\Mmatch^\uns(0)=\Minner(\psiin^\uns,\delta)-\Minner(\psiin^\uns,0)$.
Thus, from definition \eqref{defopMinner} of $\Minner$ we obtain:
\begin{align*}
\Mmatch^\uns(0)=&\coef\delta^2\vs^2\partial_\vs\psiin^\uns+2\param\frac{\coef+1}{2b}\left(\left(\delta^3-\frac{\delta}{\vs^2}\right)+\delta\psiin^\uns\right)\\
&+\hat F(\psiin^\uns,\delta)-\hat F(\psiin^\uns,0)+\frac{\coef+1}{b}\vs^{-1}\left(\hat H(\psiin^\uns,\delta)-\hat H(\psiin^\uns,0)\right)\\
&-\left(\hat G(\psiin^\uns,\delta)-\hat G(\psiin^\uns,0)\right)\partial_\theta\psiin^\uns+\vs^2\left(\hat H(\psiin^\uns,\delta)-\hat H(\psiin^\uns,0)\right)\partial_\vs\psiin^\uns.
\end{align*}
Now, using that for $\vs\in\Dmchin{\uns}$ (see \eqref{boundmch}):
$K_1\dist\leq|\vs|\leq K_2\delta^{\gamma-1}$, and the fact that $\llfloor\psiin^\uns\rrfloor_{3,\ost}^\uns\leq K$, it is easy to check that:
$$\left\|\delta^2\vs^2\partial_\vs\psiin^\uns\right\|_{3,\ost}^\uns\leq K\delta^{1+\gamma},$$
and since $\param=\mathcal{O}(\delta)$:
$$\left\|2\param\frac{\coef+1}{2b}\left(\left(\delta^3-\frac{\delta}{\vs^2}\right)+\delta\psiin^\uns\right)\right\|_{3,\ost}^\uns\leq K\delta^{1+\gamma}.$$
These facts and bound~\eqref{itemdifFGHhatsdelta} of $D_{\delta}\hat{F}, D_{\delta}\hat{G}$ and $D_{\delta}\hat{H}$, jointly with the properties of the norm
$\|.\|_{n,\ost}^\uns$, yield directly bound~\eqref{boundnormGmatch0}.
\end{proof}

\subsection{The operator $\Delta \Mmatch^\uns$}\label{subsec:mmatchtilde}
In this subsection we are going to prove that the operator $\text{Id}-\Delta \Mmatch^\uns$ is invertible in the
Banach space $\Bsinfloor{\uns}{2}$. After that we will end the proof of Proposition~\ref{propmatching-inner2D}.

\begin{lemma}\label{lemnormMmatchtilde}
Let $\param=\mathcal{O}(\delta)$ and $C>0$. For any $\phi\in\Bsinfloor{\uns}{2}$ satisfying $\llfloor\phi\rrfloor_{2,\ost}^\uns\leq C/\dist$,
we have that $\Delta \Mmatch^\uns(\phi)\in\Bsinfloor{\uns}{2}$. Moreover there exists $M>0$ such that:
 $$\llfloor\Delta \Mmatch^\uns(\phi)\rrfloor_{2,\ost}^\uns\leq \frac{M}{\dist}\llfloor\phi\rrfloor_{2,\ost}^\uns.$$
\end{lemma}
\begin{proof}
 Recall that
$$
\Delta \Mmatch^\uns(\phi)=\Gmatch^\uns(\Mmatch^\uns(\phi))-\Gmatch^{\uns}(\Mmatch^\uns(0))=\Gmatch^\uns(\Mmatch^\uns(\phi)-\Mmatch^\uns(0)).
$$
Thus, by item~\ref{itempropsGmatch} in Section~\ref{subsec:Banach-match} it is sufficient to prove that:
\begin{equation}\label{boundnormMmatch0}
\|\Mmatch^\uns(\phi)-\Mmatch^\uns(0)\|_{3,\ost}^\uns\leq \frac{K}{\dist}\llfloor\phi\rrfloor_{2,\ost}^\uns.
\end{equation}
By definition \eqref{defMmatch} of $\Mmatch^\uns$, one has:
$$
\Mmatch^\uns(\phi)-\Mmatch^\uns(0)=\Minner(\psiin^\uns+\phi,\delta)-\Minner(\psiin^\uns,\delta).
$$
Using definition \eqref{defopMinner} of $\Minner$, one obtains:
\begin{align*}\Mmatch^\uns (\phi)-\Mmatch^\uns(0)=&c\vs^{-1}\partial_\theta\phi+\coef\delta^2\vs^2\partial_\vs\phi+2\param\delta\phi+\hat F(\psiin^\uns+\phi,\delta)-\hat F(\psiin^\uns,\vs,\delta)\\
&+\frac{\coef+1}{b}\vs^{-1}\left(\hat H(\psiin^\uns+\phi,\vs,\delta)-\hat H(\psiin^\uns,\vs,\delta)\right)\\
&-\left(\hat G(\psiin^\uns+\phi,\delta)-\hat G(\psiin^\uns,\delta)\right)\partial_\theta(\psiin^\uns+\phi)\\
&-\hat G(\psiin^\uns,\delta)\partial_\theta\phi +\vs^2\left(2b\psiin^\uns+\hat H(\psiin^\uns,\delta)\right)\partial_\vs\phi\\
&+\vs^2\left(2b\phi+\hat H(\psiin^\uns+\phi,\delta)-\hat H(\psiin^\uns,\delta)\right)\partial_\vs(\psiin^\uns+\phi).
\end{align*}
Let us denote:
$$\mathcal{R}(\phi):=\Mmatch^\uns(\phi)-\Mmatch^\uns(0)-\coef\delta^2\vs^2\partial_\vs\phi-2b\vs^2\phi\partial_\vs\phi.$$
Using properties in Section~\ref{subsec:Banach-match}, that $\llfloor\psiin^\uns\rrfloor_{3,\ost}^\uns\leq K$ and that
$\delta\leq K|\vs|^{-1}$ for $\vs\in\Dmchin{\uns}$, one obtains that
$\|\mathcal{R}(\phi)\|_{4,\ost}^\uns\leq K\llfloor\phi\rrfloor_{2,\ost}^\uns.$
Then,
\begin{equation}\label{boundN}
\|\mathcal{R}(\phi)\|_{3,\ost}^\uns\leq \frac{K}{\dist}\|\mathcal{R}(\phi)\|_{4,\ost}^\uns\leq \frac{K}{\dist}\llfloor\phi\rrfloor_{2,\ost}^\uns.
\end{equation}

Now we just need to note that since $|\vs|\leq K\delta^{\gamma-1}$:
\begin{equation}\label{boundsmallterm}
\|\coef\delta^2\vs^2\partial_\vs\phi\|_{3,\ost}^\uns\leq K\delta^{2\gamma}\llfloor\phi\rrfloor_{2,\ost}^\uns.
\end{equation}
Finally, since by assumption $\llfloor\phi\rrfloor_{2,\ost}^\uns\leq C/\dist$, then:
\begin{equation}\label{boundquadraticterm}
\|2b\vs^2\phi\partial_\vs\phi\|_{3,\ost}^\uns\leq K\left(\llfloor\phi\rrfloor_{2,\ost}^\uns\right)^2\leq \frac{K}{\dist}\llfloor\phi\rrfloor_{2,\ost}^\uns.
\end{equation}

Bounds~\eqref{boundN}, \eqref{boundsmallterm} and~\eqref{boundquadraticterm} yield~\eqref{boundnormMmatch0}, and so the proof is finished.
\end{proof}

\begin{proof}[End of the proof of Proposition~\ref{propmatching-inner2D}]
By Lemmas~\ref{lemboundPhi} and~\ref{lemboundGmatchMmatch0} we have that:
$$
\llfloor\Phi^{\uns}+\Gmatch^\uns(\Mmatch^\uns(0))\rrfloor_{2,\ost}^\uns\leq \|\Phi^{\uns}\|_{2,\ost}^\uns+\|\Gmatch^\uns(\Mmatch^\uns(0))\|_{2,\ost}^\uns
\leq K\big(\delta^{1-\gamma}+\delta\big) \leq K\delta^{1-\gamma}.
$$
In addition, using equation~ \eqref{rewritepsi1-idG-intro}, we have that
$$
\llfloor \psi_1^\uns \rrfloor_{2,\ost}^{\uns} - \llfloor \Delta \Mmatch^\uns (\psi_1^\uns) \rrfloor_{2,\ost}^{\uns} \leq K\delta^{1-\gamma}.
$$
Then, since, $\llfloor \psi_1^\uns \rrfloor_{2,\ost}^{\uns} \leq K\dist^{-1}$,  as we pointed out in~\eqref{prioriboundpsi1}, one obtains
by applying Lemma~\ref{lemnormMmatchtilde}:
$$
\llfloor \psi_1^\uns \rrfloor_{2,\ost}^{\uns} (1-K\dist^{-1}) \leq K\delta^{1-\gamma}
$$
so that Proposition \ref{propmatching-inner2D} is proved.
\end{proof}
\begin{remark}
Notice that $\Delta \Mmatch^\uns (\psi_1^{\uns})$ can be expressed as $\mathcal{B}(s) \psi_1^{\uns}$, that is a linear operator. Since $\llfloor \psi_1^\uns \rrfloor_{2,\ost}^{\uns}\leq K\dist^{-1}$, $\mathcal{B}(s)$ has norm strictly less than $1$ so that the linear operator $\textrm{Id} - \mathcal{B}(s)$ is invertible and then we can express $\psi_1^{\uns}$ as
$\psi_1^\uns = \big (\textrm{Id} - \mathcal{B}(s)\big )^{-1} \Phi^{\uns}+\Gmatch^\uns(\Mmatch^\uns(0))$.
\end{remark}
\cleardoublepage
\bibliography{references}
\bibliographystyle{alpha}

\end{document}